\documentclass[reqno,11pt,centertags, draft]{amsart}
\usepackage{hyperref}
\usepackage{amsmath,amsthm,amscd,amssymb,amsfonts,latexsym,upref,stmaryrd,
	enumerate,color,verbatim,mathrsfs, mathtools}
\usepackage{tikz}
\usepackage{cite}
\usetikzlibrary{trees}
\date{\today}

\mathtoolsset{showonlyrefs}

\newcommand*{\mailto}[1]{\href{mailto:#1}{\nolinkurl{#1}}}

\newcommand{\R}{{\bbR}}

\newcommand{\Z}{{\bbZ}}
\newcommand{\C}{{\bbC}}

\newcommand{\bbA}{{\mathbb{A}}}

\newcommand{\bbC}{{\mathbb{C}}}

\newcommand{\bbJ}{{\mathbb{J}}}
\newcommand{\bbH}{{\mathbb{H}}}

\newcommand{\bbN}{{\mathbb{N}}}
\newcommand{\bbP}{{\mathbb{P}}}
\newcommand{\bbQ}{{\mathbb{Q}}}
\newcommand{\bbR}{{\mathbb{R}}}

\newcommand{\bbZ}{{\mathbb{Z}}}

\newcommand{\cA}{{\mathcal A}}
\newcommand{\cB}{{\mathcal B}}

\newcommand{\cD}{{\mathcal D}}
\newcommand{\cE}{{\mathcal E}}

\newcommand{\cH}{{\mathcal H}}

\newcommand{\cK}{{\mathcal K}}
\newcommand{\cL}{{\mathcal L}}

\newcommand{\cN}{{\mathcal N}}
\newcommand{\cO}{{\mathcal O}}

\newcommand{\cR}{{\mathcal R}}
\newcommand{\cS}{{\mathcal S}}

\newcommand{\cU}{{\mathcal U}}
\newcommand{\cV}{{\mathcal V}}

\newcommand{\gD}{{\mathfrak{D}}}

\newcommand{\gh}{{\mathfrak{h}}}

\newcommand{\bfi}{{\bf i}}

\renewcommand{\b}{\beta}

\newcommand{\e}{\varepsilon}

\renewcommand{\l}{\lambda}

\DeclareMathOperator{\dist}{dist}
\DeclareMathOperator{\supp}{supp}

\DeclareMathOperator{\ran}{ran}
\DeclareMathOperator{\dom}{dom}
\DeclareMathOperator{\gen}{gen}

\DeclareMathOperator{\tr}{tr}

\newcommand{\SL}{\mathrm{SL}}

\renewcommand{\Im}{\text{\rm Im}}

\newcommand{\no}{\notag}
\newcommand{\lb}{\label}

\newcommand{\wti}{\widetilde}

\newcommand{\hatt}{\widehat}
\newcommand{\bi}{\bibitem}

\renewcommand{\ge}{\geqslant}
\renewcommand{\le}{\leqslant}

\let\geq\geqslant
\let\leq\leqslant

\definecolor{purple}{rgb}{.5,0,1}

\newcommand{\set}[1]{{\left\{ {#1} \right\}}}

\makeatletter
\def\theequation{\@arabic\c@equation}

\allowdisplaybreaks
\numberwithin{equation}{section}

\newtheorem{theorem}{Theorem}[section]
\newtheorem{proposition}[theorem]{Proposition}
\newtheorem{lemma}[theorem]{Lemma}

\newtheorem{hypothesis}[theorem]{Hypothesis}

\newtheorem{step}{Step}
\theoremstyle{remark}
\newtheorem{remark}[theorem]{Remark}

\begin{document}
	
	\numberwithin{equation}{section}
	\allowdisplaybreaks
	
	\title[Localization for Radial Tree Graphs]{Localization for Anderson Models on \\ Metric and Discrete Tree Graphs}

	\author[D.\ Damanik]{David Damanik}
	\address{ Department of Mathematics, Rice University, Houston, TX 77005, USA}
	\email{\mailto{damanik@rice.edu}}
	\thanks{D.D.\ was supported in part by NSF grant DMS--1700131. }
		
	\author[J. Fillman]{Jake Fillman}
	\address{ Department of Mathematics, Texas State University,  San Marcos,\ \   TX 78666, USA}
	\email{\mailto{fillman@txstate.edu}}
	\thanks{J.F.\ was supported in part by an AMS-Simons travel grant, 2016-2018}
	
	\author[S.\ Sukhtaiev]{Selim Sukhtaiev}
	\address{ Department of Mathematics, Rice University, Houston, TX 77005, USA}
	\email{\mailto{sukhtaiev@rice.edu}}
	\thanks {S.S.\ was supported in part by an AMS-Simons travel grant, 2017-2019}
	

	\date{\today}
	\keywords{Anderson localization, Laplace operator, tree graphs}

	\begin{abstract}
	We establish spectral and dynamical localization for several Anderson models on metric and discrete radial trees.  The localization results are obtained on compact intervals contained in the complement of discrete sets of exceptional energies. All results are proved under the minimal hypothesis on the type of disorder: the random variables generating the trees assume at least two distinct values. This level of generality, in particular, allows us to treat radial trees with disordered geometry as well as Schr\"odinger operators with Bernoulli-type singular potentials. Our methods are based on an interplay between graph-theoretical properties of radial trees and spectral analysis of the associated random differential and difference operators on the half-line.
	\end{abstract}
	
	\maketitle
	
	{\scriptsize{\tableofcontents}}

	
	\section{Introduction}
	\subsection{Description of Main Results}
	The central theme of this paper is Anderson localization for random models on tree graphs. In the first part of this work we establish spectral and dynamical localization for {\it continuum}  Laplace operators subject to random Kirchhoff vertex conditions on radial trees with disordered geometry. Specifically, we consider metric trees with random branching numbers and random edge lengths.  The second part of this paper addresses analogous questions for random second order difference operators on {\it discrete} radial trees with random branching numbers.  At the outset, we emphasize that our results are all proved under the minimal possible hypotheses. Namely, we assume that the random variables used to generate the trees take at least two distinct values. We will formulate this assumption more precisely as Hypothesis~\ref{hyp2}. In particular, we can handle the case of Bernoulli distributions, which is generally considered to be the most challenging setting.

To begin, let us describe the models. Let $\Gamma$ be a  metric tree with vertices $\mathcal V$, edges $\mathcal E$, and uniformly bounded edge lengths $\{\ell_e >0 : e\in \mathcal E\}$. We further assume that there is a unique vertex $o \in \cV$ with degree 1, which we call the \emph{root} of $\Gamma$; see, for example, Figure~\ref{pic}.  For each vertex $v$, $\gen(v)$ (the generation of $v$) is the combinatorial distance from $v$ to the root. One defines $\gen(e)$ for $e \in \mathcal E$ similarly. We consider the Laplace operator $\bbH:=-\frac{d^2}{dx^2}$ acting in $L^2(\Gamma)$. In order to ensure self-adjointness of $\bbH$, we impose a Dirichlet condition at $o$, that is,
\begin{equation}\label{eq:dirichletBCgraph}
f(o) = 0,
\end{equation}
as well as Kirchhoff vertex conditions  given by
	\begin{equation}
	\begin{cases}
	\ f\text{\ is continuous at \ }v, & v \in \mathcal{V}\\
	\sum\limits_{e \in \mathcal E : v \in e}\partial_{\nu}^{e}f(v)=q{(v)} f(v) & v \in \mathcal{V} \setminus\set{o},
	\end{cases}
	\lb{con3}
	\end{equation}
where $q:\cV\rightarrow\bbR$ is a real-valued function, and $\partial_{\nu}^{e}$ denotes the inward-pointed derivative along the edge $e \in \mathcal E$. The assumption that $\deg(o)=1$ is purely for convenience. If the root has degree 2 or higher, the Dirichlet condition \eqref{eq:dirichletBCgraph} implies that the operators we study decompose into a direct sum of operators covered by the $\deg(o)=1$ case. In the simplified case $\Gamma=\bbR_+$ the vertex conditions \eqref{con3} provide a rigorous description of the self-adjoint realization of  Schr\"odinger operators with zero-range potentials and coupling constants $q(v)$ (cf., e.g., \cite[Section III.2.1]{AGHH}, \cite[Section 1.4.1]{BK}).

 We denote the branching number of each vertex by $b(v) = \deg(v) - 1$ for $v \in \mathcal{V} \setminus \set{o}$. In this work, we assume that all quantities are \emph{radial}. That is to say, we assume that $q(v)$ and $b(v)$ depend only on $\gen(v)$ and $\ell_e$ depends only on $\gen(e)$. The three continuum random models treated in this paper are: the random branching model (RBM),  the random lengths model (RLM), and the random Kirchhoff model (RKM). In these models, the branching numbers, the Kirchhoff coupling constants, and the edge lengths are independent identically distributed Bernoulli-type random variables which depend only on the distance to the root $o$; the precise description of these models is provided in Section \ref{sec3.1}. In fact, our approach can allow all three parameters to vary simultaneously; we simply single out RBM, RLM, and RKM as prominent applications of our method. Thus, these models are parameterized by a choice of a probability measure $\widetilde \mu$ supported on a set of the form $\mathcal{A} = \set{b^-,b^-+1,\ldots, b^+} \times [\ell^-,\ell^+] \times [q^-,q^+]$, which gives the probability distribution for the branching numbers, the edge lengths, and the Kirchhoff potential at each generation. To be a little more specific, the probability space is $\Omega = \cA^{\bbN}$ with measure $\mu = \wti\mu^{\bbN}$; then, each $\omega \in \Omega$ produces a tree model with parameters dictated by
 \[
 b(v)
 =
 \omega_1(\gen(v)),
 \quad
 \ell_{e}
 =
 \omega_2(\gen(e)),
 \quad
 q(v)
 =
 \omega_3(\gen(v)),
 \quad
 v \in \cV,\ e \in \cE.
 \]
		
			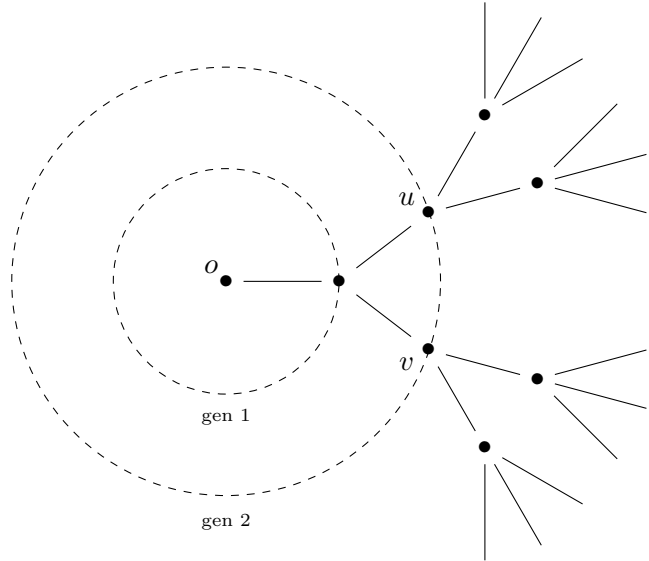
\begin{figure}
			\tikzstyle{level 1}=[sibling angle=120]
			\tikzstyle{level 2}=[sibling angle=75]
			\tikzstyle{level 3}=[sibling angle=45]
			\tikzstyle{level 4}=[sibling angle=30]
			\begin{tikzpicture}[grow cyclic, level distance=15mm]

			\draw[thin, dashed] (0,0) circle (1.5cm);
			\draw[thin, dashed] (0,0) circle (2.85cm);
			\node {{$\bullet$}}
			child {node {$\bullet$}
				child {node {$\bullet$}
					child {node {$\bullet$}child childchild}
					child {node {$\bullet$}child childchild}}
				child{node {$\bullet$}
					child {node {$\bullet$}child childchild}
					child {node {$\bullet$}child childchild}}}
			;
			\draw (-0.2,0.2) node {{$o$}};
			\draw (2.4, 1.1) node {{$u$}};	\draw (2.4, -1.1) node {{$v$}};
			\draw (0,-1.8) node {\tiny{gen 1}};
			\draw (0,-3.2) node {\tiny{gen 2}};
			\end{tikzpicture}
			\caption{\ $b_0=1, b_1=2, b_2=2, b_3=3$}\lb{pic}
		\end{figure}
		
	Our approach is based on the orthogonal decomposition of $L^2(\Gamma)$ into a countable collection of reducing subspaces of the operator $\bbH$; cf.\ \cite{NS}, \cite{NS2} (see also \cite{Car}). The restriction of $\bbH$ on each subspace is unitarily equivalent to a shifted version of the model half-line operator $H:=-\frac{d^2}{dx^2}$ acting in $L^2(\bbR_+)$, subject to the Dirichlet condition at $0$ and self-adjoint vertex conditions of the form
	\begin{equation}
	\begin{cases}
	\sqrt{b_j} f(t_{j}^-)=f(t_{j}^+), & j \in \bbN\\
	f'(t_{j}^-)+q_jf(t_j^-)={\sqrt{b_j}}{f'(t_{j}^+)} & j \in \bbN,
	\end{cases}\lb{vcon}	
	\end{equation}
	where $t_j$ denotes the distance from the root to vertices of generation $j\in\bbN$. Similarly $b_j$ denotes the branching number and $q_j$ is the Kirchhoff coupling constant at generation $j$.
	
	The natural occurrence of Bernoulli models in this paper is due to random branching; in particular, the branching at each generation may only take integral values, so any randomness in the branching parameter must necessarily be discrete. Broadly speaking, the behavior of random models (at least in one spatial dimension) tends to be monotone in the randomness. In particular, increasing the randomness of the model tends to make the spectrum more singular. Thus, proving localization statements in the situation in which the single-site distribution is supported on two points (the Bernoulli case) is the most challenging task.
	
To prove localization for the 1D half-line operator $H_{\omega}$, we adapt the approach of \cite{BuDaFi}, which itself fits into the general framework of spectral analysis via transfer matrix techniques, see, e.g., \cite{DSS, PF} for illuminating discussions. Recall that a \emph{generalized eigenfunction} is an solution $\psi$ of the eigenvalue equation $H_\omega \psi = E\psi$ that enjoys a linear upper bound; in this case, $E$ is known as the corresponding \emph{generalized eigenvalue}.

 For the proof, we first employ F\"urstenberg's Theorem to ensure positivity of the Lyapunov exponent away from a discrete set $\mathfrak D$ (Theorem~\ref{thm34}), and then show that almost surely all generalized eigenfunctions exhibit Lyapunov behavior in every compact interval $I\subset\bbR\setminus \mathfrak D$, (Theorem~\ref{main1}).   This shows that the generalized eigenfunctions decay exponentially, which establishes spectral localization. At that point, the established exponential decay of generalized eigenfunctions is combined with the proof of spectral localization to bootstrap sharper bounds for the eigenfunctions in terms of their centers of localization, cf.\ \eqref{426}. The latter are crucial for showing dynamical localization. We summarize this discussion by formulating the first main result of this work.
	\begin{theorem}\lb{thm48}
	Suppose $\supp\widetilde\mu$ contains at least two points.  Then there exists a discrete set $\mathfrak D\subset \bbR$ such that for every compact interval $I\subseteq \bbR\setminus \mathfrak D$ and  every $p>0$, there exists $\wti \Omega\subset\Omega$ with $\mu(\wti\Omega)=1$ such that
	\begin{equation}\lb{3.102}
	\sup\limits_{t>0} \left\||X|^p\chi_I(H_{\omega}) e^{-itH_{\omega}}\psi\right\|_{L^2(\bbR_+)}<\infty,\ \omega\in\wti\Omega,
	\end{equation}
	whenever $\psi\in L^2(\bbR_+)$ and
	\begin{equation}\lb{1.4nn}
	\psi(x)\underset{x\rightarrow\infty}{=}\cO (e^{-\log^{C} x}),
	\end{equation}
	for some universal constant $C>0$.
	\end{theorem}
	We prove this Theorem in Section~\ref{sec:half-line}. We deduce the second main result of the paper by combining Theorem~\ref{thm48} and the orthogonal decomposition of radial trees; see Section~\ref{rcm}.
	\begin{theorem}\lb{main2} Suppose $\supp\wti\mu$ contains at least two points. Then, there exists a discrete set $\mathfrak{D} \subseteq \R$ such that the following two assertions hold.
		\begin{enumerate}
		\item[{\rm(i)}] The operator $\bbH_{\omega}$ exhibits Anderson localization at all energies outside of $\mathfrak{D}$. That is, almost surely, $\bbH_{\omega}$ has pure point spectrum and any eigenfunction of $\bbH_\omega$ corresponding to an energy $E \in \sigma(\bbH_{\omega}) \setminus \mathfrak{D}$ enjoys an exponential decay estimate of the form
		\begin{equation} \label{eq:TEDest}
		|f(x)| \leq \frac{Ce^{-\lambda|x|}}{\sqrt{w_o(|x|)}}
		\end{equation}
with $C>0$ and $\lambda > 0$, where $w_o(|x|)$ denotes the number of vertices in the generation of $x$, i.e., $w_o(|x|) = \#\set{y \in \mathcal{V} : \gen(y) = \gen(x)}$.\vspace{2mm}
		\item[{\rm(ii)}] For every compact interval $I\in\bbR\setminus \mathfrak{D}$ and every $p>0$, there exists a set $ \Omega^*\subset\Omega$ with $\mu(\Omega^*)=1$ such that for every $\omega\in\Omega^*$ and every compact set $\cK\subset \Gamma_{b_{\omega}, \ell_{\omega}}$ one has
		\begin{equation}\lb{341n}
		\sup\limits_{t>0} \left\||X|^p\chi_I(\bbH_{\omega}) e^{-it\bbH_{\omega}}\chi_{\cK}\right\|_{L^2(\Gamma_{b_{\omega}, \ell_{\omega}})}<\infty,\
		\end{equation}
		where $\chi_I(\bbH_{\omega})$ is the spectral projection corresponding to $I$, and $|X|^p $ denotes the operator of multiplication by the radial function $f(x):=|x|^p$, $x\in \Gamma_{b_{\omega}, \ell_{\omega}}$, where $|x|$ denotes the distance from $x$ to the root $o$.
		\end{enumerate}
	\end{theorem}
We note that the theorem above gives localization for RBM, RLM, and RKM. We also note that the spectrum of $\bbH_{\omega}$ is given by a deterministic set.  This is addressed in Section \ref{almsure} where we also point out that the analogous question for the half-line operator $H_{\omega}$ presents some complications which are not typical for full--line ergodic models, see Remark \ref{remark4.2}.
\begin{remark}
A few remarks:
\begin{enumerate}
\item The assumption that the support of the single-generation distribution contains at least two points is clearly necessary. For, if $\supp\wti\mu$ consists of a single point, then there is only one operator $H_\omega$, which is then periodic and hence does not exhibit Anderson localization.
\item We will refer to functions on trees obeying an estimate like \eqref{eq:TEDest} as \emph{tree-exponentially decaying}. Since the number of vertices at the $n$th generation grows exponentially with $n$, the factor of $\sqrt{w_o(|x|)}$ in the denominator implies that the eigenfunction decay leads to square-integrability.
\item The transfer matrices for the half-line models can be bounded at isolated energies, and hence one cannot avoid excluding a discrete set of energies. This will be discussed in more detail in Section~\ref{sec:half-line}.
\end{enumerate}
\end{remark}
	In Part~\ref{sec:discrete} we address analogous questions for the discrete versions of RBM, RLM, and RKM, namely, we consider discrete Schr\"odinger and weighted adjacency operators on radial trees with random branching numbers, hopping parameters, and vertex potentials. Concretely, we consider rooted radial tree graphs $\Gamma$ as before. Given functions $q: \cV \to \R$ and $p:\cE \to (0,\infty)$, the corresponding Schr\"odinger operators $\bbA$ and $\mathbb S$ are given by
	\begin{equation}\lb{int18}
			[\bbA f](u)
		=
		-\sum_{v \sim u} p(u,v)f(v),
		\quad f \in \ell^2(\cV),\ v \in \cV.
	\end{equation}
	\begin{equation}\lb{int19}
	[\mathbb S f](u)
	=
	\sum_{v \sim u} (q(u) f(u) - f(v)),
	\quad f \in \ell^2(\cV),\ v \in \cV.
	\end{equation}
	
	As before, we will assume that $b$, $p$, and $q$ are bounded radial functions, so the randomness will be encoded in a measure $\widetilde\mu$ which gives the distribution of branching numbers, edge weights, and vertex potentials in each generation. We will define this more precisely in Part~\ref{sec:discrete}.  Our third main result is the following theorem which is proved in Section~\ref{dynexp}. The quantity $w_y(r)$ in \eqref{536} below denotes the number of points  in the subtree rooted at $y$ that are at a distance $r$ from $y$; see \eqref{eq:wvdef} for the definition.
		\begin{theorem}\lb{thm5.6} Assume $\supp\wti\mu$ contains at least two points. Let $\bbJ_{\omega}=\mathbb A_{\omega}$ or $\bbJ_{\omega}=\mathbb S_{\omega}$.  Then there exists a  set $\cD$ of cardinality at most one such that the following assertions hold.
		\begin{enumerate}
		\item[{\rm(i)}]
		The operator $\bbJ_{\omega}$ exhibits Anderson localization at all energies outside of $\cD$. That is, almost surely, $\bbJ_{\omega}$ has pure point spectrum and any eigenfunction of $\bbJ_\omega$ corresponding to an energy $E \in \sigma(\bbJ_{\omega}) \setminus {\cD}$ enjoys an exponential decay estimate of the form
		
		\begin{equation}\lb{1.8nn}
		|f(x)| \leq \frac{C e^{-\lambda|x|}}{\sqrt{w_o(|x|)}}, x\in\cV,
		\end{equation}
		where $C, \lambda > 0$ are constants.
		\item[{\rm(ii)}] For every compact interval $I\subset \bbR\setminus\cD$ there exist $\Omega^*\subset\Omega$ with $\mu(\Omega^*)=1$ and $\theta>0$ such that for every $x,y\in\cV$, $|x|\geq |y|$,   $\omega\in\Omega^*$ one has
		\begin{equation}\lb{536}
		\sup\limits_{t>0}|\langle\delta_x, \chi_I(\bbJ_{\omega})e^{-it\bbJ_{\omega}} \delta_y\rangle_{\ell^2(\cV)}|\leq \frac{Ce^{-\theta\dist(x,y)}}{\sqrt{w_y(|x|-|y|)}},
		\end{equation}
		for some $C=C(y,\omega, \theta)>0$. In particular, for all $y\in\cV$, $\omega\in\Omega^*$, $R>0$ one has
		\begin{equation}\lb{537}
		\sum_{|x|\geq R} \sup\limits_{t>0} |\langle\delta_x, \chi_I(\bbJ_{\omega})e^{-it\bbJ_{\omega}} \delta_y\rangle_{\ell^2(\cV)}|\leq \gamma e^{-\kappa R},
		\end{equation}
		for some $\kappa=\kappa(y)>0$ and $\gamma=\gamma(y)>0$.
		\end{enumerate}
	\end{theorem}
	It is proved in Section \ref{almsuredis} that the spectrum of $\bbA_{\omega}$ is given by a deterministic set.  It is interesting to contrast this result with the work of Klein \cite{K3} (see also \cite{ASW2, FHS} for alternative proofs), which works without the radial assumption. In that model, each vertex potential is an i.i.d. random variable, and that model exhibits absolutely continuous spectrum in suitable energy regions for small coupling; it therefore does not exhibit localization uniformly, whereas the model in this work does. In particular, the model of \cite{K3} is more random than this one, and yet the spectral type is more regular.
	
	Our work is motivated by the paper \cite{HiPo}, which investigated RLM and RKM, and can be viewed as a natural continuation of \cite{DS} where discrete RBM was considered. It is worth noting that the methods of \cite{HiPo} are not applicable in the present setting since they are based on spectral averaging and hence rely heavily on the assumption that the random variables are absolutely continuous. Of course, in the case of random branching numbers such a hypothesis cannot be made.  We stress again that RBM naturally presents the most challenging case of random models, which are commonly referred to as Bernoulli--Anderson-type models. A textbook discussion of some difficulties presented by Bernoulli-type potentials is provided in the Notes sections of Chapters 4, 7, and 12 of \cite{AWBook}.

	\subsection{Background} The spectral theory of Schr\"odinger operators on tree graphs has attracted a lot of attention cf., e.g., \cite{ASW1, ASW2, ASW3, Br07, Br07b, BDE, BF09, JBK, BL, Car, DS, EFK08, EFK11, EFK13, EHP, FHS1, FHS, FLSSS, GZ01, HM, HiPo, K1, K2, K3, NS, NS2, SST, SS, S}. The recurring topic in these works is the dependence of the spectrum of differential operators on the geometry of trees, in particular, on their growth rates, edge lengths, and branching numbers. For example,  Ekholm, Frank, and Kovarik established Lieb--Thirring inequalities which heavily depend on the growth rate and the global dimension of underlying trees, cf.\ \cite{EFK11}, and Frank and Kovarik obtained heat kernel estimates for various trees in \cite{EFK13}.  Evans, Harris, and Pick studied Hardy inequalities on trees in the context of eigenvalue counting for the  Neumann Laplacian on bounded domains with fractal boundaries cf.\ \cite{EH, EHP}. This topic was further developed by Naimark and Solomyak in \cite{NS, NS2}. As far as the discrete spectrum is concerned, Solomyak also obtained Weyl's asymptotic formula for compact metric trees with the standard power-law behavior replaced by $c(\Gamma)\sqrt{\lambda}\log{\lambda}$ (this hints on mixed dimensionality of the model) with $c(\Gamma)$ depending on the tree, cf.\ \cite{S}. Further, the dependence of the spectral type on the geometry was investigated by Breuer, Frank, and Kovarik in \cite{BF09, Br07}. Exponential decay of the eigenfunctions on trees (and more general graphs) was recently discussed by Harrell and  Maltsev in \cite{HM}. Aizenman, Sims, and Warzel studied the effects of disorder in the geometry of trees. In particular, they considered trees with edge lengths given by $\ell_e(\omega)=e^{\lambda \omega_e}$ where $\lambda\in[0,1]$ and $\set{\omega_e}_{e\in\cE}$ are i.i.d. random variables, and proved  in \cite{ASW1} that the absolutely continuous spectrum of the Laplace operator is continuous (in the sense of \cite[Theorem~1.1]{ASW1}) at $\lambda=0$ almost surely. That such a continuity property fails in the case of radial disorder is conjectured in \cite{ASW1} and proved by Hislop and Post in \cite{HiPo}. As already mentioned earlier, the existence of absolutely continuous spectrum for the Anderson Hamiltonian on the regular trees in the regime of small disorder was shown by Klein in \cite{K3} (and also by Aizenman, Sims, and Warzel in \cite{ASW2} as well as by Froese, Hasler, and Spitzer in \cite{FHS}). Thematically related recent results are due to Aizenman and Warzel \cite{AW11, AW13} showing delocalization near the spectral edges for random Schr\"odinger oprators on discrete trees.
	
	The structure of the paper follows. In Section~\ref{deterministic}, we discuss the spectral theory of deterministic continuum operators on metric tree graphs. We use this to set notation and to give the reader relevant background on a reduction from the metric tree graphs to Schr\"odinger operators on a half-line with singular potentials. We prove a localization result for these half-line operators in Section~\ref{sec:half-line}, which we then use to prove our main results for metric tree graphs in Section~\ref{rcm}. The case of discrete operators on random tree graphs is taken up in Part~\ref{sec:discrete}.
	
	\part{Anderson Localization for Continuum Radial Trees}
	\section{Spectral Theory of Deterministic Continuum Operators} \lb{deterministic}
	In this section we introduce deterministic Laplace operators on radial tree graphs, discuss their orthogonal decomposition, and establish several auxiliary results regarding the spectral theory of the one-dimensional half-line operators arising in such a decomposition.
	
	To set the stage, we fix a metric rooted tree $\Gamma = (\cV, \cE)$ with  vertices $\cV$, edges $\cE$, root $o\in\cV$, and edge lengths $\set{\ell_e}_{e\in\cE}$.  The shortest path connecting $x\in\Gamma$ and $y\in\Gamma$ and its length are denoted by $p(x,y)$ and $d(x,y)$, respectively, and $|x|:=d(o,x)$. The generation and the branching number of a vertex $v$ are defined by
	\[
	\gen(v) := \#\set{x\in\cV\setminus\set{v}: x\in p(o,v)}, \quad
	b(v) := \begin{cases} \deg (v)-1, & v\not=o,\\ 1 & v = o.
	\end{cases}
	\]
	In other words, $\gen(v)$ is the combinatorial graph distance from $v$ to the root and $b(v)$ is the number of children of $v$. For an edge $e = (u,v)$, we define $\gen(e) = \max(\gen(u), \gen(v))$. Furthermore, $T_v\subset \Gamma$ denotes the ``forward" subtree of $\Gamma$ rooted at $v$, that is,  $T_{v}:=\set{x\in\Gamma: v\in p(o,x), |v|\leq |x| }$; its branching function is given by
	\begin{equation}\lb{eq:wvdef}
	w_{v}(t):=\#\set{x\in T_v: d(v,x)=t},\ t>0.
	\end{equation}
	For example, given a vertex $v$, $w_o(|v|)$ counts the number of vertices in the same generation as $v$.

	\begin{hypothesis}\lb{na5}  $\Gamma$ is a rooted radial metric tree with bounded branching $b$ and bounded edge lengths, $\ell$, and $q: \cV \to \bbR$ is a bounded radial potential. More precisely:
    \begin{enumerate}		
		\item[{\rm(i)}]There are constants $b^\pm \in [2,\infty)$, $\ell^\pm \in (0,\infty)$ and sequences $b:=\set{b_n}_{n=0}^{\infty}$,  $\ell:=\set{ \ell_n}_{n=1}^{\infty}$ such that
		\begin{itemize}
			\item  $b(v)=b_{\gen(v)} \in [b^-,b^+]\cap \bbN$ for all $v\in\cV$ {\rm(}except, $b(o) = b_0  = 1${\rm)},
			\item  $\ell_e= \ell_{\gen(e)} \in [\ell^-,\ell^+]$ for all $e\in\cE$.
		\end{itemize}
		
		\item[{\rm(ii)}]  There are constants $q^\pm\in \R$ and a sequence $\set{q_n}_{n=1}^\infty$ such that $q(v)= q_{\gen(v)} \in [q^- ,q^+]$.
		\end{enumerate}
	\end{hypothesis}
When $\Gamma$ satisfies Hypothesis~\ref{na5}, we will write $\Gamma = \Gamma_{b,\ell}$ to emphasize the dependence of $\Gamma$ on the branching and length sequences.
	
Given $\Gamma$ satisfying Hypothesis~\ref{na5}, we equip $\bbR_+$ with a sequence of degree two vertices  $\set{t_j}_{j=1}^{\infty}$, where $t_j$ denotes the distance from the root to vertices at generation $j$, that is,
	\begin{equation}\lb{ts}
	t_0:=0,\ t_j:=\sum_{i=1}^{j}\ell_i,\ j>0.
	\end{equation}
	Then, we introduce the Sobolev spaces on such a chain of intervals
	\begin{align}
	&\hatt H^k(\bbR_+):=\bigoplus_{j=0}^{\infty} H^k(t_{j}, t_{j+1}),\; j\in\bbZ_+, \; k=0, 1, 2.
	\end{align}	
	A note on notation: throughout this paper, we write $\bbN$ for $\set{1,2,3,\ldots}$ and $\bbZ_+$ for $\bbN \cup\set{0}$. 	Let us note that we use the notation $\widehat{H}^k(\bbR_+)$ even though the exact composition of the space depends on the vertices $\set{t_j}_{j=0}^\infty$. Similarly, on $\Gamma$, we define
	\begin{align}
	&\hatt H^k(\Gamma):=\bigoplus_{e\in\cE} H^k(e),\ \|f\|^2_{\hatt H^k(\Gamma)}:=\sum_{e\in\cE} \|f\upharpoonright_{e}\|^2_{H^k(e)},\ k=0,1,2.
	\end{align}
	Notice that the elements of $\hatt H^k(\bbR_+)$ or $\hatt H^k(\Gamma)$ may be discontinuous at the vertices.

\subsection{Orthogonal Decomposition of Radial Trees} \lb{sec2.1}

Given  a radial tree $\Gamma_{b,\ell}$ and a potential $q$ satisfying Hypothesis~\ref{na5}, we consider the self-adjoint operator $\bbH = \bbH(b,\ell,q)$ defined by
	\begin{align}
	\begin{split}\lb{det1}
	&\bbH(b,\ell,q) :=-\frac{d^2}{dx^2},\ \bbH(b, \ell, q): \dom(\bbH(b, \ell, q))\subset L^2(\Gamma_{b,\ell})\rightarrow L^2(\Gamma_{b,\ell}),\\
	&\dom(\bbH(b, \ell, q))=\set{ f\in\hatt H^2(\Gamma_{b,\ell}):
	\begin{matrix}
	f\text{\ satisfies \eqref{eq:dirichletBCgraph} and \eqref{con3}}
	\end{matrix}}.\\
	\end{split}
	\end{align}
Due to the radial structure of the graph,  $L^2(\Gamma_{b, \ell})$ enjoys an orthogonal decomposition into $\bbH$-reducing subspaces; cf.\ \cite{Car}, \cite{NS}, \cite{SS}, \cite{S}.
	Namely, to every vertex $v\in\cV$ there corresponds an $\bbH$-reducing subspace $\cS_v$ such that
	\begin{align}\lb{24new}
	L^2(\Gamma_{b,\ell})=\bigoplus_{v\in \cV}\cS_v,\ \   \bbH P_{\cS_v} = P_{\cS_v}\bbH,
	\end{align}
	where $P_{\cS_v}$ denotes the orthogonal projection onto $\cS_v$ in $L^2(\Gamma_{b,\ell})$.  Furthermore, each subspace $\cS_v$ can be further decomposed into $b_{\gen(v)}-1$ subspaces, each of which is also  $\bbH$-reducing, that is,
	\begin{equation}\lb{25new}
	 \cS_v = \begin{cases}
	\bigoplus\limits_{k=1}^{b_{\gen(v)}-1} \cL_{v,k},\   & v\not=o,\\
	\qquad\cL_o,\  &v=o,
	\end{cases}
	\end{equation}
	and $\bbH P_{\cL_{v,k}} = P_{\cL_{v,k}}\bbH,$ $\bbH P_{\cL_{o}} = P_{\cL_{o}}\bbH$. Moreover, the reduced operators are unitarily equivalent to 1D Schr\"odinger operators acting in $L^2(\bbR_+)$. Concretely, the operators
	\begin{equation}
\bbH(b, \ell, q)P_{\cL_{v,k},},\   \bbH(b, \ell, q)P_{\cL_{o},}
	\end{equation}
 are unitarily equivalent to the operator
	\begin{equation}\lb{2.9newnew}
	H(T^{\gen(v)}b, T^{\gen(v)} \ell, T^{\gen(v)} q)\text{\ acting in\ } L^2(t_{\gen(v)}, \infty),\ v\in\cV
	\end{equation}
	where  $T$ denotes the left shift $(Tx)_n:=x_{n+1}$ and
	\begin{align}
	&H(T^\varkappa b, T^\varkappa \ell, T^\varkappa q):=-\frac{d^2}{dx^2},\\
	&H(T^\varkappa b, T^\varkappa \ell, T^\varkappa q): \dom(H(T^\varkappa b, T^\varkappa \ell, T^\varkappa q))\subset L^2(t_{\varkappa},\infty)\rightarrow L^2(t_{\varkappa},\infty) \lb{28op}\\
	&\dom(H(T^\varkappa b, T^\varkappa \ell, T^\varkappa q))=\set{f\in \hatt H^2(t_\varkappa ,\infty):\ \begin{matrix}
	f(t_\varkappa ^+)=0, \begin{matrix}
	\text{$f$ satisfies \eqref{vcon}\,}\\
	\text{for all }j>\varkappa\
	\end{matrix}
	\end{matrix}}\lb{29new}
	\end{align}
	for $\varkappa\in\bbZ_+$.
	The unitary map
	\begin{equation}
	\cU_{v,k}:  \cL_{v,k}\rightarrow L^2(t_{\gen(v)}, \infty), v\in\cV\setminus \set{o}, 1\leq k\leq b_{\gen(v)}-1,
	\end{equation}
	 realizing the equivalence is defined by
	\begin{equation}\lb{29nn}
	(\cU_{v,k}^{-1} f)(x)=\begin{cases}
	{\frac{\exp\left({\frac{2\pi\bfi j k}{b_{\gen(v)}}}\right)f(|x|)} {\sqrt{w_v(|x|)}}},& x\in T_v(j),\  1\leq j\leq b_{\gen (v)},\\
	0,& \text{otherwise},
	\end{cases}
	\end{equation}
	where $T_v(j)\subset T_v$ denotes the forward subtree determined by the $j$th edge emanating from the vertex $v$. Letting $k=0$ in \eqref{29nn}, one defines $\cU_{o}$.
	We point out that $(\cU_{v,k}^{-1} f)\in \dom(\bbH(b,\ell,q))$ whenever $f$ belongs to the domain of the operator defined in \eqref{2.9newnew}. Indeed, continuity of $\cU^{-1}_{v,k}f$ at $v$ is ensured by the Dirichlet condition \eqref{29new} while the Kirchhoff condition at $v$ is satisfied due to \eqref{29nn} and the fact that the sum of roots of unity is equal to zero. At all other vertices, one has continuity and the Kirchhoff condition by \eqref{vcon}.

	 Combining these unitary operators together, one defines
	\begin{equation}\lb{212new}
	\Psi_{b, \ell}:=\cU_{o} \oplus \bigoplus_{v\in\cV \setminus\set{o}} \bigoplus_{k=1}^{b_{\gen(v)}-1} \cU_{v,k},
	\end{equation}	and has, \cite[Theorem~4.1]{NS},
	\begin{align}\lb{na}
	&\Psi_{b, \ell}:  L^2(\Gamma_{b, \ell})\rightarrow   \bigoplus_{n=0}^{\infty} \bigoplus_{k=1}^{m(n)}  L^2(t_n,\infty),\\
	&\Psi_{b, \ell} \bbH(b, \ell, q)\Psi^{-1}_{b, \ell}\lb{na8}
	=\bigoplus_{n=0}^{\infty} \bigoplus_{k=1}^{m(n)} H(T^nb, T^n\ell, T^nq),\\
	&m(n):=\begin{cases}\lb{213nn}
	b_0\cdot b_1\cdots\cdot b_{n-1}\cdot(b_{n}-1),\ &n \geq 1,\\
	1,\ &n=0.
	\end{cases}
	\end{align}
	
	Next, we turn to the spectral analysis of $H(b,\ell, q)$ for fixed admissible $b,\ell,q$. First, the eigenvalue problem for this operator can be written in terms of suitable $\SL(2,\bbR)$ matrices. Namely, if $f$ is a solution to the problem
	\begin{equation}
	\begin{cases}\lb{na11}
	-f''=Ef,\ f(t_0)=0,\\[2mm]
	{f(t_{j}^+)}=\sqrt{b_j}f(t_{j}^-) & j \in \bbN\\[2mm]
	{f'(t_{j}^+)}=\frac{f'(t_{j}^-)+q_jf(t_j^-)}{\sqrt{b_j}}  & j \in \bbN, \\[2mm]
	f \in H^2(t_{j-1},t_j)  & j \in \bbN,
	\end{cases}
	\end{equation}
	then one has
	\begin{align} \label{eq:halfLineTMdiscrete}
	&\begin{bmatrix}
	f(t_j^+)\\
	f'(t_j^+)
	\end{bmatrix}=M^{E}(b_j,\ell_j,q_j)
	\begin{bmatrix}
	f(t_{j-1}^+)\\
	f'(t_{j-1}^+)
	\end{bmatrix}\text{\ for all\ }\ j\in\bbN,
	\end{align}
	where $M^{E}(\beta,\lambda,\varkappa):=D(\beta) S(\varkappa)R_{\sqrt{E} }(\lambda \sqrt{E})$ , $\Im(\sqrt{E})\geq 0$ and
	\begin{align}
	&D(\beta) := \begin{bmatrix}
	\beta^{1/2} & 0 \\ 0 & \beta^{-1/2}
	\end{bmatrix},
	\ S(\varkappa):=
	\begin{bmatrix}
	1 & 0 \\ \varkappa & 1
	\end{bmatrix},\
	R_\mu(\varphi) :=
	\begin{bmatrix}
	\cos\varphi &\frac{\sin\varphi}{\mu} \\
	-\mu\sin\varphi & \cos\varphi
	\end{bmatrix}\lb{227n}.
	\end{align}
	In this case, we can interpolate between the vertices to get
	\begin{equation}
	f(x)=
	f(t_{j-1}^+)\cos(\sqrt{E}(x-t_{j-1}))+\frac{f'(t_{j-1}^+)\sin(\sqrt{E}(x-t_{j-1}))}{\sqrt{E}},\lb{2.23}\\
	\end{equation}
	for all $x\in (t_{j-1}, t_j)$, $j\in\Z_+$. Conversely, given initial data $(f(0^+),f'(0^+))^\top$, then \eqref{eq:halfLineTMdiscrete} and \eqref{2.23} construct a solution to the problem \eqref{na11}. Furthermore, $f\in L^2(\bbR_+)$ if and only if
	\begin{align}\lb{225}
	\set{\begin{bmatrix}
	f(t_j^+)\\
	f'(t_j^+)
	\end{bmatrix}}_{j=0}^{\infty}\in \ell^2(\bbZ_+,\C^2).
	\end{align}

\subsection{Quadratic Form of the Model Half-Line Operator}

The following proposition describes the quadratic form of $H(b,\ell, q)$ and provides prerequisites for the Weyl criteria used in  the proof of later results (e.g.\ Theorem~\ref{prop3.1}).
	\begin{lemma}\lb{prop2.2} Assume Hypothesis~\ref{na5} and  consider the sesquilinear form $\gh = \gh(b,\ell,q)$ defined by
		\begin{align}
		&\gh: \dom(\gh)\times\dom(\mathfrak{h}) \rightarrow \bbC,\lb{2.11new}\\
		&\dom(\gh)
		  =\set{f\in\hatt H^1(t_0,\infty):\ \begin{matrix}
		f(0^+)=0,\\
		\sqrt{b_j} f(t_{j}^-)=f(t_{j}^+),\ j>0
		\end{matrix}},\lb{213}\\
		&\gh[u,v] = \langle u',v' \rangle_{L^2(t_0,\infty)}+ \sum_{j=1}^{\infty}q_j\overline{u(t_j^-)}v(t_j^-) \quad u,v\in\dom(\gh).\lb{2.13new}
		\end{align}
		Then $\gh$ is densely defined, closed, and bounded from below $($i.e. for some $\gamma\in\bbR$ one has $
		\mathfrak{h}[u,u]\geq \gamma \|u\|_{L^2(t_0,\infty)}^2,
		$ $u\in\dom(\mathfrak{h})$$)$. It is uniquely associated with the operator $H = H(b,\ell, q)$, that is,
		\begin{equation}\lb{nc1}
		\mathfrak{h}[u,v]
		= \langle u,Hv\rangle_{L^2(t_0,\infty)},
		\end{equation}
		for all $u\in\dom(\mathfrak{h})$ and $v\in \dom(H)$.
		Furthermore, there exist positive constants $c,C>0$ such that
		\begin{equation}\lb{2.9}
		c\|u\|^2_{\hatt H^1(t_0,\infty)}\leq (\gh - \gamma+1)[u,u]\leq C\|u\|^2_{\hatt H^1(t_0,\infty)},\, u\in\dom(\mathfrak{h}),
		\end{equation}
		where $\gamma$ is a lower bound of $\gh$. In addition, the space of compactly supported functions contained in $\dom(\gh)$ is a core of the form $\gh$.
	\end{lemma}
	\begin{proof}
		Throughout this proof we will abbreviate $\mathfrak h:=\mathfrak h(b,\ell, q)$ 	and $H:=H(b,\ell, q)$ for an admissible fixed triple $(b,\ell, q)$.  First, we show that $\mathfrak h$ is bounded from below. If $q^-\geq 0$, the  form is non-negative. Suppose that $q^-<0$.  By a standard Sobolev-type inequality (cf., e.g. \cite[Corollary~4.2.10]{Bu}, \cite[IV.1.2]{K80}) one has
		\begin{equation}\lb{2.10}
		\max \set{ |u(t^+_{j-1})|^2,|u(t^-_{j})|^2}\lesssim C\|u\|^2_{L^2(t_{j-1}, t_j)}+\varepsilon\|u'\|^2_{L^2(t_{j-1}, t_j)},
		\end{equation}
		for all $\varepsilon>0$ and  $j\in\bbN$, where $C=C(\varepsilon, \ell^-,\ell^+)>0$. Then
		\begin{align}
		\mathfrak{h}[u,u]&\gtrsim \|u'\|^2_{L^2(t_0,\infty)} + q^- C\|u\|_{L^2(t_0,\infty)}^2 + q^-\varepsilon\|u'\|^2_{L^2(t_0,\infty)} \lb{213new}\\
		&\geq  (1+q^-\varepsilon)\|u'\|^2_{L^2(\bbR_+)}+q^- C\|u\|_{L^2(t_0,\infty)}^2\lb{nc2}\\&\geq \gamma \|u\|_{L^2(t_0,\infty)}^2,\lb{2.14}
		\end{align}
		where we chose $\varepsilon>0$ so that $1+q^-\varepsilon> 0$ and set $\gamma:=q^- C$.
		
		Next, we prove that $\mathfrak h$ is closed, i.e., that $\dom(\mathfrak{h})$ is closed with respect to the topology induced by the inner product $\gh-\gamma+1$. First, using  \eqref{213new}--\eqref{2.14} one infers
		\begin{equation}\lb{2.12}
		(\mathfrak h-\gamma+1)[u,u] \gtrsim \|u\|_{\hatt H^1(t_0,\infty)}^2.
		\end{equation}
		Suppose that $\set{u_k}_{k\geq 1}\subset \dom(\mathfrak h)$ is a Cauchy sequence with respect to the inner product $\mathfrak h-\gamma+1$. In that case, it is Cauchy  in $\hatt H^1(t_0,\infty)$ and hence has a limit $u\in \hatt H^1(t_0,\infty)$:
		\begin{equation}\lb{2.13}
		u_k \underset{{\small\hatt H^1(t_0,\infty)}}{\longrightarrow} u, k\rightarrow\infty.
		\end{equation}
		In order to show that $\mathfrak h$ is closed, it is enough to prove that $u$ satisfies the vertex conditions at every vertex $t_j$. To that end, we notice that for all $k\in\bbN$, $j>0$ we have $\sqrt{b_j}u_k(t_{j}^-)=u_k(t_{j}^+)$. Then, by \eqref{2.10} and \eqref{2.13} we may pass to the limit as $k\rightarrow\infty$ and obtain $\sqrt{b_j}u(t_{j}^-)=u(t_{j}^+)$ for all $j> 0$. Similarly, we get $u(t_0^+)=0$.
		
		The first inequality in \eqref{2.9} is already proved; see \eqref{2.12}. The second one follows from the Cauchy--Schwarz inequality and the Sobolev-type estimate \eqref{2.10}.
		
		Next, we prove \eqref{nc1}. Notice that the subspace
		\begin{equation}
		\set{v\in\dom(H): \text{supp}(v) \text{\ is compact in }[t_0,\infty) }\subset\dom(H),\no
		\end{equation}
		is a core of $H$. Hence it is sufficient to check \eqref{nc1} for arbitrary $u\in\dom(\mathfrak{h})$, $v\in \dom(H)$ with supp$(v)\subset [t_0, t_K)$ for some $K\in\bbN$. One has
		\begin{align}
		&\langle u,Hv\rangle_{L^2(t_0,\infty)}=-\sum_{j=1}^{K}\int_{t_{j-1}}^{t_{j}} \overline{u(x)}v''(x) \, dx\no\\
		&\quad=\langle u',v'\rangle_{L^2(t_0,\infty)}+\overline{u(t_0^+)}v'(t_0^+)+\sum_{j=1}^{K}\overline{u(t_j^+)} v'(t_j^+)-\overline{u(t_j^-)} v'(t_j^-)\no\\
		&\quad=\langle u',v'\rangle_{L^2(t_0,\infty)}+\sum_{j=1}^{K}\overline{\sqrt{b_j}u(t_j^-)}\frac{{v'(t_j^-)+q_jv(t_j^-)}}{\sqrt{b_j}}-\overline{u(t_j^-)}{v'(t_j^-)} \\
		& \quad =\mathfrak h[u,v].\no
		\end{align}
	\end{proof}
	The following  Weyl-type criterion holds.
	\begin{proposition}\lb{rem2.3}
Assume Hypothesis~\ref{na5}, and denote $\gh = \gh(b,\ell,q)$ and $H = H(b,\ell,q)$ as in Lemma~\ref{prop2.2}.  Let  $D\subset\dom(\gh) $ be a dense subset with respect to the $\hatt H^1(t_0,\infty)$ norm {\rm(}or, equivalently, with respect to the norm $\|\cdot \|^2_{\mathfrak h} := (\gh-\gamma+~1)[\cdot,\cdot]${\rm)}. Then  $E\in\sigma(H)$ if and only if there exist $\set{\varphi_k}_{k=1}^{\infty}\subset D$ and $\set{m_k}_{k=1}^{\infty}\subset\bbN$ such that
		\begin{align}
		&\|\varphi_k\|_{L^2(t_0,\infty)}=1,\  \supp(\varphi_k)\subset[t_0,t_{m_k}], \lb{2.22n}\\
		&\sup\limits_{k\in\bbN}\|\varphi_k\|_{\hatt H^1(t_0,\infty)}<\infty,\lb{223n}\\
		&\sup\limits_{\substack{g\in\dom(\gh) \\ \|g\|_{ \hatt H^1(t_0,\infty)}\leq 1}}  (\gh - E)[\varphi_k, g]\rightarrow 0,\ k\rightarrow\infty.\lb{2.24n}
		\end{align}
	\end{proposition}
	\begin{proof}
Since the norm $\|\cdot\|_{ \hatt H^1(t_0,\infty)}$  is equivalent to the form domain norm $\|\cdot \|_{\mathfrak h}$, \eqref{2.22n}, \eqref{2.24n}, together with the standard Weyl's criterion cf., e.g, \cite[Proposition~1.4.4]{St01}, yield $E\in\sigma(H)$ proving the ``if" part.
		
		To prove the ``only if~" part we combine Weyl's criterion and the last part of Lemma~\ref{prop2.2} to obtain a sequence satisfying \eqref{2.22n}, \eqref{2.24n}. Without loss of generality we may assume that $\gamma
		\geq0$. In that case, one has
		\begin{align}
		\|\varphi_k\|^2_{\hatt H^1(t_0,\infty)}&\lesssim  | \gh[\varphi_k, \varphi_k]|
		= \sup\limits_{\substack{g\in\dom(\gh) \\ \|g\|_{\gh}= 1}}   |\gh[\varphi_k, g]|\\
		&\leq \sup\limits_{\substack{g\in\dom(\gh) \\ \|g\|_{\gh}= 1}}   |(\gh-E)[\varphi_k, g]|\\
		&\qquad+\sup\limits_{\substack{g\in\dom(\gh) \\ \|g\|_{\mathfrak h }= 1}}   | E\langle \varphi_k, g \rangle_{L^2(\bbR_+)}|\underset{k\rightarrow\infty}{=}o(1)+\mathcal{O}(1).
		\end{align}
		Thus \eqref{223n} holds as asserted.
	\end{proof}
	
	In the sequel we will refer to the Dirichlet--Neumann truncation of the half-line operator $H(b,\ell,q)$ defined as follows
	\begin{align}
	\begin{split}\no
	&H^{k}(b,\ell, q):=-\frac{d^2}{dx^2},\\
	&H^{k}(b,\ell, q): \dom(H^{k}(b,\ell, q))\subset L^2(t_0,t_{k})\rightarrow L^2(t_0,t_{k}),\\
	&\dom(H^{k}(b,\ell, q)) = \set{\hatt H^2(t_0,t_{k}):\ \begin{matrix}
	f(t_0^+)=f'(t_{k}^-)=0\\
	\text{$f$ satisfies \eqref{vcon}\ for all}\ 0< j<k
	\end{matrix}}.
	\end{split}
	\end{align}

	\begin{proposition}\lb{prop24}
	Let us fix $n\geq 1$, $E\not\in\sigma( H^n(b,\ell,q))$, and suppose that $u_\pm $ satisfy \eqref{vcon} for all $0<j<n$, $-u''_{\pm}=Eu_{\pm}$, $u_-(t_0^+) = u'_+(t_{n}^-)=0$, and $u'_-(t_0^+)=u_+(t_{n}^-)=1$.
	Then the Green function of the operator $H^n(b,\ell,q)$ is given by
	\begin{equation}\lb{238}
	G^E_{n}(x,y)=G^E_{[t_0,t_{n}]}(x,y):=\frac{1}{W(u_+, u_-)}\begin{cases}
	u_+( y) u_-( x),& y\geq x,\\
	u_+( x) u_-(y),& y\leq x,\\
	\end{cases}
	\end{equation}
	where $0\not=W(u_+, u_-)=u'_-(t_{n}^-)=u_+(t_0^+)$ denotes the Wronskian of linearly independent solutions $u_{\pm}$. That is, $(H^n(b,\ell,q)-E)^{-1}$ is an integral operator with the kernel $G^E_{[t_0,t_{n}]}$.
\end{proposition}
\begin{proof}
	For a fixed $g\in L^2(t_0,t_{n})$ the unique nonzero function $u$ satisfying all vertex conditions and solving the non-homogeneous differential equation $-u''-Eu=g$ is given by
	\begin{equation}
	u(y)= [\cR_Eg](y):=\int_{t_0}^{t_{n}} G^E_{[t_0,t_{n}]}(x,y)g(x)dx.
	\end{equation}
	Evidently, the operator $R_E$ is bounded and
	\begin{equation}
	(H^n(b,\ell,q)-E) \cR_E=\cR_E(H^n(b,\ell,q)-E) =I_{L^2(t_0,t_{n})},
	\end{equation}
	as asserted. Finally, evaluating the Wronskian at $t_{0}$ and $t_{n+1}$, we get
	\begin{equation}
		W(u_+, u_-)=u'_-(t_{n}^-)=u_+(t_0^+)
	\end{equation}
	(see also \cite[Lemma~D.12]{HiPo}).
\end{proof}

\section{Proof of Localization for Half-Line Random Operators} \lb{sec:half-line}

The main goal of this section is to prove dynamical and spectral localization for the random half-line operators $H_{\omega}$ arising in the orthogonal decomposition of $\bbH_{\omega}$. Theorem~\ref{thm34} ensures positivity of the Lyapunov exponent outside of a discrete set $\mathfrak{D}$. In Theorem~\ref{main1} we prove spectral localization and SULE for $H_{\omega}$. Finally, we conclude with the proof of Theorem~\ref{thm48}, which addresses dynamical localization.

\subsection{Description of Random Models}\lb{sec3.1}

The	\textit{random branching model} (abbreviated RBM) is described by a family of  Laplace operators subject to Neumann--Kirchhoff vertex conditions on radial metric trees with random branching numbers. In other words,  we assume Hypothesis~\ref{na5} with the following parameters
	\begin{equation}\lb{nb1}
	b = \set{b_{\omega}(n)}_{n\in\bbN}\subset \set{2,...,d},\ d\geq3,\  \quad \ell^-=\ell^+=1,\quad q^-=q^+=0,
	\end{equation}
	where $\set{b_{\omega}(n)}_{n\in\bbN}$ is a sequence of independent and identically distributed random variables whose common distribution contains at least two points in its support.
	
	The {\it random lengths model} (RLM) is given by a family of the Neumann--Kirchhoff Laplace operators on radial metric trees with random edge lengths. That is,  we assume Hypothesis~\ref{na5} with
	\begin{equation}\lb{nb3}
	b^-=b^+=d,\quad \ell = \set{\ell_{\omega}(n)}_{n\in\bbN}\subset[\ell^-,\ell^+],\quad  q^-=q^+=0,
	\end{equation}
	where $\set{\ell_{\omega}(n)}_{\in\bbN}$ is a sequence of independent and identically distributed random variables whose common distribution contains at least two points in its support.
	
	The {\it random Kirchhoff model} (RKM) is given by the Laplace operators subject to random $\delta$-type vertex conditions. That is, we assume Hypothesis~\ref{na5} with
	\begin{equation}\lb{nb2}
	b^-=b^+=d,\quad \ell^-=\ell^+=1, \quad q=\set{q_{\omega}(n)}_{n\in\bbN} \subset[q^-,q^+],
	\end{equation}
	where $\set{q_{\omega}(n)}_{n\in\bbN}$ is a sequence of independent and identically distributed random variables whose common distribution contains at least two points in its support.
	
	In order to unify these models we consider three-dimensional random variables with common  distribution $\wti\mu$.
	\begin{hypothesis}\lb{hyp2}
		Let $\widetilde \mu$ be a probability measure  with
		\[
		\supp(\wti \mu)\subset \mathcal{A}:
		=
		\set{b^-,\ldots,b^+} \times[\ell^-,\ell^+] \times[q^-,q^+].
		\]
		Suppose that $\supp(\wti \mu)$ contains at least two distinct points, and let $(\Omega, \mu) := (\mathcal{A}^{\bbN}, \wti\mu^{\bbN})$.
	\end{hypothesis}
	\begin{remark}We notice that
		\begin{itemize}
			\item RBM arises when $\supp\widetilde\mu \subseteq \set{b^-,\ldots,b^+} \times \set{1} \times \set{0}$,
			\item RLM arises when $\supp\widetilde\mu \subseteq \set{d} \times [\ell^-,\ell^+] \times \set{0}$,
			\item RKM arises when $\supp\widetilde\mu \subseteq \set{d} \times \set{1} \times [q^-,q^+]$.
		\end{itemize}
		
	\end{remark}
	
	For $\omega \in \Omega$ we denote the components of $\omega$ as $\omega(n) = (b_\omega(n),\ell_\omega(n), q_\omega(n))$, since we will use them to define the branching, edge lengths, and Kirchhof potential of an operator. In particular, the vertices in $\R_+$ are denoted $t_\omega(n)$. Given $\omega$, define the operators $\bbH_\omega=\bbH({b_\omega,\ell_\omega,q_\omega})$ acting in  $L^2(\Gamma_{b_{\omega}, \ell_{\omega}})$ as in \eqref{det1}. Similarly, for $j\in\bbZ_+$, define
	\begin{equation}\lb{ranops}
H_{T^j\omega}:= H(T^jb_{\omega}, T^j\ell_{\omega}, T^jq_{\omega})\text{\ acting in\ } L^2(t_{\omega}(j), \infty),
	\end{equation}
	as in \eqref{28op}, \eqref{29new} and let $\mathfrak h_{T^{j}\omega}$ denote the corresponding quadratic forms.

\subsection{Positivity of Lyapunov Exponents via F\"urstenberg's Theorem}

Inspired by \eqref{eq:halfLineTMdiscrete} and \eqref{227n}, we introduce an $\SL(2,\bbR)$-cocycle over $T$ (the left shift $\Omega \to \Omega$) as follows. First, let $\cA$, $b^\pm$, $\ell^\pm$, and $q^\pm$ be as in Hypothesis~\ref{hyp2}. For each $E \in \R$, \eqref{eq:halfLineTMdiscrete}--\eqref{227n} lead us to define $M^E : \cA \to \SL(2,\R)$ by
\begin{equation} \label{eq:MEdef}
\mathcal{A} \ni \alpha = (\beta,\lambda,\varkappa)\mapsto
M^E(\alpha)
=
D(\beta) S(\varkappa) R_{\sqrt{E}}(\lambda \sqrt{E}).
\end{equation}
This induces a map $M^E: \Omega \to \SL(2,\R)$ via $M^E(\omega) = M^E(\omega(1))$, and then a skew product
	\begin{equation}\no
	(T,M^E): \Omega\times\bbR^2\rightarrow  \Omega\times\bbR^2,\ (T,M^{ E})(\omega, v)=(T\omega, M^E(\omega)v).
	\end{equation}
	Then denoting  the $n$-step transfer matrix by
	\begin{equation}\lb{tm}
	M^{E}_n(\omega)=\prod_{r=n-1}^0M^{E}(T^r\omega)
	=
	M^E(T^{n-1}\omega)\cdots  M^E(T\omega)  M^E(\omega),\ n\in\bbN,\
	\end{equation}
	we note that the iterates over the skew product are given by $(T,M^{E})^n=(T^n, M^{E}_n)$.
	The Lyapunov exponent is defined by
	\begin{equation}\lb{LE1}
	L(E):=\lim\limits_{n\rightarrow\infty}\frac{1}{n}\int_{\Omega}\log \|M_n^{E}(\omega)\| \, d\mu(\omega).
	\end{equation}
	By Kingman's Subadditive Ergodic Theorem we have
	\begin{equation}\lb{LE2}
	L(E)=\lim\limits_{n\rightarrow\infty}F_n(\omega,E);\ F_n(\omega,E):=\frac1n\log\|M^E_n(\omega)\|,
	\end{equation}
	for $\mu$-almost every $\omega$.

\begin{remark}\label{r.2cocycles}
Let us note that there are two natural cocycles that one can work with here. In addition to the discrete cocycle just described, there is also the continuum cocycle $\widetilde{M}^E$ defined by
\[
\widetilde{M}^E_x(\omega) : \begin{bmatrix}
u(0^+) \\ u'(0^+)
\end{bmatrix}
\mapsto
\begin{bmatrix}
u(x^+) \\ u'(x^+)
\end{bmatrix}
\]
whenever $-u'' = Eu$ and $u$ satisfies the vertex conditions defining $\dom(H_\omega)$. Evidently,
\[
M^E_n(\omega)
=
\wti{M}^E_{t_\omega(n)}(\omega).
\]
This leads to a simple relationship between the Lyapunov exponents of $M^E$ and $\wti{M}^E$. By Birkhoff's Ergodic Theorem,
\[
\lim_{n\to\infty} \frac{1}{n} t_\omega(n)
=
\langle \ell \rangle  :=
\int_{\cA} \alpha_2 \, d\wti{\mu}(\alpha),
\]
the average length. Then, one has
\begin{equation}\label{e.2cocycles}
L(E) = \wti{L}(E) \cdot \langle \ell \rangle.
\end{equation}
\end{remark}

	Our next goal is to show that Lyapunov exponents are positive away from a discrete set of energies. To that end, we first recall  F\"urstenberg's Theorem and some related facts. In order to state F\"urstenberg's Theorem, let us recall that a few definitions from the general theory. A group $G \subseteq \SL(2,\bbR)$ is called \emph{strongly irreducible} if there does not exist a finite set $\Lambda \subseteq \bbR\bbP^1$ such that $\{gv :v \in \Lambda\} = \Lambda$ for all $g \in G$; $G$ is called \emph{contracting} if there exist $g_n \in G$, $n \geq 1 $ such that $\|g_n\|^{-1} g_n$ converges to a rank-one operator as $n \to \infty$. Given  Borel probability measures $\nu_k$ supported in $\SL(2,\R)$, $k \geq 1$, we say $\nu_k \to \nu$ \emph{weakly and boundedly} if
\begin{equation} \label{eq:boundedconv}
\int_{\|M\| \geq N} \! \log^+\|M\| \, d\nu_k(M)
+
\int_{\|M\| \geq N} \! \log^+\|M\| \, d\nu(M)
\to 0
\end{equation} as $N \to \infty$, uniformly in $k$ and
\[
\int \! f \, d\nu_k\to \int \! f \, d\nu
\]
for all the space of continuous complex-valued functions $f:\SL(2,\R) \to \bbC$ having compact support.
	\begin{theorem}\lb{a16}
		Let $\nu$ be a probability measure on $\SL(2,\bbR)$ satisfying
		\begin{equation}\no
		\int \log\|M\| \, d\nu(M) < \infty.
		\end{equation}
		Let $G_{\nu}$ be the smallest closed subgroup of $\SL(2,\bbR)$ that contains $\supp \nu$.
		\begin{enumerate}
		\item[{\rm(i)}] \cite[Theorem~8.6]{F63} Assume that $G_{\nu}$ is not compact and that it is strongly irreducible. Then the Lyapunov exponent $L(\nu)$ associated with $\nu$ is positive.
		
		\item[{\rm(ii)}] \cite[Theorem~B]{FuKi} Assume that the set
		\begin{equation}\no
		\emph{Fix}(G_{\nu}):=\set{V\in\bbR\bbP^1: MV=V\text{\ for every\ } M\in G_{\nu}}
		\end{equation}
		contains at most one element. If $\nu_k\rightarrow\nu$ weakly and boundedly, then $L(\nu_k)\rightarrow L(\nu)$ as $k\rightarrow\infty$.
		\end{enumerate}
	\end{theorem}

In the present setting, we have a one-parameter family of measures induced on $\SL(2,\bbR)$, namely, we consider $\nu_E$, the pushforward of $\widetilde\mu$ under the map $M^E$ in \eqref{eq:MEdef}.
	\begin{theorem}\lb{thm34}
		Assume Hypothesis~\ref{hyp2}. Then there is a discrete set $\mathfrak{D} \subseteq \bbR$ such that $G = G_{\nu(E)}$ enjoys the following properties for $E \in \bbR \setminus \gD$.
		\begin{enumerate}
		\item[{\rm(i)}] $G$ is noncompact
		\item[{\rm(ii)}] $G$ is strongly irreducible
		\item[{\rm(iii)}] $G$ is contracting
		\item[{\rm(iv)}] $\mathrm{Fix}(G) = \emptyset$
		\end{enumerate}
In particular, $L$ is continuous and positive on $\bbR \setminus \mathfrak{D}$.
	\end{theorem}
	
	\begin{proof}
	In view of Theorem~\ref{a16}, positivity follows from (i) and (ii), while continuity on $\bbR\setminus\gD$ follows from (iv). Moreover, (ii)$\implies$(iv), so we only need to prove (i)--(iii). Write
		\begin{align*}
		&M^E(\beta,\lambda,\varkappa)
		=
		D(\beta) S(\varkappa) R_{\sqrt{E}}(\lambda\sqrt{E}) \\
		& =
		\begin{bmatrix} \sqrt{\beta} & 0 \\ 0 & \frac{1}{\sqrt{\beta}}\end{bmatrix}
		\begin{bmatrix} 1 & 0 \\ \varkappa & 1 \end{bmatrix}
		\begin{bmatrix} \cos(\lambda\sqrt{E}) & \frac{\sin(\lambda\sqrt{E})}{\sqrt{E}} \\ - \sqrt{E}\sin(\lambda\sqrt{E}) & \cos(\lambda\sqrt{E}) \end{bmatrix} \\
		& =
		\small{\begin{bmatrix}
			\beta^{1/2}\cos(\lambda\sqrt{E}) & \beta^{1/2}\frac{\sin(\lambda\sqrt{E})}{\sqrt{E}} \\
			\varkappa\beta^{-1/2} \cos(\lambda\sqrt{E}) - \beta^{-1/2}\sqrt{E}\sin(\lambda\sqrt{E}) &  \frac{\varkappa\sin(\lambda \sqrt{E})}{\beta^{1/2}\sqrt{E}} +  \beta^{-1/2}\cos(\lambda\sqrt{E})
			\end{bmatrix}}
		\end{align*}
		Now, let $(b_1,\ell_1,q_1) \neq (b_2,\ell_2,q_2)$ be distinct elements of $\supp\widetilde\mu$, abbreviate
		\[
		M_j = M_j(E): = M^E(b_j,\ell_j,q_j),
		\]
		and define the commutator
		\begin{equation}
		g= g(E) = [M_1,M_2] = M_1M_2 - M_2 M_1.\lb{ge}
		\end{equation}
		
To conclude the proof, it suffices to show that $g(E)$ does not vanish identically. Concretely, it is easy to see that the matrices $M_j$ are analytic functions of $E$ with non-constant trace and that the entries of $M_j$ are real  whenever $\tr M_j \in [-2,2]$. Thus, the matrices $M_j(E)$  satisfy the first three hypotheses of \cite[Theorem~2.1]{BuDaFi2}, so, if $g(E)$ does not vanish identically, we can conclude that there is a discrete set $\gD$ such that (i)--(iii) hold for $E \in \bbR \setminus \gD$ by \cite[Theorem~2.1]{BuDaFi2}.
		\medskip
		
		To that end, suppose for the purpose of establishing a contradiction that $g$ vanishes identically. In particular, the upper left matrix element $g_{11}(E)$ vanishes identically. One may calculate  $g_{11}(E)$ directly:
		\begin{align*}
		g_{11}(E)
		& =
		b_1^{1/2} \frac{\sin(\ell_1\sqrt{E})}{\sqrt{E}} \left( q_2b_2^{-1/2} \cos(\ell_2\sqrt{E}) - b_2^{-1/2}\sqrt{E} \sin(\ell_2\sqrt{E})  \right) \\
		& \qquad \quad - b_2^{1/2} \frac{\sin(\ell_2\sqrt{E})}{\sqrt{E}} \left( q_1b_1^{-1/2} \cos(\ell_1\sqrt{E}) - b_1^{-1/2}\sqrt{E} \sin(\ell_1\sqrt{E})  \right).
		\end{align*}
		For ease of notation, write $r_1 = b_2^{1/2} / b_1^{1/2}$, $r_2 = b_1^{1/2} / b_2^{1/2}$, and $w = \sqrt{E}$. Expanding the trigonometric functions, we get
		\begin{align}
		g_{11}
		& =
		\frac{q_2 r_2}{4iw} (e^{i\ell_2 w} + e^{-i\ell_2 w})(e^{i\ell_1 w}- e^{-i\ell_1 w}) \\
		&\quad - \frac{q_1 r_1}{4iw} (e^{i\ell_1 w} + e^{-i\ell_1 w})(e^{i\ell_2 w} - e^{-i\ell_2 w}) \\
		& \quad - \frac{r_1-r_2}{4}(e^{i\ell_1 w} - e^{-i\ell_1 w})(e^{i\ell_2 w} - e^{-i\ell_2 w}).
		\end{align}
		Thus,
		\begin{equation}
		\label{eq:g11:expExpansion}
		\begin{split}
		4iw^2 g_{11} &=
		\left( q_2 r_2 w - q_1 r_1 w - iw^2(r_1-r_2) \right) e^{i(\ell_1+\ell_2)w}\\
		&\quad+ \left( q_1 r_1 w - q_2 r_2 w - iw^2(r_1-r_2) \right) e^{-i(\ell_1+\ell_2)w} \\
		&\quad  + \left( q_2 r_2 w + q_1 r_1 w + iw^2(r_1-r_2) \right) e^{i(\ell_1-\ell_2)w} \\
		&\quad + \left( -q_1 r_1 w - q_2 r_2 w + iw^2(r_1-r_2)\right) e^{-i(\ell_1-\ell_2)w}.
		\end{split}
		\end{equation}
Since $g_{11}$ vanishes identically and $\ell_1,\ell_2>0$, this forces
		\begin{align}
		q_2 r_2 w - q_1 r_1 w - iw^2(r_1-r_2) & \equiv 0 \\
		q_1 r_1 w - q_2 r_2 w - iw^2(r_1-r_2) & \equiv 0
\end{align}
		It is easy to see that this yields $r_1=r_2$ (hence $b_1 = b_2$) and $q_1 = q_2$. Since $(b_1,\ell_1,q_1) \neq (b_2,\ell_2,q_2)$, we must have $\ell_1 \neq \ell_2$. Going back to \eqref{eq:g11:expExpansion}, this implies	
		\begin{align*}
		q_2 r_2 w + q_1 r_1 w + iw^2(r_1-r_2) & \equiv 0 \\
		-q_1 r_1 w - q_2 r_2 w + iw^2(r_1-r_2) & \equiv 0.
		\end{align*}
and hence $q_1 = q_2 = 0$. Writing $b_1 = b_2 =: b$, and substituting $q_1=q_2=0$, we may directly calculate $g$:
		\begin{equation} \label{eq:RLMg}
		g(E)
		=
		\begin{bmatrix}
		0 & \frac{b-1}{\sqrt{E}} \sin\left((\ell_2-\ell_1)\sqrt{E}\right) \\
		\frac{b-1}{b}\sqrt{E} \sin\left((\ell_2-\ell_1)\sqrt{E}\right) & 0\end{bmatrix}
		\end{equation}
		which clearly only vanishes on the discrete set
		\begin{equation}
\mathfrak{D} = \set{(\ell_1-\ell_2)^{-2} \pi^2 k^2 : k \in \Z_+},
		\end{equation}
		a contradiction.

	\end{proof}
	The proof above implicitly uses the following statement.
	\begin{lemma}
		Suppose $\set{a_j : j=0,\ldots,n}$ is a set of $n+1$ distinct complex numbers and $\set{p_j : j=0,\ldots,n}$ are polynomials in $z$. Then, the function
		\[
		Q(z)
		:=
		\sum_{j=0}^n p_j(z) e^{a_j z}
		\]
		vanishes identically if and only if $p_j \equiv 0$ for each $j$.
	\end{lemma}
	
	\begin{proof}
		Write $D = d/dz$ and $M=\max(\deg(p_j))$. Suppose on the contrary that
		\[
		p_0(z) e^{a_0 z}
		\equiv
		\sum_{j=1}^n p_j(z) e^{a_j z}
		\]
		with $p_0 \not\equiv 0$. Notice that $\prod_{j=1}^n (D-a_j)^{M+1}$ annihilates the right hand side. However, if $b \neq a_0$, one readily verifies that
		\[
		(D-b)[p_0(z) e^{a_0 z}]
		=
		\widetilde{p}_0(z) e^{a_0 z},
		\]
		where $\widetilde{p}_0$ has the \emph{same} degree as $p_0$. Consequently, a straightforward induction implies that
		\[
		\prod_{j=1}^n (D - a_j)^{M+1} [p_0(z) e^{a_0 z}]
		\]
		does not vanish identically, a contradiction.
	\end{proof}
	\begin{remark}\lb{rem3.6} Let us make a few comments about the proof of Theorem~\ref{thm34}.
	\begin{enumerate}
\item		Since the argument above is  soft, we do not get any information about $\mathfrak{D}$, except that $\mathfrak{D}$ is discrete. However, in concrete situations in which one has more information, one can say more. For example, the $g$ from \eqref{eq:RLMg} corresponds to the RLM; we can explicitly see that $\mathfrak{D} = \set{(\ell_1-\ell_2)^{-2} \pi^2 k^2 : k \in \Z_+}$. For another example, in the RBM, one has $\supp\widetilde\mu \subseteq \set{b_-,\ldots,b_+} \times \set{1} \times \set{0}$, so one can choose $(b_1,1,0) \neq (b_2,1,0) \in \supp\widetilde\mu$. After some calculations, one obtains
		\[
		\det g
		=
		-\frac{(b_1 - b_2)^2}{b_1 b_2} \sin^2(\sqrt{E}),
		\]
		so F\"urstenberg's Theorem holds away from $\mathfrak{D} = \set{\pi^2 k^2 : k \in \Z_+}$. In this setting there exists a finite set of invariant directions at these special energies. That said, we note that the Lyapunov exponent is still positive by direct calculation.
		
\item		Let us also remark that the transfer matrices may be bounded at a discrete set of energies (compare \cite{DSS2}). For example, take parameters $(b_1,\ell_1,q_1) = (2,1,0)$ and $(b_2,\ell_2, q_2) = (2,3,0)$. Then, at energies $E = \frac{1}{4}\pi^2(2k+1)^2$ with $k \in \Z_+$, $M_1$ and $M_2$ are commuting and elliptic.\footnote{I.e., $|\tr M_j|<2$.} In particular,  the transfer matrices at these energies  are uniformly bounded, so \cite[Corollaries~2.1 and 2.2]{DLS2006} suggest that dynamical localization as formulated in Theorem~\ref{main2}.(ii) cannot hold without excluding these energies.
\end{enumerate}
	\end{remark}

\begin{remark}
As far as  spectral localization is concerned, it suffices to ensure that for every compact interval $I\in\bbR\setminus \mathfrak{D}$, almost surely all generalized eigenvalues exhibit Lyapunov behavior. We will construct a full measure set $\Omega^*\subset \Omega$ such that one has	
\begin{equation}\lb{newquation}
0<L(E)=\lim\limits_{n\rightarrow\infty}\frac{1}{n}\log \|M_n^{E}(\omega)\|
\end{equation}
for every generalized eigenvalue $E\in I$ of $H_{\omega}$ ($M_n^{E}(\omega)$ is defined in \eqref{tm}). As discussed in \cite{BuDaFi}, the work of Gorodetski and Kleptsyn \cite{GK17} shows that dropping the assumption that $E$ is a generalized eigenvalue invalidates the above assertion.
\end{remark}

\subsection{Dynamical Localization for Half-Line Operators}

Our approach relies on the Large Deviation Theorem (LDT) \cite[Theorem~3.1]{BuDaFi}. Although this  is not stated explicitly in \cite{BuDaFi}, the LDT and its corollaries \cite[Theorem~4.1, Corollary~5.3, (5.13)]{BuDaFi} are applicable whenever the conditions of the F\"urstenberg Theorem are met, the corresponding subgroup is contracting and the transfer matrices satisfy Lipschitz estimates which are supplied by the following lemma.
	\begin{lemma}\lb{lem3.7}
	Fix a compact interval $I \subseteq \R$. There are constants $C  > 0$, $\rho > 0$ such that
\begin{equation}\lb{321nn}
	\|M^E_n(\omega) - M^{E'}_n(\omega')\|
\leq
C n\rho^{n-1} \big( |E-E'| + \|\omega - \omega'\|_\infty \big)
\end{equation}
	for all $\omega,\omega' \in \Omega$, $E,E' \in I$, and $n \in \Z_+$. The constants depend only on $I$ and $\supp\widetilde \mu$. Consequently,
	\begin{equation}\lb{322nn}
	|F_n(\omega,E) - F_n(\omega',E')
	\leq
	C\rho^{n-1}(|E-E'| + \|\omega - \omega'\|_\infty),
	\end{equation}
	where $F_n$ is defined as in \eqref{LE2}.
	\end{lemma}
	
	\begin{proof}
	Let $n$, $E$, $E'$, $\alpha = (\b,\varkappa,\lambda) \in \cA$, and $\alpha' = (\b',\varkappa', \lambda')\in \cA$ be given. One immediately has
	\begin{equation} \label{eq:StypeLipschitz}
	\|S(\varkappa) - S(\varkappa')\|
	= |\varkappa - \varkappa'|
	\end{equation}
	and
\begin{equation} \label{eq:DtypeLipschitz}
\|D(\b) - D(\b')\| =\left| \sqrt{\b} - \sqrt{\b'} \right| \le \frac{1}{2\sqrt{2}} |\b - \b'|
\end{equation}
since $\b,\b' \geq 2$. Writing $\kappa = \sqrt{E}$, and $\kappa' = \sqrt{E'}$, we get
	\begin{align}  \nonumber
	\|R_\kappa(\lambda \kappa) - R_{\kappa'}(\lambda' \kappa')\|
	& \leq \|R_\kappa(\kappa\lambda) - R_{\kappa'}(\kappa'\lambda)\| + \|R_{\kappa'}(\kappa'\lambda) - R_{\kappa'}(\kappa'\lambda')\| \\
	& \leq
	C(\ell^\pm,I) (|E-E'| + |\lambda - \lambda'|).\label{eq:RtypeLipschitz}
	\end{align}
Using the triangle inequality to change a single one-step transfer matrix at a time, one has
	\begin{align*}
	&\|M_n^E(\omega) - M_n^{E'}(\omega') \| \\
	& \quad\leq \sum_{k=0}^{n-1} \Big\| M_{n-k-1}^{E'}(T^{k+1}\omega') (M_1^E(T^k\omega) - M_1^{E'}(T^k \omega')) M_{k}^E(\omega) \Big\|,
	\end{align*}
	where $T$ is the left shift operator.  Writing
	\begin{equation}\lb{rho}
		\rho
	=
	\sup\set{ \|M^E_{1}(\omega) \| : E \in I, \; \omega \in \Omega },
	\end{equation}
we can estimate the first and third factors by $\rho^{n-k-1}$ and $\rho^k$ respectively. On other other hand, \eqref{eq:StypeLipschitz}, \eqref{eq:DtypeLipschitz}, and \eqref{eq:RtypeLipschitz} yield
\[
\|M_1^E(T^k\omega) - M_1^{E'}(T^k\omega')\|
\leq
C(|E - E'| + \|\omega-\omega'\|_\infty),
\]
so, putting everything together, we have
	\begin{align*}
\|M_n^E(\omega) - M_n^{E'}(\omega') \|
&	\leq
\sum_{k=0}^{n-1} C \rho^{n-1} (|E-E'| + \|\omega - \omega'\|_\infty) \\
& =
Cn\rho^{n-1} (|E-E'| + \|\omega - \omega'\|_\infty),
\end{align*}
proving the first inequality. The second follows from this and the statement $|\log a - \log b| \leq |a-b|$ for $a,b \geq 1$.
	\end{proof}
	
	Having established Theorem~\ref{thm34} and Lemma~\ref{lem3.7}, we may utilize the LDT in our setting. In particular, we have the following:
	\begin{theorem}\lb{thm41} Assume Hypothesis~\ref{hyp2} holds true.
\begin{enumerate}
		\item[{\rm(i)}]  \cite[Theorem~3.1]{BuDaFi}  For any $\varepsilon > 0$, there exist $C,\eta > 0$ such that
		\begin{equation}\lb{42n}
		\mu\set{\omega \in \Omega : \left| L(E) - \frac{1}{n}\log \|M_n^{E}(\omega)\|  \right| \geq \varepsilon}
		\leq
		C e^{-\eta n},
		\end{equation}
		for all $n\ge 0$ and all $E \in I$.
		\item[{\rm(ii)}]  \cite[Theorem~4.1]{BuDaFi}  There exist constants $C=C(I, \wti\mu)$, $\beta=\beta(I, \wti\mu)>0$ such that
		\begin{equation}\lb{44n}
		|L(E)-L(E')|\leq C|E-E'|^{\beta},\  E, E'\in I.
		\end{equation}
		
		\item[{\rm(iii)}]  \cite[Corollary~5.3]{BuDaFi} For every $\varepsilon\in(0,1)$ there exists a full measure set $\Omega_{1}(\varepsilon)$ with   $\mu(\Omega_{1}(\varepsilon))=1$ such that for every $\omega\in\Omega_1(\varepsilon)$ there exists $n_1=n_1(\varepsilon,\omega)$ such that
		\begin{equation}\lb{45new}
		\frac{1}{n}\log \|M_n^{E}(T^{\zeta_0}\omega)\|\leq L(E)+\varepsilon ,
		\end{equation}
		for any $\zeta_0\in\bbZ_+$ and $n\geq \max(n_1, \log ^2(\zeta_0+1))$.\newline
		\item[{\rm(iv)}]
		For every $\varepsilon\in(0,1)$ there exists $\Omega_2(\varepsilon) \subseteq \Omega$, $\mu(\Omega_2(\varepsilon))=1$ with the following property: For every $\omega \in \Omega_2(\varepsilon)$, there exists $n_2 = n_2(\omega,\varepsilon)$ such that
		\begin{equation}\lb{413new}
		\left| L(E) - \frac{1}{n^2} \sum_{s=0}^{n^2-1}\frac{ \log\|M^E_n(T^{\zeta+sn}\omega)\|}{n} \right|
		<
		\varepsilon,
		\end{equation}
		for all $\zeta\in\bbZ_+$, $n \geq \max (n_2, \log^{\frac23}(\zeta+1))$, and $E \in I$.
		\end{enumerate}
	\end{theorem}

	Part (iii) yields
	\begin{equation}\lb{46new}
	\mu\set{\omega: \text{for all } E\in I,\  \limsup\limits_{n\rightarrow\infty}\frac{1}{n}\log \|M_n^{E}(\omega)\|\leq L(E)}=1. \
	\end{equation}
	This fact may also be derived from the Craig--Simon approach \cite{CS} (see also \cite{JZ}). Our main focus is on showing
	\begin{equation}\lb{47new}
	\mu\set{\omega \; : \; \begin{aligned}
	&  \liminf\limits_{n\rightarrow\infty}\frac{1}{n}\log \|M_n^{E}(\omega)\|\geq L(E) \\
	& \quad \text{for all {\it generalized eigenvalues}}\  E\in I \end{aligned} }=1.
	\end{equation}

	The next key step is an analog of the elimination of double resonances. Let us note that we do not use the typical formulation of double resonances (cf., e.g., \cite[(9.21)]{Ki08}), since our ultimate goal is to work with transfer matrices in order to apply the Avalanche Principle. The resonances we wish to exclude are those for which there are large disjoint intervals $I_1,I_2 \subseteq \Z$ so that some energy $E$ is very close to an eigenvalue of $H_\omega$ restricted to $I_1$, and the norm of the transfer matrix across $I_2$ at energy $E$ deviates substantially from $\exp(|I_2|L(E))$. In particular, we would like to show that this event occurs with very small probability, see \cite{BoSc1}. We shall make this precise and quantitative in Theorem~\ref{dr}.
	
	   By convention, we write $\|(H^{n}_{\omega}-E)^{-1}\|_{\cB(L^2(t_0, t_{n}))}=+\infty$ whenever $E\in\sigma(H^{n}_{\omega})$. Let us recall $F_n(\omega, E)$ from \eqref{LE2}, and abbreviate $\overline{K}:= \lfloor K^{\log K}\rfloor$.
	\begin{theorem}\lb{dr}
		Given $\varepsilon\in(0,1)$, $N\in\bbN$, let
		\begin{equation}\no
		\cD_N(\varepsilon):= \set{
		\omega\in\Omega\, : \, \begin{aligned}
		&\qquad\quad\text{for some\ } \zeta\in\bbZ_+,\ E\in I,\\
		&K\geq \max\{N,\log^2(\zeta+1)\},\  0<n\leq K^9,\text{\ one has:\ }\\
		&\begin{cases}
\|(H^{\zeta+n}_{\omega}-E)^{-1}\|_{\cB(L^2(t_0, t_{\zeta+n}))}\geq e^{K^2}\\
and\  |F_m(T^{r+\zeta}\omega,E)|\leq L(E)-\varepsilon\\
\text{for some\  }K^{10}\leq r\leq \overline{K}, m\in\{K,2K\}
		\end{cases}
		\end{aligned}
		}
		\end{equation}
		Then there exist $C=C(\varepsilon)>0$, $\eta(\varepsilon)>0$ such that
		\begin{equation}\lb{410m}
		\mu(\cD_N(\varepsilon))\leq Ce^{-\eta N}.
		\end{equation}
		In particular, one has
		\begin{equation}\lb{411m}
		\mu(\Omega_3(\varepsilon))=1\text{\ where \ }  \Omega_3(\varepsilon):=\Omega\setminus\limsup\limits_{N\rightarrow\infty}\cD_N(\varepsilon).
		\end{equation}
	\end{theorem}
	\begin{proof}
		Let us fix
		\begin{equation}\lb{412n}
		\zeta\in\bbZ_+,K\geq\max \set{N,\log^2(\zeta+1)}, 0<n\leq K^9,\  K^{10}\leq r\leq \overline{K}, j\in\set{1,2},
		\end{equation}
		and denote
		\begin{equation}\no
		\cD_j(K, n, r, \zeta):= \set{
		\omega\in\Omega\, : \, \begin{aligned}
		&\text{\ for some\  }E\in I,  \text{\ one has}\\
		&\|(H^{\zeta+n}_{\omega}-E)^{-1}\|_{\cB(L^2(t_0, t_{\zeta+n}))}\geq e^{K^2}\text{\ and\ }\\
		&|F_{jK}(T^{r+\zeta}\omega,E)|\leq L(E)-\varepsilon
		\end{aligned}
		}
		\end{equation}
		In order to estimate $\mu(\cD_j(K, n, r, \zeta))$, we pick $\omega\in \cD_j(K, n, r, \zeta)$, consider the corresponding $E\in I$, and notice that (due to the resolvent bound) $E$ is close to an eigenvalue of the Dirichlet--Neumann truncation,  that is,
		\begin{equation}\lb{416n}
		|E-E_0|\leq e^{-K^2}\text{\ for some\ }E_0\in\sigma(H^{\zeta+n}_{\omega}).
		\end{equation}
		 Combining \eqref{322nn}, \eqref{44n}, \eqref{416n}, and choosing $N$ (hence $K$) sufficiently large we obtain
		\begin{equation}
		F_{jK}(T^{\zeta+r}\omega, E_0)\leq L(E_0)-\frac{\varepsilon}{2},
		\end{equation}
		whenever $\omega\in \cD_j(K, n, r, \zeta)$ and $E_0=E_0(\omega_1, ..., \omega_{\zeta+n})$ is as in \eqref{416n}.
		In other words
		\begin{equation}
		\cD_j(K, n, r, \zeta)\subset \hatt \cD_j(K, n, r, \zeta),
		\end{equation}
		where
		\begin{equation}\no
		\hatt \cD_j(K, n, r, \zeta)
		:=
		\bigcup_{E_0\in \sigma(H^{\zeta+n}_{\omega})\cap\hatt I} \set{
		\omega\in\Omega\, : \,\frac{\varepsilon}{2}\leq L(E_0)- F_{jK}(T^{\zeta+r}\omega, E_0)
		},
		\end{equation}
		where $\hatt I:=[\min I -1, \max I +1]$.  We note that $H_{\omega}^{\zeta+n}$ and the standard Dirichlet Laplacian $H_{D}^{ \zeta+n}$  on $(t_0, t_{\zeta+n})$ are self-adjoint extensions of a symmetric (minimal) operator with deficiency indices $(2(\zeta+n), 2(\zeta+n))$, cf. \cite[Section 2.1]{BLS}. Then the spectral shift for these two operators is at most $2(\zeta+n)$, see \cite[Lemma 9.3.2 p.214, Theorem 9.3.3, p. 215]{BirSol}. Combining this with an explicit computation of eigenvalues of $H_{D}^{ \zeta+n}$ we get
		\begin{equation}
		\#\left(\sigma(H^{\zeta+n}_{\omega})\cap \hatt I \right)\leq C|\hatt I|(n+\zeta),
		\end{equation}
		where $C>0$ is a universal constant (we recall from \eqref{ts} that $\ell^-(\zeta+n) \leq |t_{\zeta+n}|\leq \ell^+(\zeta+n)$). Then using \eqref{42n} and $[0,\zeta+n]\cap[\zeta+r, \zeta+r+jK]=\emptyset$, we estimate
		\begin{equation}\lb{420n}
		\mu(\hatt \cD_j(K, n, r, \zeta))\leq C(n+\zeta)e^{-\eta K}\leq C( K^9+e^{\sqrt{K}})e^{-\eta K}\leq Ce^{-\eta_1 K},
		\end{equation}
		for some $\eta_1=\eta_1(\varepsilon)>0$. Clearly, one has
		\begin{equation}
		\mu(\cD_N(\varepsilon))\leq \sum_{K, n, r,\zeta, j\text{\ as in\ }\eqref{412n}}\mu(\hatt \cD_j(K, n, r, \zeta)).
		\end{equation}
		Then for a fixed $K$, the summation with respect to $n, r$ introduces a subexponential number of terms bounded by $e^{-\eta_1 K}$, and summation with respect to $\zeta$ introduces no more than $\lceil e^{\sqrt{K}}\rceil$ terms bounded by $e^{-\eta_1 K}$ (the precise calculation is carried out in the proof of \cite[Proposition~6.1]{BuDaFi}). Thus \eqref{410m} holds as asserted, which together with the Borel--Cantelli lemma yields \eqref{411m}.
	\end{proof}
	
		Let us recall the Avalanche Principle employed in the proof of Theorem \ref{main1}.
	\begin{lemma}[Avalanche Principle]
		\label{l.avlanche-principle}
		Let $A^{(1)},\ldots, A^{(n)}$ be a finite sequence in $\mathrm{SL}(2,\R)$ satisfying the following conditions:
		\begin{align}\label{condition-AP}
		&\min_{1\le j\le n}\|A^{(j)}\|\ge \l > n,\\ \label{condition-AP2}
		&\max_{1\le j<n}\left|\log \|A^{(j+1)}\|+\log \|A^{(j)}\|-\log\|A^{(j+1)}A^{(j)}\|\right|<\frac12\log\l.
		\end{align}
		Then for some absolute constant $C>0$ one has
		\begin{equation} \label{avalanche-principle}
		\left|\log\|A^{(n)}\ldots A^{(1)}\|+\sum_{j=2}^{n-1}\log\|A^{(j)}\|-\sum_{j=1}^{n-1}\log\|A^{(j+1)}A^{(j)}\|\right|\le C\frac{n}{\l}.
		\end{equation}
	\end{lemma}
	\noindent See \cite[Proposition 2.2]{GS01}  for a proof of Lemma~\ref{l.avlanche-principle}.
	
	In order to streamline notation, we use the shorthand $t_n$ for the point $t_{\omega}(n)$.
	\begin{theorem}\lb{main1}
		There exist a discrete set $\mathfrak{D}\subset \bbR$ and a set $\wti \Omega\subset\Omega$ with $\mu(\wti \Omega)=1$ such that for every compact interval $I\subset\bbR\setminus \mathfrak{D}$ and every $\omega\in\wti\Omega$ the following assertions hold:
		\begin{enumerate}
		\item[{\rm(i)}] For every generalized eigenvalue $E\in I$ of the operator $H_{\omega}$, one has
		\begin{equation}\lb{a51}
		\lim\limits_{n\rightarrow\infty}\frac{1}{n}\log \|M_n^E(\omega)\|=L(E).
		\end{equation}
		\smallskip
		
		\item[{\rm(ii)}] The spectral subspace $\ran(\chi_I(H_{\omega}))$ admits a basis of exponentially decaying eigenfunctions.
		\smallskip
		
		\item[{\rm(iii)}] Given $\delta\in(0,1)$ and a normalized eigenfunction
		\begin{equation}
f\in\ker(H_{\omega}-E)\setminus\set{0}, E\in I, \|f\|_{L^2(\bbR_+)}=1,
		\end{equation}
		there exist $\zeta=\zeta(f)\in\bbN$, $C_{\omega, \delta}>0$, $C_{\delta}>0$ such that\footnote{Recall that $L$ and $\wti{L}$ are related via \eqref{e.2cocycles}.}
		\begin{equation}\lb{426}
		|f(x^+)|\leq C_{\omega,\delta}e^{C_{\delta}\log^{C}(\zeta+1)}e^{-(1-\delta)\wti{L}(E)|x-\zeta|},\  x\geq0,
		\end{equation}
		for an absolute constant $C>0$.
	\end{enumerate}
	\end{theorem}
	\begin{proof}
		We will show that the statement of the theorem holds with $\mathfrak{D}$ as in Theorem~\ref{thm34} and
		\begin{equation}
		\wti\Omega:= \bigcap_{\varepsilon\in(0,\tau)\cap\bbQ}\Omega_1(\varepsilon)\cap\Omega_2(\varepsilon)\cap\Omega_3(\varepsilon),\ \tau:= \frac{1}{3} \min_{E \in I} L(E),
		\end{equation}
		where $\Omega_{1,2,3}(\varepsilon)$ are defined in Theorem~\ref{thm41} (iii), (iv) and in Theorem~\ref{dr} respectively. Note that $\tau > 0$ by Theorem~\ref{thm34}.
		
		{\it Proof of  Part (i)}. Due to \eqref{46new}, it is enough to prove that for a given $\omega\in\wti\Omega$  and for a generalized eigenvalue $E=E_{\omega}\in I$  (which are henceforth fixed) one has
		\begin{equation}\lb{416new}
		\liminf\limits_{n\rightarrow\infty}\frac{1}{n}\log \|M_n^{E}(\omega)\|\geq L(E).
		\end{equation}
		Let $u$ be the generalized eigenfunction of $H_{\omega}$ corresponding to $E$,  that is,
		\begin{align}
		\begin{split}&-u''=Eu, u(0^{+})=0,\text{$u$ satisfies \eqref{vcon} for all\ }j>0,\\
		&\max\set{|u'(t_n^{\pm})|,|u(t_n^{\pm})|}\leq C_u(1+n), n\in\bbZ_+, \text{\ for some\ } C_u>0.\lb{417new}
		\end{split}
		\end{align}
		Our goal is to show that for a given $\varepsilon\in(0,\tau)$ and for all sufficiently large $K$ one has
		\begin{equation}\lb{418}
		\frac{1}{n}\log \|M_n^{E}(\omega)\|\geq L(E)-6\varepsilon,\text{\ for all $n\in[K^{11}+K^{10}, \overline{K}]$. }
		\end{equation}
		Since these intervals cover a half-line, \eqref{418} yields \eqref{416new}.
		
		For a given\footnote{in the sequel $\zeta$ will be determined by the center of localization} $\zeta\in\bbZ_+$  let
		\begin{equation}\lb{349nn}
		K(N):=\max\set{N,  n_1, n_2, n_3, \lceil\log^2(\zeta+1)\rceil},
		\end{equation}
		where $N\in\bbN$ is to be determined\footnote{$N$ will depend on $u$ through $C_u$. In particular, if all generalized eigenfunctions are uniformly bounded, $N$ is $u$-independent.}, $n_1, n_2$ are as in Theorem~\ref{thm41} (iii), (iv) correspondingly, and $n_3=n_3(\omega, \varepsilon)$ is the smallest integer for which
		\begin{equation}\lb{431}
		\omega\in \bigcap_{i\geq n_3}\big(\Omega \setminus \cD_i(\varepsilon)\big).
		\end{equation}
		\begin{step}
	There exists $N=N(C_u)>0$ such that for all $K\geq K(N)$ there exists an integer $m\in[0, \zeta+K^9]$ such that
	\begin{equation}\lb{350nn}
	|u({t_{m}^-})|\leq e^{-2K^2},\ |u'({t_{m}^-})|\leq e^{-2K^2}.
	\end{equation}
\end{step}
\begin{proof}
	First we note that  \eqref{413new} with $n=K^3$ yields
	\begin{equation}\lb{421}
	L(E) - \frac{  \log\|M^E_{K^3}(T^{\zeta+sK^3}\omega)\|}{K^3} < \varepsilon,
	\end{equation}
	or, equivalently,
	\begin{equation}\lb{425}
	\exp((L(E) -\varepsilon)K^3)<  \|M^E_{K^3}(T^{\zeta+sK^3}\omega)\|,
	\end{equation}
	for some $s\in[0, K^6-1]\cap \bbZ_+$ . Focusing on the $s-$th block we introduce the following notation
	\begin{equation}
	[\alpha, \beta]:=[\zeta+sK^3, \zeta+(s+1)K^3],\ m:= \lfloor\frac{\alpha+ \beta}{2}\rfloor.
	\end{equation}
	Our argument is based on a representation of $u$ in terms of its boundary values $u(t_\alpha^+)$, $u(t_\beta^-)$ and special solutions $\psi_{\pm}$ satisfying certain boundary conditions.  The choice of the boundary conditions, hence the representation of $u$, depends on the entry of the matrix
	\begin{equation}\lb{356nn}
	S^{-1}(q_{\beta})D^{-1}(b_{\beta})M^E_{K^3}(T^{\alpha}\omega)
	\end{equation}
	that dominates its norm.  Specifically,  letting $m_{ij}$ denote the $ij$th entry of \eqref{356nn} and assuming that $\psi_{\pm}$ satisfy  $-\psi_{\pm}''=E\psi_{\pm}$, the interior vertex conditions in the interval $[\alpha,\beta]$, and the boundary conditions indicated below, we consider the following four cases.
	
	\noindent \textbf{Case 1.} If $\|S^{-1}(q_{\beta})D^{-1}(b_{\beta})M^E_{K^3}(T^{\alpha}\omega)\|\leq 4 |m_{11}|$ then we let
	\begin{equation}
	\psi_-({{t_{\alpha}^+}})=1,\ \psi'_-({{t_{\alpha}^+}})=0, \psi_+({{t_{\beta}^-}})=0,\ \psi'_+({{t_{\beta}^-}})=1,
	\end{equation}
	and observe that
	\begin{equation}\lb{437n}
	|W(\psi_+,\psi_-)|=|\psi_+'({{t_{\alpha}^+}})|=|\psi_-({{t_{\beta}^-}})|=|m_{11}|>0.
	\end{equation}
	In particular, \eqref{437n} shows that $\psi_-$ and $\psi_+$ are linearly independent, which shows that we may represent
	\begin{equation}\lb{436n}
	 u({t_{m}^-})=u'({{t_{\alpha}^+}})\frac{\psi_+({t_{m}^-})}{\psi_+'({{t_{\alpha}^+}})}+u({{t_{\beta}^-}})\frac{\psi_-({t_{m}^-})}{\psi_-({{t_{\beta}^-}})}.
	\end{equation}
	
	\noindent \textbf{Case 2.} If  $\|S^{-1}(q_{\beta})D^{-1}(b_{\beta})M^E_{K^3}(T^{\alpha}\omega)\|\leq 4 |m_{12}|$ then
	\begin{align}
	&\psi_-({{t_{\alpha}^+}})=0,\ \psi'_-({{t_{\alpha}^+}})=1, \psi_+({{t_{\beta}^-}})=0,\ \psi'_+({{t_{\beta}^-}})=1,\\
	 &u({t_{m}^-})=u({{t_{\alpha}^+}})\frac{\psi_+({t_{m}^-})}{\psi_+({{t_{\alpha}^+}})}+u({{t_{\beta}^-}})\frac{\psi_-({t_{m}^-})}{\psi_-({{t_{\beta}^-}})},\lb{439new}\\
	&|W(\psi_+,\psi_-)|=|\psi_+({{t_{\alpha}^+}})|=|\psi_-({{t_{\beta}^-}})|=|m_{12}|>0.
	\end{align}
	
	\noindent \textbf{Case 3.} If  $\|S^{-1}(q_{\beta})D^{-1}(b_{\beta})M^E_{K^3}(T^{\alpha}\omega)\|\leq 4 |m_{21}|$ then
	\begin{align}
	&\psi_-({{t_{\alpha}^+}})=1,\ \psi'_-({{t_{\alpha}^+}})=0, \psi_+({{t_{\beta}^-}})=1,\ \psi'_+({{t_{\beta}^-}})=0,\\
	 &u({t_{m}^-})=u'({{t_{\alpha}^+}})\frac{\psi_+({t_{m}^-})}{\psi'_+({{t_{\alpha}^+}})}+u'({{t_{\beta}^-}})\frac{\psi_-({t_{m}^-})}{\psi'_-({{t_{\beta}^-}})},\lb{442new}\\
	&|W(\psi_+,\psi_-)|=|\psi'_+({{t_{\alpha}^+}})|=|\psi'_-({{t_{\beta}^-}})|=|m_{21}|>0.
	\end{align}
	
	\noindent \textbf{Case 4.} If  $\|S^{-1}(q_{\beta})D^{-1}(b_{\beta})M^E_{K^3}(T^{\alpha}\omega)\|\leq 4 |m_{22}|$ then
	\begin{align}
	&\psi_-({{t_{\alpha}^+}})=0,\ \psi'_-({{t_{\alpha}^+}})=1, \psi_+({{t_{\beta}^-}})=1,\ \psi'_+({{t_{\beta}^-}})=0,\\
	 &u({t_{m}^-})=u({{t_{\alpha}^+}})\frac{\psi_+({t_{m}^-})}{\psi_+({{t_{\alpha}^+}})}+u'({{t_{\beta}^-}})\frac{\psi_-({t_{m}^-})}{\psi'_-({{t_{\beta}^-}})},\lb{445new}\\
	&|W(\psi_+,\psi_-)|=|\psi_+({{t_{\alpha}^+}})|=|\psi'_-({{t_{\beta}^-}})|=|m_{22}|>0.\lb{366nn}
	\end{align}
	We proceed with Case 1; the other three cases can be handled similarly.  Let us estimate each term in the right-hand side of \eqref{436n}. Combining \eqref{425} and \eqref{437n}, we get
	\begin{align}
	\begin{split}\lb{447n}
	 |\psi_+'({{t_{\alpha}^+}})|=|\psi_-({{t_{\beta}^-}})|=|m_{11}|&\geq\frac{\|S^{-1}(q_{\beta})D^{-1}(b_{\beta})M^E_{K^3}(T^{\alpha}\omega)\|}{4}\\
	&\geq \frac{ \|M^E_{K^3}(T^{\alpha}\omega)\|}{4\|D(b_{\beta})S(q_{\beta})\|}\\
	&\geq c(b^{\pm}, q^{\pm})\exp((L(E) -\varepsilon)K^3),
	\end{split}
	\end{align}
	for some $c(b^{\pm}, q^{\pm})>0$. By \eqref{417new} we get
	\begin{equation}\lb{425new}
	\max\set{|u'({{t_{\alpha}^+}})|, |u({{t_{\beta}^-}})}|\leq C_u(\beta+1)\leq C_u(K^9+e^{\sqrt {K}}).
	\end{equation}
	 Employing \eqref{45new} with $n=\lfloor \frac{K^3}{2}\rfloor$, $\zeta_0=\zeta+sK^3$, and choosing $N$ so that $\lfloor \frac{K^3}{2}\rfloor\geq\log^2(\zeta+sK^3)$ we obtain
	\begin{align}
	\begin{split}
	|{\psi_-({t_{m}^-})}|&\leq \left|\left\langle \begin{bmatrix}1 \\ 0 \end{bmatrix} , S^{-1}(q_{m})D^{-1}(b_{m})M_{\lfloor \frac{K^3}{2}\rfloor}^E(T^{\zeta+sK^3}\omega)\begin{bmatrix} 1 \\ 0\end{bmatrix} \right \rangle\right|\\
	&\leq C(b^{\pm}, q^{\pm})\exp\left(\frac{(L(E)+\varepsilon) K^3}{2}\right),\lb{449new}
	\end{split},
	\end{align}
	for some $C(b^{\pm}, q^{\pm})>0$. Similarly for $N$ so large that $\lfloor \frac{K^3}{2}\rfloor\geq\log^2(\zeta+sK^3+\frac{K^3}{2})$ we obtain
	\begin{align}\lb{427}
	\begin{split}
	\left|\psi_+({t_{m}^-})\right|\leq C(b^{\pm}, q^{\pm}) \exp\left(\frac{(L(E)+\varepsilon)K^3}{2}\right),\ C(b^{\pm}, q^{\pm})>0.
	\end{split}
	\end{align}
	Combining \eqref{436n}, \eqref{447n}--\eqref{427} one obtains
	\begin{align}\lb{428}
	|u({t_{m}^-})|\leq 2C_uC(b^{\pm}, q^{\pm})(K^9+e^{\sqrt {K}})\exp\left(\frac{-L(E)K^3+3\varepsilon K^3}{2}\right)\leq e^{-2K^2},
	\end{align}
	where the last inequality holds whenever $N=N(C_u)$ is large enough and $C(b^{\pm}, q^{\pm})>0$. Replacing $u({t_{m}^-})$ by $u'({t_{m}^-})$, $\psi_{\pm}({t_{m}^-})$ by $\psi'_{\pm}({t_{m}^-})$ in \eqref{436n}, and $[1,0]^{\top}$ by $[0,1]^{\top}$ in \eqref{449new}, \eqref{427} we obtain
	\begin{equation}\lb{452n}
	|u'({t_{m}^-})|\leq e^{-2K^2}.
	\end{equation}
\end{proof}
\begin{step}
	Suppose that $|u(\tau)|=1$ for some $\tau\in\bbR_+$, let $\zeta$ be the largest integer such that $t_{\zeta}\leq \tau$, and recall $m\in[0, \zeta+K^9]$  from Step~1 for such $\zeta$. Then
	\begin{equation}\lb{3.73newnew}
	\|(H^{m}_{\omega}-E)^{-1}\|_{\cB(L^2(t_0, t_{m}))}\geq e^{K^2}.
	\end{equation}
\end{step}
\begin{proof}
	It suffices to show that
	\begin{equation}\lb{440}
	|G^E_{\omega, [0, t_m]}(x, y)|\geq C e^{2K^2},\ (x,y)\in J_1 \times (t_{m}-\delta, t_m),
	\end{equation}
	for some $K-$independent interval $J\subset(t_{\zeta},t_{\zeta+1})$, $K-$independent $\delta>0$, and $C=C(\ell^{\pm}, I)$. Indeed, denoting the characteristic functions of $J, (t_{m}-\delta, t_m)$ by $\chi_1, \chi_2$ respectively, we get
	\begin{equation}
	e^{K^2}\leq \frac{|\langle\chi_1, (H^{m}_{\omega}-E)^{-1}\chi_2\rangle_{L^2(t_0, t_{m})}|}{\|\chi_1\|_{L^2(t_0, t_{m})}\|\chi_2\|_{L^2(t_0, t_{m})}}\leq \|(H^{m}_{\omega}-E)^{-1}\|_{\cB(L^2(t_0, t_{m}))},
	\end{equation}
			for $N$ in \eqref{349nn} sufficiently large (depending only on $C(\ell^{\pm}, I)$).
			To prove \eqref{440} we notice that
			\begin{align}
			\begin{split}
			 u(x)&=u(0^+)\frac{\psi_+(x)}{W(\psi_+,\psi_-)}+u'(t_m^-)\frac{\psi_-(x)}{W(\psi_+,\psi_-)}=u'(t_m^-)\frac{\psi_-(x)}{W(\psi_+,\psi_-)}\\
			&=u'(t_m^-)G^E_{\omega, m}(x, t_m),\ x\in(t_{\zeta},t_{\zeta+1}),\lb{4.53new}
			\end{split}
			\end{align}
		 (this is similar to Case 4 in Step 1 above). By right-continuity of $u$ and $|u(\tau)|=1$  we have
			\begin{equation}
			1/2\leq  |u(x)|,\  x\in J\subset (t_{\zeta}, t_{\zeta+1}),
			\end{equation}
			for some $K-$independent interval $J$.
			Employing \eqref{350nn}  one infers
			\begin{equation}
			1\lesssim |u(x)|=|u'(t_m)| \left|\frac{\psi_-(x)}{W(\psi_+,\psi_-)}\right| \ \leq e^{-2K^2}\left|\frac{\psi_-(x)}{W(\psi_+,\psi_-)}\right|,
			\end{equation}
			for all $x\in  J$. That is,
			\begin{equation}
			e^{2K^2}\lesssim  \left|\frac{\psi_-(x)}{W(\psi_+,\psi_-)}\right|,\ x\in J.
			\end{equation}
			Furthermore,  noticing that
			\begin{equation}\no
			\psi_+(y)=\cos(\sqrt{E}(y-t_m))\geq 1/2\text{\ for all\ }y\in(t_{m}-\delta, t_m],
			\end{equation}
			for some $K-$independent sufficiently small constant $\delta>0$, and using Proposition \ref{prop24} we arrive at
			\begin{equation}
			|G^E_{\omega, [0,t_m]}(x, y)|=\left|\frac{\psi_-(x) \psi_+(y)}{W(\psi_+,\psi_-)}\right|\geq \left|\frac{\psi_-(x)}{2W(\psi_+,\psi_-)}\right|\gtrsim  e^{2K^2},
			\end{equation}
			for all $(x,y)\in J\times (t_{m}-\delta, t_m]$.	 Thus \eqref{440} holds as required.
		\end{proof}
		\begin{step} Let  $\zeta$ be as in Step~2. Then there exists $N=N(C_u)$ such that for all $K\geq K(N)$ and all $n\in[K^{11}+K^{10}, \overline{K}]$ one has
			\begin{equation}\lb{443}
			\frac{1}{n}\log \|M_n^{E}(T^{\zeta}\omega)\|\geq L(E)-5\varepsilon.
			\end{equation}
		\end{step}
		\begin{proof}
			Combining \eqref{431}, \eqref{3.73newnew}  and Theorem~\ref{dr} one infers
			\begin{equation}\lb{443new}
			\frac{1}{mK}\log \|M_{mK}^{E}(T^{\zeta+r}\omega)\|\geq L(E)-\varepsilon,\  r\in[K^{10}, \overline{K}],\ m\in\{1,2\}.
			\end{equation}

			We will use \eqref{443new} to apply the Avalanche principle, see Lemma \ref{l.avlanche-principle}. Concretely, choose $q \in \Z_+$ with $K^{10} \le q \leq K^{-1}\overline{K} - K^9$, define
			\[
			A^{(j)}
			:=
			M_{K}^E(T^{\zeta + K^{10}+(j-1)K} \omega),
			\quad 1 \le j \le q.
			\]
			With $\lambda := \exp(K(L(E) - \e))$, \eqref{443new} gives
			\[
			\|A^{(j)}\|
			\geq
			\lambda
			\geq
			q
			\]
			for all $j$, where the second inequality holds as long as $N$, cf. \eqref{349nn}, is sufficiently large. Since $K \ge \widetilde n_1$ and $K \geq \log^2(|\zeta| + |\overline{K}|+1)$ (enlarge $N$ if necessary), we may use \eqref{45new} to obtain
			\[
			\|A^{(j)}\|
			\le
			\exp\big( K(L(E)+\e)\big),\; 1\le j\le q.
			\]
			Thus,  implies
			\begin{align*}
			&\left| \log\|A^{(j+1)}\| + \log\|A^{(j)}\| - \log\|A^{(j+1)} A^{(j)}\| \right|\\
			& \log\|A^{(j+1)}\| + \log\|A^{(j)}\| - \log\|A^{(j+1)} A^{(j)}\| \\
			& <
			2K(L(E) + \e) - 2K(L(E) - \e)\\
			& =
			4K\e \\
			& \leq
			\frac{1}{2}\log\lambda,
			\end{align*}
			where the final inequality needs $\e$ to be sufficiently small; we note that this smallness condition depends only on $\widetilde \mu$. Thus, taking $\hat N = qK$ and $r_0 =  K^{10}$, we have $\hat N \in [ K^{11}, \overline{K} - K^{10}]$ and the Avalanche Principle (Lemma~\ref{l.avlanche-principle}) yields
			\begin{align*}
			\log\|M_{\hat N}(T^{\zeta+ r_0}\omega)\|
			& =
			\log\|A^{(q)} \cdots A^{(1)}\| \\
			& \geq
			\sum_{j=1}^{q-1} \log\|A^{(j+1)}A^{(j)}\|
			- \sum_{j=2}^{q-1} \log\|A^{(j)}\|
			- C \frac{q}{\lambda} \\
			& \geq
			(q-1)2K(L(E) - \e) - (q-2)K(L(E) + \e) - C \\
			& \geq
			\hat N (L(E) - 4\e)
			\end{align*}
			again, by choosing $N$ large.
			
			Putting this together, we can control $\|M_n^E(T^\zeta\omega)\|$ for general $K^{11} + K^{10} \leq n \leq \bar K$ by interpolation. In particular, writing $n = qK + p$ with $0 \leq p < K$ and $q\ge K^{10} + K^9$, we have
			\begin{equation} \label{eq:largeLE:on:largescale}
			\begin{split}
			\|M_n^E(T^\zeta \omega)\|
			& \geq
			\frac{\|M_{n-K^{10}}(T^{\zeta + K^{10}}\omega)  \|}{\| M_{K^{10}}(T^\zeta \omega) \|} \\
			& \geq
			\rho^{-K^{10}-p} \|M_{qK-K^{10}}(T^{\zeta + K^{10}}\omega)\| \\
			& \geq
			\rho^{-K^{10} - p} e^{(qK -K^{10})(L(E) -4\e) } \\
			& \geq e^{n (L(E) - 5\e)},
			\end{split}
			\end{equation}
			as long as $N$ is sufficiently large (recall $\rho$ from \eqref{rho}).
		\end{proof}
		Picking $\tau\in(t_0, t_1)$ such that $u(\tau)\not=0$, replacing $u$ by  $\frac{u}{u(\tau)}$, and using \eqref{443} one infers \eqref{418} which in turn yields \eqref{416new} and \eqref{a51}.
		
		{\it Proof of  Part (ii).}  By {\it Part (i)} and Ruelle's deterministic version of Oseledec' Theorem \cite{Ruelle1979,Osc}, every generalized eigenvalue is, in fact, an eigenvalue corresponding to an exponentially decaying eigenfunction. Furthermore, since the spectral measure of $H_{\omega}\chi_I(H_{\omega})$ is supported by the generalized eigenvalues belonging to $I$, cf. \cite[Theorem~C.17]{HiPo}, one infers that $\ran(\chi_I(H_{\omega}))$ admits a basis of exponential decaying eigenfunctions.
		
		{\it Proof of  Part (iii).}  First, we notice that
		\begin{align}
		\begin{split}\lb{464n}
		&\max\set{\|f\|_{L^{\infty}(t_j,t_{j+1})}, \|f'\|_{L^{\infty}(t_j,t_{j+1})}} \\
		&\quad\leq c(\ell^-,\ell^+) ({\|f\|_{L^2(t_j,t_{j+1})}}+\|f''\|_{L^2(t_j,t_{j+1})}) \\
		&\quad \leq c(\ell^-,\ell^+, I)\|f\|_{L^2(\bbR_+)}=c(\ell^-,\ell^+, I),
		\end{split}
		\end{align}
		and
		\begin{align}
		\begin{split}\lb{377new}
		&\|f'\|_{L^{\infty}(t_{j}, t_{j+1})}\leq C(\ell^-,\ell^+)(\|f \|_{L^{2}(t_{j}, t_{j+1})}+\|f''\|_{L^{2}(t_{j}, t_{j+1})})\\
		&\qquad \leq C(\ell^-,\ell^+, I)\|f \|_{L^{2}(t_{j}, t_{j+1})}\leq  C(\ell^-,\ell^+, I)\|f\|_{L^{\infty}(t_{j}, t_{j+1})},
		\end{split}
		\end{align}
		for some $C(\ell^-,\ell^+, I)>0$, and all $j\in\bbZ_+$ cf., e.g,  \cite[Corollary~4.2.10]{Bu}, \cite[IV.1.2]{K80}. In addition we remark that $f$ attains its maximum since
		\begin{align}
		\set{\begin{bmatrix}
		f(t_j^+)\\
		f'(t_j^+)
		\end{bmatrix}}_{j=0}^{\infty}\in \ell^2(\bbZ_+,\C^2)\text{\ and thus\ }\lim\limits_{t\rightarrow\infty}(|f(t)|+|f'(t)|)=0.
		\end{align}
		Therefore, we may repeat the arguments of the proof of Part ~(i)  with
		\begin{equation}
		\begin{split}\lb{467n}
		&u=\frac{f}{\|f\|_{L^{\infty}(\bbR_+)}},\ C_u=  \max\set{1, C(\ell^-,\ell^+, I)}\,\text{\ in Step 1},\\
		&\tau=\text{argmax}{|f|} \; \text{(i.e.\ $\tau$ is chosen so that $|f(\tau)| = \|f\|_\infty$) in Step~2},
		\end{split}
		\end{equation}
		where we pick any value of argmax if there is more than one extremum. Then  for a given $\varepsilon\in(0,\tau)$ there exists $N=N(\varepsilon, \omega)$ (which does not  depend on $f$) such that for all $K\geq K(N, \log^2(\zeta+1))$ and all $n\in[K^{11}+K^{10}, \overline{K}]$ one has
		\begin{equation}
		\frac{1}{n}\log \|M_n^{E}(T^{\zeta}\omega)\|\geq L(E)-6\varepsilon.
		\end{equation}
		Utilizing this with sufficiently small $\varepsilon$ (depending on $\delta$ only) and letting
		\begin{equation}
		\varkappa:= c(b^{\pm}, \ell^{\pm} I)\max\set{1, C(\ell^-,\ell^+, I)},
		\end{equation}
		see \eqref{464n}, \eqref{377new}, we will show that
		\begin{equation}\lb{459n}
		|f(t_{\zeta+n}^+)|\leq \varkappa e^{-(1-\delta)L(E)n}, \text{\ for all\  }n\in\left[\frac{p}{4}, \frac{p-1}{2}\right],
		\end{equation}
		for all $p\in [K^{11}+K^{10}, \overline{K}]$, $K\geq K(N)$. As in Step~1 our subsequent argument relies on a representation of $f$ considered on the interval $[t_\zeta, t_{\zeta+p}]$ in terms of its boundary values. Our choice of the representation, as before, depends on the entry of
		\begin{equation}
		S^{-1}(q_{\zeta+p})D^{-1}(b_{\zeta+p})M_p^{E}(T^{\zeta}\omega)
		\end{equation}
		that dominates its norm. We will provide the argument assuming that the  maximizing entry is $11$ and note that the other three cases can be treated almost identically.
		
		One has
		\begin{align}\lb{459}
		 \frac{f(t_{\zeta+n}^+)}{M_f}=\frac{f'(t_\zeta^{+})\psi_+(t_{\zeta+n}^+)}{M_f\psi_+'(t_\zeta^+)}+\frac{f(t_{\zeta+p}^-)\psi_-(t_{\zeta+n}^+)}{M_f\psi_-(t_{\zeta+p}^-)},
		\end{align}
		where  $M_f:=\|f \|_{L^{\infty}(\bbR_+)}$, $-\psi_{\pm}''=E\psi_{\pm}$, $\psi_{\pm}$ satisfies the interior vertex conditions in the interval $[t_\zeta, t_{\zeta+p}]$, and
		\begin{align}
		&\psi_-(t_\zeta^+)=1,\ \psi'_-(t_\zeta^+)=0, \psi_+(t_{\zeta+p}^-)=0,\ \psi'_+(t_{\zeta+p}^-)=1,
		\end{align}
		and
		\begin{align}
		\begin{split}
		|W(\psi_+,\psi_-)|&=|\psi_+'(t_\zeta^+)|=|\psi_-(t_{\zeta+p}^-)|\\
		&\geq\frac{\|S^{-1}(q_{\zeta+p})D^{-1}(b_{\zeta+p})M_p^{E}(T^{\zeta}\omega)\|}{4}\\
		&\geq\frac{\|M_p^{E}(T^{\zeta}\omega)\|}{4\|D(b_{\zeta+p})S(q_{\zeta+p})\|}\\
		& \geq c(b^{\pm}, \ell^{\pm})\exp((L(E) -6\varepsilon)p),\lb{463n}
		\end{split}
		\end{align}
		for some $c(b^{\pm}, \ell^{\pm})>0$.
		In order to estimate $\psi_{-}(t_{\zeta+n}^+)$, we rewrite it in terms of the transfer matrices and use \eqref{45new} as follows
		\begin{equation}
		|\psi_{-}(t_{\zeta+n}^+)|=\left|\left\langle \begin{bmatrix}1 \\ 0 \end{bmatrix} , M_{n}^E(T^{\zeta}\omega)\begin{bmatrix} 1 \\ 0\end{bmatrix} \right \rangle\right|\leq \exp((L(E) +\varepsilon)n).
		\end{equation}
		Similarly one can estimate $\psi_{+}(t_{\zeta+n}^+)$.
		Combining this and \eqref{464n}, \eqref{377new}, \eqref{459}, \eqref{463n} we get
		\begin{align}
		\begin{split}
		&|f(t_{\zeta+n}^+)|\leq \varkappa\exp((L(E) +\varepsilon)n-(L(E) -6\varepsilon)p)\\
		&\quad\quad+\varkappa\exp((L(E) +\varepsilon)(p-n)-(L(E) -6\varepsilon)p)\\
		&\quad\leq \varkappa\exp(-(p-n)L(E) +(n+6p)\varepsilon)\\
		&\quad\quad+\varkappa\exp(-nL(E) +(7p-n)\varepsilon)\\
		&\quad\leq 2\varkappa \exp(-nL(E) +8p\varepsilon)\leq 2\varkappa \exp(-nL(E) +32n\varepsilon)\\
		& \quad\leq  2\varkappa e^{-(1-\delta)nL(E)},
		\end{split}
		\end{align}
		to facilitate the last inequality we pick $\varepsilon=\varepsilon(\delta)>0$ sufficiently small (depending only on $\delta$).  Thus
		\begin{equation}\lb{469n}
		|f(t_{\zeta+n}^+)|\leq 2\varkappa e^{-(1-\delta)L(E)n},
		\end{equation}
		for all
		$n\in [\frac{K^{11}+K^{10}}{4}, \frac{\overline{K}-1}{2}]$ and $K\geq K(N)$. Since these intervals cover the half-line $[\frac{K^{11}}{2}, \infty)$ for sufficiently large $N$, the inequality in \eqref{469n} holds for all
		\begin{equation}
		n\geq \frac{K^{11}}{2}
		=
		\frac{1}{2} \max\set{N(\omega,\varepsilon), \log^2(\zeta+1) }^{11}.
		\end{equation}
		Furthermore, estimating $f(t_{\zeta+n}^+)$ for
		\begin{equation}
		n\in[0, 2^{-1}\max\set{N(\omega,\varepsilon), \log^2(\zeta+1)}^{11}]
		\end{equation}
		trivially and changing variables $k=\zeta+n$, we get
		\begin{align}
		\begin{split}\lb{3102n}
		|f(t_k^+)|&\leq 2\varkappa e^{(1-\delta)L(E)\max\set{N(\omega,\varepsilon), \log^2(\zeta+1) }^{11}}e^{-(1-\delta)L(E)(k-\zeta)}\\
		&\leq C_{\omega,\delta}e^{C_{\delta}\log^{22}(\zeta+1)}e^{-(1-\delta)L(E)|k-\zeta|},\ k\geq\zeta.
		\end{split}
		\end{align}
		A similar estimate can be obtained for $k\in[0, \zeta]$: In this case, the Lyapunov behavior \eqref{469n} is observed only for sufficiently large  $\zeta$, in which case \eqref{469n} holds for $k\in[0, \zeta-\frac{K^{11}}{2}]$ (for small $\zeta$, use the trivial bound).
		
		In order to show a version of \eqref{3102n} with $f$ replaced by $f'$, we employ
		
	\begin{align}
	 \frac{f'(t_{\zeta+n}^+)}{M_f}=\frac{f'(t_\zeta^{+})\psi'_+(t_{\zeta+n}^+)}{M_f\psi_+'(t_\zeta^+)}+\frac{f(t_{\zeta+p}^-)\psi'_-(t_{\zeta+n}^+)}{M_f\psi_-(t_{\zeta+p}^-)},
	\end{align}
		and repeat \eqref{463n}--\eqref{3102n}.
		Finally, keeping in mind Remark~\ref{r.2cocycles} and interpolating between the discrete vertices, we infer \eqref{426}.
	\end{proof}
	Having established existence of a basis of semi-uniformly localized eigenfunctions (SULE) we turn to dynamical localization. Our argument stems from the proof of   \cite[Theorem~2.1]{GdB}.

\begin{proof}[Proof of Theorem~\ref{thm48}]
	Our first goal is to  derive an upper bound for the number of centers of localization\footnote{$\zeta$ from \eqref{426} is called the center of localization of $f$} located in a large interval $[0, L]$. Let $\set{\varphi_n}_{n=1}^{\infty}$ be an $L^2(\bbR_+)-$orthonormal basis of exponentially decaying eigenfunctions of the spectral subspace $\ran(\chi_I(H_{\omega}))$; the corresponding eigenvalues are denoted by $E_n\in I$, $n\geq 1$. Then by \eqref{426} with
	\begin{equation}
	\delta:=1/2, \nu:=\min (\min\limits_{E\in I} \wti{L} (E), 1)>0,
	\end{equation}
	we have
	\begin{equation}\lb{474}
	|\varphi_n(x)|\leq C_{\omega}e^{C\log^{22}(\zeta_n+1)}e^{-\frac{\nu{|x-\zeta_n|}}{2}},\  x\geq 0.
	\end{equation}
	We claim that
	\begin{equation}\lb{475new}
	\cN(L):=\#\set{n:  \zeta_n\leq L }
\leq C(\omega, I) L,\
	L\geq L_0,
	\end{equation}
	for sufficiently large $L_0>0$. For $L>0$ let $\chi_{3L}\in\cB(L^2(\bbR_+))$ denote the operator of multiplication by the characteristic function of $[0,3L]$, let $R( H_{\omega})$ denote the resolvent of $H_{\omega}$ at $\lambda=\min\sigma(H_\omega)-1$ and note that $ \|R^2(H_{\omega})\|_{\cB(L^2(\bbR_+))}\leq 1$.  Next we show
	\begin{equation}\lb{475}
	\cN(L)\leq C(\omega, I) \tr(\chi_{3L}R^2( H_{\omega})\chi_{3L}),
	\end{equation}
	for sufficiently large $L$ and some $ C(\omega, I)$. To that end, notice that
	\begin{align}
	\frac{1}{(E_n-\lambda)^2}&=\langle\varphi_n, R^2(H_{\omega})\varphi_n\rangle_{L^2(\bbR_+)}\\
	&=\langle\varphi_n,\chi_{3L} R^2(H_{\omega})\chi_{3L}\varphi_n\rangle_{L^2(\bbR_+)}\\
	&\quad+\langle\varphi_n,\chi_{3L} R^2(H_{\omega})(1-\chi_{3L})\varphi_n\rangle_{L^2(\bbR_+)}\\
	&\quad+\langle\varphi_n,(1-\chi_{3L}) R^2(H_{\omega})\chi_{3L}\varphi_n\rangle_{L^2(\bbR_+)}\lb{476}\\
	&\quad+\langle\varphi_n,(1-\chi_{3L}) R^2(H_{\omega})(1-\chi_{3L})\varphi_n\rangle_{L^2(\bbR_+)}.\lb{477}
	\end{align}
	Assuming that $\zeta_n\leq L$, $E_n\in I$, and  $C\log^{22}(L+1)<\frac{\nu L}{4}$ and using \eqref{474} we obtain
	\begin{equation}
	|\varphi_n(x)|\leq C_{\omega}e^{\frac{\nu L}{4}}e^{-\frac{\nu|x-\zeta_n|}{2}},\  x\geq 0,
	\end{equation}
	and
	\begin{align}
	&\langle\varphi_n,\chi_{3L} R^2(H_{\omega})(1-\chi_{3L})\varphi_n\rangle_{L^2(\bbR_+)}\leq	 \left\|(1-\chi_{3L})\varphi_n\right\|_{L^2(\bbR_+)}\no\\
	&\quad\leq  C_{\omega}e^{\frac{\nu L}{4}} \left(\int_{3L}^{\infty}e^{-\nu|x-\zeta_n|}dx\right)^{1/2}\no\\
	&\quad \leq C_{\omega}e^{\frac{\nu L}{4}} e^{\frac{\nu\zeta_n}2} e^{-\frac{3\nu L}{2}}\nu^{-\frac12}\leq  C_{\omega}e^{-\frac {3\nu L}{4} }\nu^{-\frac12} \underset{L\rightarrow\infty}{=}o(1).\no
	\end{align}
	Similar estimates hold for  \eqref{476} and \eqref{477}.  Therefore we have
	\begin{align}
	\tr(\chi_{3L}&R^2( H_{\omega})\chi_{3L})\\
	&\geq \sum_{n:\,\zeta_n\leq L}\langle\varphi_n,\chi_{3L} R^2(H_{\omega})\chi_{3L}\varphi_n\rangle_{L^2(\bbR_+)}\\
	&\geq \sum_{n:\,\zeta_n\leq  L}\left(\frac{1}{(E_n-\lambda)^2}- 3  C_{\omega}e^{-\frac {3\nu L}{4} }\nu^{-\frac12}\right)\\
	&\geq C(I, \omega)\,\#\set{n: \zeta_n \leq L},
	\end{align}
	for some $C(I, \omega)>0$.
	
	Next we estimate the right-hand side of \eqref{475}.  Let us recall that $AB\in \cB_2(L^2(\bbR_+))$ (the space of Hilbert--Schmidt operators on $L^2(\bbR_+)$) and
	\begin{equation}
	\|AB\|_{\cB_2(L^2(\bbR_+))}\lesssim \|A\|_{\cB(L^{\infty}(\bbR_+), L^{2}(\bbR_+))}\|B\|_{\cB(L^{2}(\bbR_+)), L^{\infty}(\bbR_+))},
	\end{equation}
	whenever $A\in\cB(L^{\infty}(\bbR_+), L^{2}(\bbR_+))$, $ B\in\cB(L^{2}(\bbR_+), L^{\infty}(\bbR_+))$. A discussion of this fact together with related references can be found, for instance, in \cite[Section 4.1.11]{St01} and \cite[pp. 418--419]{Stol}. This result is applicable in our case due to \cite[ Lemma~C.12]{HiPo} which asserts that $R(H_{\omega})$ maps (boundedly) $L^2(\bbR_+)$ into $L^{\infty}(\bbR_+)$.  Combining these facts we infer
	\begin{align}
	\begin{split}\lb{486}
	&\tr(\chi_{3L}R^2( H_{\omega})\chi_{3L}) = \|\chi_{3L}R( H_{\omega})\|_{\cB_2(L^2(\bbR_+))}^2\\
	&\quad\leq \left(\sqrt{3L}\|R( H_{\omega})\|_{\cB(L^{2}(\bbR_+), L^{\infty}(\bbR_+))}\right)^2\leq C(\omega)L,
	\end{split}
	\end{align}
	for some $C(\omega)>0$. Then \eqref{475} and \eqref{486} yield \eqref{475new}.
	
	Next, we turn to \eqref{3.102}. For brevity, denote $\gamma:=22+\varepsilon$ and let $\kappa>0$ be such that
	\begin{equation}\lb{3.121}
	|\log^{\gamma}(x+\kappa)-\log^{\gamma}(y+\kappa)|\leq \frac{\nu|x-y|}{4},\ x,y> 0.
	\end{equation}
	Then we have
	\begin{align}
	&\left\||X|^p\chi_I(H_{\omega}) e^{-itH_{\omega}}\psi \right\|_{L^2(\bbR_+)}\\
	&\quad \leq \sum_{n=1}^{\infty} |\langle \varphi_n, \psi\rangle_{L^2(\bbR_+)}|\,\||X|^p\varphi_n\|_{L^2(\bbR_+)}\lb{4.88n}\\
	&\quad\leq \sum_{n=1}^{\infty} C_{\omega, I}\,e^{2C\log^{22}(\zeta_n+1)}\int_{\bbR_+}|\psi(x)|e^{-\frac{\nu|x-\zeta_n|}{2}}dx\left(\int_{\bbR_+}x^{2p}e^{-{\nu|x-\zeta_n|}}dx\right)^{1/2}\\
	&\quad\leq \sum_{n=1}^{\infty} C_{{\omega}, I, p,\psi}\,e^{2C\log^{22}(\zeta_n+1)}\zeta_n^p\int_{\bbR_+}e^{-\log^{\gamma} (x+\kappa)}e^{-\frac{\nu|x-\zeta_n|}{2}}dx \\
	&\quad\leq  \sum_{n=1}^{\infty} C_{{\omega}, I, p,\psi}\,e^{2C\log^{22}(\zeta_n+1)+p\log(\zeta_n+1)-\log^{\gamma}(\zeta_n+\kappa)}\\
	&\hspace{4cm}\times\int_{\bbR_+}e^{-\log^{\gamma} (x+\kappa)+\log^{\gamma}(\zeta_n+\kappa)-\frac{\nu|x-\zeta_n|}{2}}dx \\
	&\quad\underset{\eqref{3.121}}{\leq} \sum_{n=1}^{\infty}C_{{\omega}, I, p,\psi}\,e^{2C\log^{22}(\zeta_n+1)+p\log(\zeta_n+1)-\log^{\gamma}(\zeta_n+\kappa)}\\
	&\hspace{4cm}\times\int_{\bbR_+}e^{-\frac{\nu|x-\zeta_n|}{4}}dx \\	
	&\quad\leq  \wti C_{{\omega}, I, p,\psi}\sum_{n=1}^{\infty} e^{2C\log^{22}(\zeta_n+1)+p\log(\zeta_n+1)-\log^{\gamma}(\zeta_n+\kappa)}\\
	&\quad\leq   \wti C_{{\omega}, I, p,\psi}\sum_{L=0}^{\infty}\,\sum_{n: \zeta_n=L} e^{2C\log^{22}(\zeta_n+1)+p\log(\zeta_n+1)-\log^{\gamma}(\zeta_n+\kappa)}\\
	&\quad\leq   \wti C_{{\omega}, I, p,\psi}\sum_{L=0}^{\infty}\cN(L)\ e^{2C\log^{22}(L+1)+p\log(L+1)-\log^{22+\varepsilon}(L+\kappa)}<\infty,\lb{4.91n}
	\end{align}
	where we used \eqref{475new} in the last inequality.
\end{proof}

\section{Random Metric Trees}\lb{rcm}
\subsection{The Almost-Sure Spectrum for Continuum Models}\lb{almsure}

 Our first objective is to show that almost surely the spectrum of  $\bbH_{\omega}$ is given by a deterministic set $\Sigma$.

	\begin{theorem}\lb{prop3.1} There exists a full $\mu$-measure set $\hatt\Omega\subset\Omega$ such that
		\begin{align}\lb{3.4}
		\sigma(\bbH_{\omega} )=\Sigma:= \overline{\bigcup_{(b, \ell, q) \textup{ periodic}} \sigma( \bbH(b,\ell,q ))},\, \omega\in\hatt\Omega.
		\end{align}
	\end{theorem}
	\begin{proof} Since
		\begin{equation}\lb{35nn}
		\sigma( \bbH(b,\ell,q ))=\overline{\bigcup_{k\in\bbZ_+} \sigma(H(T^kb,T^k\ell,T^kq))},
		\end{equation}
		one has
		\begin{equation}
		\sigma(\bbH_{\omega})=\overline{\bigcup_{k\in\bbZ_+} \sigma(H_{T^k\omega})};\ \Sigma=\overline{\bigcup_{(b, \ell, q)\ \text{ periodic}} \sigma(H(b,\ell,q))}.
		\end{equation}
		First, we will first show that
		\begin{align}\lb{34}
		\sigma(H_{\omega} )\subset \Sigma,\text{\ for all\ } \omega\in\Omega,
		\end{align}
		and therefore $\sigma( \bbH_{\omega})\subset \Sigma$.
		Let us fix $\omega\in\Omega$. Seeking a contradiction, we pick $E\in \sigma(H_{\omega})\setminus \Sigma$. Then there exist
		\begin{equation}
		\set{f_k}_{k=1}^{\infty}\subset\dom(H_{\omega})\text{\ and\ } \set{m_k}_{k=1}^{\infty}\subset\bbN,
		\end{equation}
		such that
		\begin{align}
		&\|f_k\|_{L^2(t_0,\infty)}=1,\  \text{supp}(f_k)\subset[t_0,t_{m_k}], \\
		&\sup\limits_{\substack{g\in\dom(\mathfrak h_{\omega})\\ \|g\|_{ \hatt H^1(t_0,\infty)}\leq 1}}  (\mathfrak h_{\omega}-E)[f_k, g]\rightarrow 0,\ k\rightarrow\infty,\lb{3.10}
		\end{align}
		where $\mathfrak h_{\omega}=\mathfrak h(b_{\omega}, \ell_{\omega}, q_{\omega})$, cf. \eqref{2.11new}--\eqref{2.13new} (we recall that $\hatt H^1-$norm is equivalent to the form norm, see \eqref{2.9}). Let $(b^k, \ell^k, q^k)\in\Omega$ denote the $m_k-$periodic sequence  whose first $m_k$ elements are given by $\omega_1,\dots, \omega_{m_k}$. Then  since $E\not\in\Sigma$ one has
		\begin{equation}\lb{3.11}
		C:=\sup\limits_{k\in\bbN}\|F_k\|_{\hatt H^1(t_0,\infty)}<\infty,\ F_k:=(H(b^k, \ell^k, q^k)-E)^{-1}f_k,
		\end{equation}
		where the first inequality follows from the fact that $F''_k=-EF_k-f_k$ and Sobolev inequalities.
		Suitable truncations of $F_k$ belong to $\dom(\mathfrak h_{\omega})$. Indeed, for $k\in\bbN$, let $\varphi_k\in C_0^{\infty}[t_0,\infty)$ be such that $\supp \varphi_k\subset[t_0, t_{m_k+1}]$, $0\leq \varphi_k(x)\leq 1$, $x\geq t_0$,  and
		\begin{equation}
		\varphi_k(x)=\begin{cases}
		1, & x\in[t_0, t_{m_k}],\\
		0, & x\in [t_{m_k+1},\infty).
		\end{cases}
		\end{equation}
		Then  for all $k\in\bbN$ one has
		\begin{equation}\lb{317new}
		\begin{split}
		&(\varphi_k F_k)\in\dom(\mathfrak h_{\omega})\\
		&\|\varphi_k F_k\|_{\hatt H^1(\bbR_+)}\leq \max\set{1, \|\varphi_k\|_{H^1(t_{m_k}, t_{m_k+1})}}\| F_k\|_{\hatt H^1(\bbR_+)}\lesssim 1,
		\end{split}
		\end{equation}
		where we used $\|\varphi_k F_k\|_{ H^1(t_{m_k}, t_{m_k+1})}\lesssim \|\varphi_k \|_{ H^1(t_{m_k}, t_{m_k+1})}\|F_k \|_{ H^1(t_{m_k}, t_{m_k+1})}$, see \cite[Theorem 4.14]{Gr}. Moreover, one has
		\begin{align}
		\begin{split}\lb{3.14}
		(\mathfrak h_{\omega}&-E)[f_k,  \varphi_kF_k]= \langle \varphi_kF_k, -f_k''-Ef_k\rangle_{L^2(\bbR_+)}\\
		&=\langle F_k, -f_k''-Ef_k\rangle_{L^2(\bbR_+)}\\
		&= \left\langle (H(b^k, \ell^k, q^k)-E)^{-1}f_k, (H(b^k, \ell^k, q^k)-E)f_k\right\rangle_{L^2(\bbR_+)}=1.
		\end{split}
		\end{align}
		Combining \eqref{3.10}, \eqref{317new} and \eqref{3.14} we obtain a contradiction.
		
		Next we show that exists a full $\mu$-measure set $\hatt \Omega\subset\Omega$ such that
		\begin{align}\lb{3.15n}
		\Sigma\subset\sigma(\bbH_{\omega} ),\  \omega\in\hatt\Omega.
		\end{align}
		First of all, we note  that $E\in\sigma(\bbH_{\omega})$ whenever there exist two sequences of natural numbers
		\begin{equation}\lb{t16}
		\{r_k\}_{k=1}^{\infty}\subset\bbN,\ \{m_k\}_{k=1}^{\infty}\subset\bbN,
		\end{equation}
		and a sequence of functions $ \{ f_k\}_{k=1}^{\infty}$ such that $f_k \in\dom(\mathfrak h_{T^{r_k} \omega})$ satisfying
		\begin{equation}\lb{t17}
		\liminf_{k\rightarrow\infty}\|f_k\|_{L^2(t_{\omega}(r_k),\infty)}>0,\  \text{supp}(f_k)\subset[t_{\omega}(r_k),t_{\omega}(r_k+m_k)],\ k\in\bbN,
		\end{equation}
		and
		\begin{align}\lb{t18}
		&\sup\ (\mathfrak h_{T^{r_k}\omega}-E)[f_k, g]\rightarrow 0,\ k\rightarrow\infty,
		\end{align}
		where the supremum is taken over the set
		\begin{equation}
		\{{g\in\dom(\mathfrak h_{T^{r_k}\omega}): \|g\|_{ \hatt H^1(t_{\omega}(r_k),\infty)}\leq 1}\}.
		\end{equation}
		This is due to orthogonal decomposition \eqref{na} and the standard Weyl criterion for $\bbH_{\omega}$. Secondly, there exists $\hatt \Omega\subset \Omega$, $\mu(\hatt \Omega)=1$ such that for arbitrary
		\begin{equation}\lb{319new}
		\omega\in\hatt\Omega,\ (b,\ell,q)\in\supp( \mu),  \{m_k\}_{k=1}^{\infty}\subset \bbN,
		\end{equation}
		there exists a sequence $\{r_k\}_{k=1}^{\infty}$ such that for all  $k\in\bbN$ one has
		\begin{align}\lb{3.14new}
		&b_{\omega}({r_k+i})=b_i \text{\ for all\ }i\in\set{1,..., m_k},\\
		&\max_{1\leq i\leq m_k}|\ell_{\omega}({i+r_k}) - \ell_i|\leq \frac{\sqrt{\ell^{-}}}{k},\lb{3.21nn}\\
		&\max_{1\leq i\leq m_k}|q_{\omega}({i+r_k}) - q_i|\leq \frac{1}{k},\lb{332nn}
		\end{align}
		see, for example,  \cite[Proposition~3.8]{Ki08}. We claim that \eqref{3.15n} holds with this choice of $\hatt \Omega$. Indeed, pick any  periodic sequence $(b,\ell,q)$ and $E\in \sigma(H(b,\ell, q))$.  Then by Proposition~\ref{rem2.3} there exist
		\begin{equation}
		\{\varphi_k\}_{k=1}^{\infty}\subset\dom(\mathfrak h(b,\ell,q)),\  \{m_k\}_{k=1}^{\infty}\subset\bbN,
		\end{equation}
		such that
		\begin{align}
		&\sup\limits_{k\in\bbN}\|\varphi_k\|_{\hatt H^1(t_0,\infty)}<\infty, \|\varphi_k\|_{L^2(t_0, \infty)}=1,\  \supp(\varphi_k)\subset[t_0,t_{m_k}],  \lb{323}\\
		&\sup\limits_{\substack{g\in\dom(\mathfrak h(b,\ell,q))\\ \|g\|_{ \hatt H^1(t_0,\infty)}\leq 1}}  (\mathfrak h(b,\ell,q)-E)[\varphi_k, g]\rightarrow 0,\ k\rightarrow\infty.
		\end{align}In order to produce a singular sequence for $\bbH_{\omega}$ we will rescale  $\varphi_k$
		from $[t_{i-1}, t_{i}]$ to $[t_{\omega}(r_k+i-1), t_{\omega}(r_k+i)]$.
		That is, for every $i, k\in\bbN$ we let
		\begin{align}\lb{328}
		f_k(y)&:= \varphi_k(s_{i,k}^{-1}(y)),\ y\in [t_{\omega}(r_k+i-1), t_{\omega}(r_k+i)],
		\end{align}
		where
		\begin{align}
		s_{i,k}(x)&:=\frac{t_{\omega}(r_k+i)-t_{\omega}(r_k+i-1)}{\ell_{i}}(x-t_{i-1})+ t_{\omega}(r_k+i-1),
		\end{align}
		for $x\in[t_{i-1}, t_{i}]$. Then changing variables one obtains
		\begin{align}\lb{2.23n}
		&\langle f_k', g'\rangle_{L^2\big(t_{\omega}(r_k+i-1), t_{\omega}(r_k+i)\big)}
		={\frac{\ell_{i}}{\ell_{\omega}(r_k+i)}}\ \langle \varphi_k', (g\circ s_{i,k})'\rangle_{L^2( t_{i-1},\   t_{i})},\\
		& \langle f_k, g \rangle_{L^2\big(t_{\omega}(r_k+i-1), t_{\omega}(r_k+i)\big)}
		={\frac{\ell_{\omega}(r_k+i)}{\ell_{i}}}\ \langle \varphi_k, g\circ s_{i,k} \rangle_{L^2( t_{i-1},\   t_{i})},\lb{330}
		\end{align}
		where $g\in \hatt H^1(t_{\omega}(r_k), \infty)$.
		Let us denote
		\begin{equation}
		\wti g_k(x):=(g\circ s_{i,k})(x),\ x\in[t_{i-1}, t_i],\ i\in\bbN, k\in\bbN.
		\end{equation}
		Then using  \eqref{2.23n}, \eqref{330} with $f_k$ replaced by $g$ we note that there exists a constant $C>0$ which does not depend on $k$ such that
		\begin{equation}\lb{3.33}
		\|\wti g_k\|_{\hatt H^1(t_0, t_{m_k})}\leq C \text{\ if\ } \|g\|_{\hatt H^1(t_{\omega}(r_k), \infty)}\leq 1,\ k\in\bbN.
		\end{equation}
		We claim that $\{f_k\}_{k=1}^{\infty}$ is a singular sequence satisfying \eqref{t16}--\eqref{t18}. First, we know that $f_k \in \dom(\gh_{T^{r_k}\omega})$ holds since the vertex conditions displayed in \eqref{213} are scale-invariant. Next,  the  conditions in \eqref{t17} hold due to \eqref{323} and \eqref{330} (with $g = f_k$). In order to check \eqref{t18}, let us fix $k\in\bbN$ and  $g$ with $\|g\|_{ \hatt H^1(t_{\omega}(r_k),\infty)}\leq 1$. Then one has
		\begin{align}
		&|(\mathfrak h_{T^{r_k}\omega}-E)[f_k, g]-(\mathfrak h(b,\ell,q)-E)[\varphi_k, \wti g_k]|\leq \lb{3.35}\\
		&\leq \Big|\sum_{i=1}^{m_k}\left({\frac{\ell_{i}}{\ell_{\omega}(r_k+i)}}-1\right)\ \langle \varphi_k', (g\circ s_{i,k})'\rangle_{L^2( t_{i-1},\   t_{i})}\\
		&\qquad\qquad-E\left(\frac{\ell_{\omega}(r_k+i)}{\ell_{i}}-1\right)\ \langle \varphi_k, g\circ s_{i,k}\rangle_{L^2( t_{i-1},\   t_{i})}\Big|\no\\
		&\quad+\left| \sum_{i=1}^{m_k}(q_i-q_{\omega}(r_k+i))\, \overline{\varphi_k(t_{i}^-)}( g\circ s_{i,k})(t_{i}^-)\right|\\
		&\,\lesssim  \frac{ \|\varphi_k\|_{\hatt H^1(\bbR_+)}\,\| \wti g_k\|_{\hatt H^1(\bbR_+)}}{k}\\
		&\, \lesssim \frac1k\rightarrow 0,\ k\rightarrow\infty.\lb{3.36}
		\end{align}
In the first inequality we employed \eqref{2.23n} and \eqref{330};  in the second one we used the Cauchy--Schwarz inequality, the fact that $|\varphi_k(t_i^-)|\lesssim \|\varphi_k\|_{\hatt H^1(t_{i-1},t_i)}$, \eqref{3.21nn}, and \eqref{332nn}; and finally in the last inequality we used \eqref{323} and \eqref{3.33}. Hence, \eqref{t18} holds and $E\in\sigma(\bbH_{\omega})$ as asserted.
	\end{proof}

\begin{remark}\lb{remark4.2}
It is natural to conjecture that the spectrum for the half-line operator $H_{\omega}$ is a deterministic set given by the union of periodic spectra of $H(b,\ell,q)$. The latter, under some spectral monotonicity assumption, in turn equals the union of constant spectra, which in certain scenarios can be computed explicitly. However, neither standard ergodicity arguments (e.g., proof of Pastur's Theorem) nor spectral theoretical arguments (cf.\ \cite[proof Lemma 1.4.2]{St01} and \cite{KirschMartinelli1982}) seem to be applicable to the {\it half-line} models in question. We note that the half-line models present both probabilistic and spectral-theoretical complications which are not typical for operators on $\bbR$.
\end{remark}
	
\subsection{Proof of Dynamical and Exponential Localization for Metric Trees}

	We say that a function $f: \Gamma_{b, \ell}\rightarrow \bbR$ is {\it tree-exponentially decaying} if there exist $\lambda\geq0$ and $C=C(f, \lambda)>0$ such that
	\begin{equation}
	|f(x)|\leq  \frac{Ce^{-\lambda |x|}}{\sqrt{w_o(|x|)}},
	\end{equation}
	where $w_o(|x|)$ denotes the number of vertices in the same generation as $x$; cf.\ \eqref{eq:wvdef}.
	\begin{proof}[Proof of Theorem~\ref{main2}] (i) By Theorem~\ref{prop3.1} and part (ii) of Theorem~\ref{main1}, there exist full measure sets $\hatt \Omega, \wti\Omega\subset\Omega$ such that
		\begin{equation}
		\sigma(\bbH_{\omega})=\Sigma,\ \sigma_c(H_{\omega})=\emptyset,\ \omega\in\hatt \Omega\cap\wti \Omega,
		\end{equation}
and the operator $H_{\omega}$ enjoys a basis of exponentially decaying eigenfunctions. Then letting
		\begin{equation}\lb{340new}
		\Omega^*:=\bigcap_{n\in\bbZ_+}T^{-n}(\hatt \Omega\cap\wti \Omega),
		\end{equation}
		we notice that $\mu(\Omega^*)=1$ and that
		\begin{align}
		&\sigma(\bbH_{\omega})=\Sigma,\ \sigma_c(\bbH_{\omega})=\overline{\bigcup_{n\in\bbZ_+}\sigma_c(H_{T^n\omega})}=\emptyset,\  \omega\in\Omega^*,
		\end{align}
		where we used the orthogonal decomposition \eqref{na8}. Next we show that $\bbH_{\omega}$ admits a basis of tree-exponentially decaying eigenfunctions almost surely. To that end, let us fix $\omega\in \wti \Omega$,  $v\in\cV\setminus\{o\}$, $\gen(v)=n\in\bbN$, and $1\leq k\leq b_{n-1}$. Then it suffices to construct a basis of tree-exponentially decaying eigenfunctions in $\cL_{v,k}=\cU_{v,k}^{-1}(L^2(t_{\omega}(n),\infty))$, cf. \eqref{24new},  \eqref{25new}. For a basis element $f\in\ker(H_{T^n\omega}-E)$ of $L^2(t_{\omega}(n),\infty)$, we define the corresponding basis element of $\cL_{v,k}$,
		\begin{equation}
		\psi_f:=\cU_{v,k}^{-1}f,\  \psi_f\in\dom(\bbH_{\omega}).
		\end{equation}
		Then \eqref{29nn} yields
		\begin{equation}
		|\psi_f(x)|\leq \frac{C_{f} e^{-\frac{\wti L(E) |x|}{2}}}{\sqrt{w_v(|x|)}}.
		\end{equation}
		A basis of tree-exponentially decaying eigenfunctions of $\cL_o$ can be constructed similarly.
		 \newline
		(ii) Let $v\in\cV$ and $n:=\gen(v)$, then by Part (iii) of Theorem~\ref{main1},  the  subspace $\ran(\chi_I(H_{T^n\omega}))$ is spanned by semi-uniformly localized eigenfunctions
		\begin{equation}\lb{3.48new}
		f_{n,j}\in \ker(H_{T^n\omega}-E_j(n)), j\in\bbZ_+, E_j(n)\in I, n=\gen(v).
		\end{equation}
		For $ 1\leq k\leq b_{n-1}$, $j\in\bbZ_+$ we introduce
		\begin{equation}
		\psi_{v,k, j}:=\cU_{v,k}^{-1}f_{n,j}\in \dom(\bbH_{\omega}),
		\end{equation}
		and notice that
		\begin{equation}\lb{350new}
		\supp(\psi_{v,k, j})\subset T_v,
		\end{equation}
the forward subtree rooted at $v$. Then for $\omega\in\Omega^*$ one has (abbreviating $\Gamma = \Gamma_{b_\omega,\ell_\omega}$):
		\begin{align}
		&\left\||X|^p\chi_I(\bbH_{\omega}) e^{-it\bbH_{\omega}}\chi_{\mathcal{K}}\right\|_{L^2(\Gamma)}\no\\
		&\quad\leq \sum_{v\in\cV}\ \  \sum_{k=1}^{b_v-1}\ \  \sum_{\set{j : \, \substack{ E_j(n)\in I, \\ E_j(n)\text{\ as in\ }\eqref{3.48new}}}} |\langle \psi_{v,k, j}, \chi_{\cK}\rangle_{L^2(\Gamma)}|\,\||X|^p\psi_{v,k, j}\|_{L^2(\Gamma)}\no\\
		&\quad\underset{\eqref{350new}}{\leq} \sum_{v\in\cV, \ T_v \cap \mathcal K \neq \emptyset}\ \  \sum_{\substack{1\leq k\leq b_v-1 \\j: E_j(n)\in I}} |\langle \psi_{v,k, j}, \chi_{\cK}\rangle_{L^2(\Gamma)}|\,\||X|^p\psi_{v,k, j}\|_{L^2(\Gamma)}\no\\
		&\quad\leq \sum_{\substack{v\in\cV, \ T_v \cap \mathcal K \neq \emptyset\\1\leq k\leq b_v-1\\j: E_j(n)\in I}} \int_{\cK\cap T_v}|\psi_{v,k, j}(x)|dx  \ \ \ \, \left(\int_{\Gamma}x^{2p}|\psi_{v,k, j}(x)|^2dx\right)^{1/2} \no\\
		&\quad\leq \sum_{\substack{v\in\cV, \ T_v \cap \mathcal K \neq \emptyset,\\  n=\gen(v)\\1\leq k\leq b_v-1,\\j: E_j(n)\in I}} \int_{|\cK\cap T_v|}(w_v(t))^{1/2}|f_{n,j}(t+|v|)|dt \\
		&\hspace{4cm} \times\left(\int_{\Gamma}|x|^{2p}|\psi_{v,k, j}(x)|^2dx\right)^{1/2} \no\\
		&\quad\leq \sum_{\substack{v\in\cV, \ T_v \cap \mathcal K \neq \emptyset,\\  n=\gen(v)\\j: E_j(n)\in I}} C_{v, K}\int_{|\cK\cap T_v|}|f_{n,j}(t+|v|)|dt \\
		& \hspace{4cm}\times \left(\int_{|v|}^{\infty}|\tau|^{2p}|f_{n, j}(\tau)|^2d\tau\right)^{1/2} \lb{351new},
		\end{align}
		where $|\cK|:=[0, \text{diam}(\cK) ]$. Proceeding as in \eqref{4.88n}--\eqref{4.91n} with $\psi$ replaced by the characteristic function of the interval $[0, \text{diam}(\cK) ]$, we deduce that \eqref{351new} converges as asserted.
	\end{proof}

\begin{remark}\lb{rem42}We notice that all eigenfunctions $\psi_E$ (including those corresponding to energies $E\in\mathfrak D$) satisfy
	\begin{equation}\lb{4.34nn}
	|\psi_E(x)|\leq  \frac{Ce^{-\lambda_E |x|}}{\sqrt{w_o(|x|)}},
	\end{equation}
	for some $\lambda_E\geq0$ and $C>0$, where $w_o(|x|)$ denotes the number of vertices in the same generation as $x$; cf.\ \eqref{eq:wvdef}. Moreover, one has $\lambda_E>0$ whenever $E\not\in\mathfrak D$, in particular, \eqref{4.34nn} yields $\psi_E\in L^2(\Gamma_{b, \ell})$ in this case. Furthermore,  if $E\in\mathfrak D$ and $\lambda_E=0$ then $\psi_E$ still decays exponentially, $|\psi_E(v)|\leq  \frac{C}{2^{|\gen (v)|/2}}$ for all $v\in\cV$. However, this inequality alone is insufficient to deduce $L^2(\Gamma_{b,\ell})$ integrability.  The analogous issue does not arise in the setting of metric graphs for which the volume of the ball centered at the root with radius $r$ grows polynomially as $r\uparrow+\infty$, e.g., as in the metric graph spanned by $\bbZ^d$.
\end{remark}

\part{Anderson Localization for Discrete Radial Trees} \lb{sec:discrete}

\section{Random Discrete Trees} This part of the paper concerns Anderson localization for discrete radial trees.
\begin{hypothesis}\lb{hyp5.1} Let $\Gamma=(\cV,\cE)$ be a rooted, radial discrete tree.  Assume that the branching numbers $b_v\in[b^-,b^+]$, $b^-\geq2$, and the potential $q_v\in[q^-,q^+]$ are radial. Let
	\begin{equation}
	p: \{(u,v)\in \cV^2: d(u,v)=1\}\rightarrow[p^-, p^+],
	\end{equation}	
	be radial, symmetric, and bounded, that is,
	\begin{equation}
	p(u,v)= p_{\min(\gen(u),\gen(v))},\text{\ for\ } u,v\in\cV;
	\end{equation}	
	and $p:=\{ p_n\}_{n=0}^{\infty}\subset[p^-,p^+],$ $p_{-1}=0$, $\ p^{\pm}\in(0,\infty)$.
\end{hypothesis}
Assuming this hypothesis, we introduce a bounded operator $\bbJ(b,p,q)\in\cB(\ell^2(\cV))$ as follows
\begin{equation}\lb{5.1}
(\bbJ({b, p, q}) f)(u):=\sum_{v\sim u} p(u,v)\big(q(u)f(u)-f(v)\big),\ f\in \ell^2(\cV).
\end{equation}

In this part, we adopt the notation of the previous sections with the convention that all edges have length one. Thus, for vertices $x,y \in \cV$, $\dist(x,y)$ is the combinatorial distance between them, and, in particular $|x| = \gen(x)$ for all $x \in \cV$.

\subsection{The Almost-Sure Spectrum for Discrete Models}\lb{almsuredis}

The following hypothesis is assumed throughout this section.
\begin{hypothesis}\lb{hyp61}
	Let $\widetilde \mu$ be a probability measure  with $\supp(\wti \mu)= \mathcal{A}$, $\#\cA\geq 2$, and either
	\begin{equation}\lb{515nn}
	\cA\subseteq
	\{b_-,\ldots,b_+\}\times \{1\} \times[q_-,q_+]
	\end{equation}
	or
	\begin{equation}\lb{516nn}\begin{split}
	\cA&\subseteq
	\{b_-,\ldots,b_+\} \times[p_-,p_+]\times \{0\} \\ &  \text{ and } \exists (b,p,0), (b',q',0) \in \supp\wti\mu \text{ with } p\sqrt{b} \neq p' \sqrt{b'}.
	\end{split}
	\end{equation}
\end{hypothesis}
Let us remark that the secondary hypothesis in \eqref{516nn} is essential, for, if $\supp\wti\mu$ is concentrated on a set for which $q = 0$ and $p\sqrt{b} = \mathrm{const.}$, then the Jacobi matrices arising in the orthogonal decomposition of $\mathbb J_\omega$ will all have constant entries.

We introduce $(\Omega, \mu):=(\mathcal{A}^{\bbZ_+}, \wti\mu^{\bbZ_+})$. For $\omega \in \Omega$, define the operators $\bbJ_\omega:=\bbJ(b_\omega, p_\omega, q_\omega)$ and Jacobi matrices $J_{\omega}:= J(b_{\omega},   p_{\omega}, q_{\omega})$, where
\begin{equation}
\{( b_\omega(n), p_\omega(n), q_\omega(n))\}_{n=0}^{\infty},
\end{equation}
is a sequence of i.i.d. random vectors with common distribution $\wti\mu$. Let us notice that
\begin{equation}
\bbJ_{\omega}=
\begin{cases}
\mathbb S_{\omega}\  (\text{cf}.\ \eqref{int18}), &\text{if \eqref{515nn} holds,}\\
\mathbb A_{\omega}\  (\text{cf}.\ \eqref{int19}), &\text{if \eqref{516nn} holds.}
\end{cases}
\end{equation}
In particular,
\begin{itemize}
	\item Random Branching Model (RBM) arises when  \[\supp\widetilde\mu \subseteq \{b_-,...,b_+\} \times \{1\} \times \{1\}\],
	\item Random Weight Model (RWM) arises when  $\supp\widetilde\mu \subseteq \{d\} \times[p_-,p_+]\times\{0\}$,
	\item Random Schr\"odinger Operator (RSO) arises when $\supp\widetilde\mu \subseteq \{d\}  \times \{1\}\times[q_-,q_+]$.
\end{itemize}
\begin{remark}\lb{rem5.3} We point out that RBM and RSO concern random realizations of the {\it discrete Laplace operator}, while RWM is focused on {\it the adjacency matrices}, i.e. $q\equiv 0$. Typically (e.g., for $\bbZ^d$ models) the distinction between the discrete Laplace operator and the adjacency matrix of the graph is irrelevant as the two operators differ by a scalar multiple of the identity operator. In the setting of non-constant trees, however, the distinction is more subtle since it depends on the branching numbers. What is more, the consecutive transfer matrices for RWM are correlated unless $q\equiv 0$.
\end{remark}

Abusing notation somewhat, we will identify a scalar with a constant sequence consisting of that scalar, for example writing $\mathbb{A}(2,1,0)$ to mean the adjacency operator for which all branching numbers are two and all $p$'s are one.

\begin{theorem}\lb{thm5.5}There exists a full $\mu$-measure set $\hatt \Omega\subset\Omega$ such that
	\begin{align}\lb{6.2}
	\sigma(\bbA_{\omega} )= \Sigma:=\overline{\bigcup_{(b, p)\ \text{ periodic}} \sigma(\bbA(b,p,0))},\, \omega\in\hatt\Omega.
	\end{align}
\end{theorem}
\begin{proof}
	First, we show that
	\begin{align}\lb{65}
	\sigma(\bbA_{\omega} )\subset \Sigma,\text{\ for all\ } \omega\in\Omega.
	\end{align}
	Seeking contradiction, we assume that $E\in\sigma(\bbA_{\omega})\setminus\Sigma$ for some $\omega\in\Omega$. Then there exist
	\begin{equation}
	\{f_k\}_{k=1}^{\infty}\subset\ell^2(\Gamma)\text{\ and\ } \{m_k\}_{k=1}^{\infty}\subset\bbN,
	\end{equation}
	such that
	\begin{align}
	&\|f_k\|_{\ell^2(\Gamma)}=1,\  \text{supp}(f_k)\subset B(o; m_k), \\
	& \|(\bbA_{\omega}-E)f_k\|_{\ell^2(\Gamma)}\rightarrow 0,\ k\rightarrow\infty.\lb{68}
	\end{align}
	where $B(o; m_k)$ denotes the ball centered at $o$ with radius $m_k$.
	The $m_k+2$-periodic sequence with the first $m_k+2$ elements given by $\omega_1,\dots, \omega_{m_{k}+1}$ is denoted by $(b^k, p^k, 0)$. Then  since $E\not\in\Sigma$ one has
	\begin{equation}
	\|(\bbA(b^k, p^k, 0)-E)^{-1}\|_{\cB(\ell^2(\Gamma))}\leq C<\infty,
	\end{equation}
	and thus for all $k$ we get
	\begin{equation}
	\|(\bbA_{\omega}-E)f_k\|_{\ell^2(\Gamma)}= \|(\bbA(b^k, p^k, 0)-E)f_k\|_{\ell^2(\Gamma)}\geq C^{-1}>0,
	\end{equation}
	which contradicts \eqref{68}.

	Next, we show
	\begin{equation}
	\Sigma\subset\sigma(\bbA_{\omega})
	\end{equation}
	for almost all $\omega$. To that end, we first notice that there exists $\hatt\Omega\subset \Omega$, $\mu(\hatt\Omega)=1$ such that for arbitrary
	\begin{equation}\lb{611}
	\omega\in\hatt\Omega,\ (b,p,0)\in\supp( \mu), \text{\ and\ } \{m_k\}_{k=1}^{\infty}\subset \bbN,
	\end{equation}
	there exists a sequence $\{r_k\}_{k=1}^{\infty}$ such that for all  $k\in\bbN$ one has
	\begin{align}
	&b_{\omega}({r_k+i})=b_i \text{\ for all\ }i\in\{0,..., m_k+1\},\lb{bran}\\
	&\max_{0\leq i\leq m_k+1}|p_{\omega}({i+r_k}) - p_i|\underset{k\rightarrow\infty}{=}o(1),\lb{new5}
	\end{align}
	see, for example,  \cite[Proposition~3.8]{Ki08}.  Pick an arbitrary periodic sequence $(b,p,0) \in \supp(\mu)$ and an arbitrary $E\in \sigma(\bbA(b, p,0))$.
	Then there exist $\{\varphi_k\}_{k=1}^{\infty}\subset\ell^2(\Gamma)$ and $\{m_k\}_{k=1}^{\infty}\subset\bbN$ such that
	\begin{align}
	&\|\varphi_k\|_{\ell^2(\Gamma)}=1,\, \supp(\varphi_k)\subset B(o; m_k), \,k\in\bbN,\lb{615} \\
	& \|(\bbA(b,p,q)-E)\varphi_k\|_{\ell^2(\Gamma)}\rightarrow 0,\ k\rightarrow\infty.\lb{616}
	\end{align}
	Given \eqref{611}--\eqref{616} we are ready to produce a Weyl sequence for $\bbA_{\omega}$.
	
	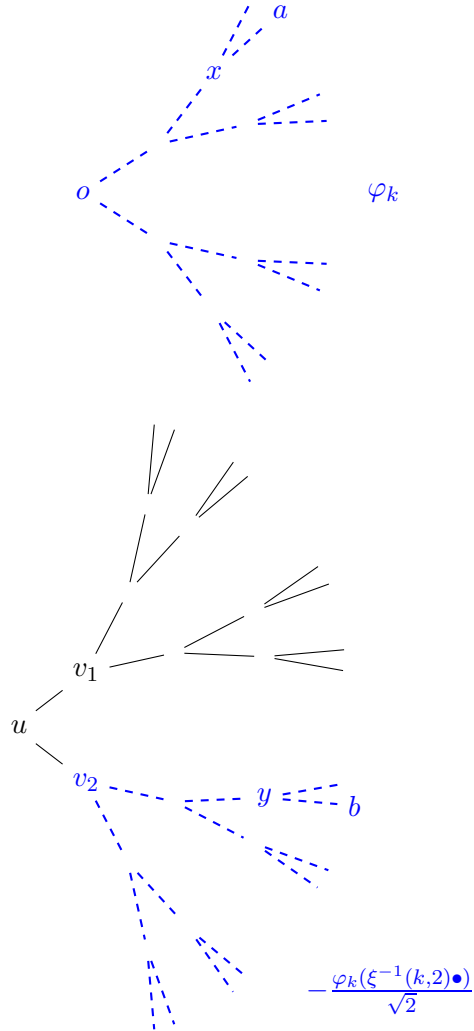
\begin{figure}
		\begin{tikzpicture}[grow cyclic, level distance=12mm]
		\tikzstyle{level 1}=[sibling angle=65]
		\tikzstyle{level 2}=[sibling angle=40]
		\tikzstyle{level 3}=[sibling angle=20]
		\tikzstyle{level 4}=[sibling angle=10]
		\node[blue] {{$o$}}
		child[blue, thick, dashed] {node {}
			child {node {}
				child {node {}}
				child {node {}}}
			child {node {}
				child {node {}}
				child {node {}}}}
		child[blue, thick, dashed] {node {}
			child {node {}
				child {node {}}
				child {node {}}}
			child {node {$x$}
				child {node {$a$}}
				child {node {}}}};
		\draw[blue](4,0)node{{$\varphi_k$}};
		\end{tikzpicture}
		
		\begin{tikzpicture}[grow cyclic, level distance=12mm]
		\tikzstyle{level 1}=[sibling angle=75]
		\tikzstyle{level 2}=[sibling angle=50]
		\tikzstyle{level 3}=[sibling angle=30]
		\tikzstyle{level 4}=[sibling angle=15]
		\node {{ $u$}}
		child{node[blue]{$v_2$} child[blue, thick, dashed] {node {}
				child[blue, thick] {node {}
					child {node {}}
					child {node {}}}
				child[blue, thick] {node {}
					child {node {}}
					child {node {}}}}
			child[blue, thick, dashed] {node{}
				child {node {}
					child {node {}}
					child {node {}}}
				child {node {$y$}
					child {node {$b$}}
					child {node {}}}}	}
		child{node{$v_1$} child {node {}
				child {node {}
					child {node {}}
					child {node {}}}
				child {node {}
					child {node {}}
					child {node {}}}}
			child {node {}
				child {node {}
					child {node {}}
					child {node {}}}
				child {node {}
					child {node {}}
					child {node {}}}}	};
		\draw[blue] (5,-3.5) node {{$-\frac{\varphi_k(\xi^{-1}(k,2)\bullet)}{\sqrt{2}}$}};
		\end{tikzpicture}
		
		\caption{Top panel: $T(k)$. Bottom panel: vertices $v_1, v_2$ in generation $r_k$ with common backward neighbor $u$, $b_{r_k}=2$. Subtree in blue (dashed) is $T(k,2)$. The isomporphism $\xi(k,2)$ maps $T(k)$ onto $T(k,2)$, in particular $o\mapsto v_2$, $x\mapsto y$, $a\mapsto b$, blue(dashed) tree in the top panel gets mapped into the blue(dashed) subtree in the bottom panel.}\lb{pic2}
	\end{figure}
	
	For a fixed $k\in\bbN$, pick two distinct vertices $v_1$, $v_2$ in generation $r_k$ with common backward neighbor $u\in\cV$ (in generation $r_k-1$), see Figure \ref{pic2}.   Then by  \eqref{bran} there exists a pair of graph isomorphisms
	\begin{equation}
	\xi(k, i): B(o;m_k+1)\rightarrow T_{v_i}\cap B(v_i;m_k+1).
	\end{equation}
	We notice that
	\begin{equation}\lb{new1}
	\xi(k, i)(o)=v_i, i=1,2, k\in\bbN.
	\end{equation}
	For brevity, we denote
	\begin{equation}\lb{new10}
	T(k):=\Gamma\cap B(o;m_k+1),\  T(k,i):=T_{v_i}\cap B(v_i;m_k+1).
	\end{equation}
	Let us define
	\begin{equation}\lb{new2}
	(W_k \varphi)(x):=\begin{cases}
	2^{-1/2}\varphi(\xi^{-1}(k,1)x),\ & x\in T(k,1),\\
	-2^{-1/2}\varphi(\xi^{-1}(k,2)x),\ & x\in T(k,2),\\
	0, &\text{otherwise}
	\end{cases}
	\end{equation}
for $\varphi \in \ell^2(\cV)$ which is supported on $B(o,m_k+1)$. We claim that $\{W_k\varphi_k\}_{k\geq 1}$ is a Weyl sequence for $\bbA_{\omega}$, $\omega\in \hatt \Omega$. To that end, let us first notice
	\begin{align}
	&\left|\|(\bbA(b,p,0)-E)\varphi_k\|_{\ell^2(\Gamma)}-\| (\bbA_{\omega}-E)W_k\varphi_k\|_{\ell^2(\Gamma)} \right|\\
	&\quad=\left|\|W_k(\bbA(b,p,0)-E)\varphi_k\|_{\ell^2(\Gamma)}-\| (\bbA_{\omega}-E)W_k\varphi_k\|_{\ell^2(\Gamma)} \right|\\
	&\quad\leq \|W_k(\bbA(b,p,0)-E)\varphi_k-(\bbA_{\omega}-E)W_k\varphi_k\|_{\ell^2(\Gamma)}\\
	&\quad=\|W_k\bbA(b,p,0)\varphi_k-\bbA_{\omega}W_k\varphi_k\|_{\ell^2(\Gamma)},\lb{new4}
	\end{align}
	where we used $\|(\bbA(b,p,0)-E)\varphi_k\|_{\ell^2(\Gamma)}=\|W_k(\bbA(b,p,0)-E)\varphi_k\|_{\ell^2(\Gamma)}$ which follows from the definition of $W_k$. Next, recalling \eqref{new1} and the fact that $u$ is the common backward neighbor of $v_1$, $v_2$ we get
	\begin{align}
	(\bbA_{\omega}W_k\varphi_k)(u)&=p_{\omega}(u, v_1)[W_k\varphi_k](v_1)+p_{\omega}(u, v_2)[W_k\varphi_k](v_2)\\
	&=\frac{p_{\omega}(u, v_1)\varphi_k(o)-p_{\omega}(u, v_2)\varphi_k(o)}{\sqrt2}=0,
	\end{align}
	since $p_{\omega}(u, v_1) = p_{\omega}(u, v_2)$. Further, one has
	\begin{equation}\lb{new3}
	W_k(\bbA(b,p,q)\varphi_k)(u)=0=(\bbA_{\omega}W_k\varphi_k)(u),
	\end{equation}
	where the first equality follows from \eqref{new2}. Next, let us fix $i=1,2$, $k\in\bbN$ and use the shorthand $\xi_k:=\xi(k,i)$. For $y\in T(k,i)$ let $ x:=\xi_k^{-1}(y)$, see Figure \ref{pic2}, then one has
	\begin{align}
	&W_k(\bbA(b,p,0)\varphi_k)(y)-[\bbA_{\omega}(W_k\varphi_k)](y)\lb{newnew8}\\
	&=\frac{1}{\sqrt{2}}(\bbA(b,p,0)\varphi_k)(x)-[\bbA_{\omega}(W_k\varphi_k)](y)\\
	&=-\frac{1}{\sqrt{2}}\sum_{a \sim x} p(x,a)\varphi_k(a)+\sum_{b \sim y} p_{\omega}(y,b)(W_k\varphi_k)(b)\lb{ur14new}\\
	&=-\frac{1}{\sqrt{2}}\Big(\sum_{a \sim x} p(x,a)\varphi_k(a) - \sum_{b \sim \xi_k(x)} p_{\omega}(\xi_k(x),b)  \varphi_k(\xi_k^{-1} b)\Big)\lb{ur14}.
	\end{align}
	Let us point out that $\xi_k^{-1}(b)$ is not defined if $b\not\in T_{v_1}\cup T_{v_2}$. However, one does have $W_k\varphi_k(b)=0$ and therefore the equality in \eqref{ur14} holds with
	\begin{equation}\lb{ur523n}
	\varphi_k(\xi_k^{-1}(b)):=W_k\varphi_k(b)=0.
	\end{equation}
	Moreover, combining this and \eqref{bran} we obtain
	\begin{align}
	\begin{split}\lb{ur523}
	&\sum_{a \sim x}p(x,a)\varphi_k(a)-\sum_{b \sim y}p_{\omega}(y,b)\varphi_k(\xi_k^{-1}(b))\\
	&\qquad=\sum_{a \sim x}(p(x,a)-p_{\omega}(\xi_k(x),\xi_k(a)))\varphi_k(a).
	\end{split}
	\end{align}

	Given \eqref{ur523n} and \eqref{ur523} we are ready to continue  \eqref{newnew8}--\eqref{ur14}. Changing variables via $b=\xi_k(a)$, we get
	\begin{align}
	\begin{split}
	&W_k(\bbA(b,p,0)\varphi_k)(y)-[\bbA_{\omega}(W_k\varphi_k)](y)\lb{new8}\\
	&=-\frac{1}{\sqrt{2}}\Big(\sum_{a \sim x}[ p(x,a)-  p_{\omega}(\xi_k(y),\xi_k(a))] \varphi_k(a)\Big),
	\end{split}
	\end{align}
	where we made a change of variable $b=\xi_k(a)$. Furthermore we note that \eqref{new8} holds for $y\in\Gamma\setminus(T(k,1)\cup T(k,2))$ trivially, i.e., both sides are equal to zero. Recalling $T(k)$ from \eqref{new10} and using \eqref{new5} yield
	\begin{equation}\lb{new9}
	c(k):=\max\limits_{x\in T(k), x \sim a}| p(x,a)-  p_{\omega}(\xi_k(x),\xi_k(a))|^2\underset{k\rightarrow\infty}{=}o(1).
	\end{equation}
	Then combining \eqref{new3}, \eqref{new8}, and  \eqref{new9}, we obtain
	\begin{align}
	&\|W_k\bbA(b,p,0)\varphi_k-\bbA_{\omega}W_k\varphi_k\|_{\ell^2(\Gamma)}^2\\
	&\quad = \sum_{y\in \Gamma}|W_k(\bbA(b,p,0)\varphi_k)(y)-[\bbA_{\omega}(W_k\varphi_k)](y)|^2\\
	&\quad=  \sum_{x\in T(k)} \big|\sum_{x \sim a} [p(x,a)- p_{\omega}(\xi_k(x),\xi_k(a))]\varphi_k(a)\,\big|^2\\
	&\quad\leq c(k)C(b^+)\|\varphi_k\|^2_{\ell^2(\Gamma)}\underset{k\rightarrow\infty}{=}o(1),
	\end{align}
	where  $C(b^+)>0$ is some fixed constant.  Therefore, we get
	\begin{equation}
	\left|\|(\bbA(b,p,0)-E)\varphi_k\|_{\ell^2(\Gamma)}-\| (\bbA_{\omega}-E)W_k\varphi_k\|_{\ell^2(\Gamma)} \right|\underset{k\rightarrow\infty}{=}o(1).
	\end{equation}
	Thus $\{W_k\varphi_k\}_{k\geq 1}$ is a Weyl sequence for $\bbA_{\omega}$ and $E\in \sigma(\bbA_{\omega})$ as asserted.
	
\end{proof}
\begin{remark}
\begin{enumerate}
\item We emphasize that the equality in \eqref{ur523} requires special attention if $y\in\partial( T(k,i))$, since in this case the inclusion
	\begin{equation}
	\xi_k(\{a\in\cV: a\sim x\})\subset \{b\in\cV: b\sim y\},
	\end{equation}
	could be strict. However, by \eqref{ur523n} the equality \eqref{ur523} holds as asserted even in this special case. Due to this nuance the current proof is not applicable to $\bbJ=\mathbb S$. (Informally, if $q\not=0$ in \eqref{5.1} then we ``see" extra bits around $v_i$ which are not observed near $o$).

\item The almost-sure spectrum $\Sigma$ for $\bbA_{\omega}=\bbA (b_{\omega}, 1,0)$ can be computed explicitly if $p\equiv1$,  $q\equiv0$, i.e. the random branching model for the adjacency matrix. Indeed, in this case, the quadratic form $\mathfrak{a}$ of the $\bbA$ is given by
\begin{equation}
\mathfrak{a}[\varphi, \varphi]=-\sum_{u\sim v} \varphi(u)\overline{\varphi(v)}, \varphi\in\ell^2(\Gamma).
\end{equation}
therefore
\begin{equation}
\|\bbA (b, 1,0)\|_{\cB(\ell^2(\Gamma))}\leq \|\bbA (\wti b, 1,0)\|_{\cB(\ell^2(\Gamma))},
\end{equation}
where $\wti b:= \max \{P_1\supp \mu\}$ and $P_1$ is the first coordinate function. Combining this and \eqref{6.2} we get
\begin{align}
\Sigma=\overline{\bigcup_{b\ \text{ periodic}} \sigma(\bbA(b,1,0))}&\subset[- \|\bbA (\wti b, 1,0)\|_{\cB(\ell^2(\Gamma))},  \|\bbA (\wti b, 1,0)\|_{\cB(\ell^2(\Gamma))}]\\
&=[-2\sqrt{\wti b}, 2\sqrt{\wti b}]\subset\Sigma.
\end{align}
As before, we note that this proof is not applicable to the case $q\not\equiv 0$ or $p\not\equiv$ const.

\item Remark~\ref{rem5.3}, the proof of Theorem~\ref{thm5.5}, the previous remark, and the question of computing the almost--sure spectrum itself illustrate a subtle distinction between adjacency matrices and  Schr\"odinger operators. This issue arises even in the most simple case $\Gamma=\bbZ_+$, $p\equiv 1$, and random $q$, since (in view of \eqref{int19})
\[
\mathbb{S}
=
\begin{bmatrix}
q(1) & -1 \\
-1 & 2q(2) & -1 \\
 & -1 & 2q(3) & -1 \\
 && \ddots & \ddots & \ddots
\end{bmatrix}.
\]
To be more specific, if one considers
\[
\wti{\mathbb{S}}
=
\begin{bmatrix}
2q(1) & -1 \\
-1 & 2q(2) & -1 \\
 & -1 & 2q(3) & -1 \\
 && \ddots & \ddots & \ddots
\end{bmatrix}
\]
where $\{q(n)\}$ is a sequence of i.i.d.\ random variables, then it is well-known that the spectrum of $\wti{\mathbb{S}}$ is almost surely given by $[-2,2] + 2 \supp\{q\}$. Since $\mathbb{S}$ is a rank-one perturbation of $\wti{\mathbb{S}}$, their essential spectra coincide. However, depending on the support of $q$, it can happen that $\mathbb S$ may have discrete eigenvalues outside of $\sigma_{\mathrm{ess}}(\wti{\mathbb{S}})$, and these eigenvalues may not be constant almost-surely. Thus, one should not expect the analogue of Theorem~\ref{thm5.5} to hold for random Schr\"odinger operators on graphs (as opposed to adjacency matrices).
\end{enumerate}

\end{remark}
\subsection{Breuer-Type Decomposition}\lb{sec5.1}
Our next objective is to revise the Breuer decomposition \cite[Theorem~2.4]{Br07} which may be viewed as a discrete version of the orthogonal decomposition of {\it metric} trees. To point out a difference between the two, we note: The invariant subspaces in  \eqref{212new} are parametrized by single vertices, while those in Breuer's decomposition are parametrized by entire generations of vertices.
\begin{theorem}\lb{thm52}
Assume Hypothesis~\ref{hyp5.1}. Then there exists a unitary operator
\begin{align}\lb{a7}
\Phi_b:  \ell^2(\cV)\rightarrow \bigoplus_{n=0}^{\infty} \bigoplus_{k=1}^{m(n)} \ell^2(\bbZ_+),
\end{align}
such that
\begin{equation}\lb{a8}
\Phi_b\, \bbJ({b, p, q}) \Phi_b^{-1}=\bigoplus_{n=0}^{\infty}\bigoplus_{k=1}^{m(n)}  J(T^nb,T^np,T^nq),
\end{equation}
where $m(n):=b_0\cdot b_1\cdots\cdot b_{n-1}\cdot(b_{n}-1)$, $n\in\bbZ_+$, and  $J(b,p,q)$ denotes the Jacobi matrix acting in $\ell^2(\bbZ_+)$ and given by
\begin{align}\lb{5.5}
&J(b,p,q):=\begin{pmatrix}
(b_{0} p_{0}+ p_{-1})q_{0}&\sqrt{b_{0}} p_{0}&0&\ \\
\sqrt{b_{0}} p_{0}&(b_{1} p_{1}+ p_{0})q_{1}&\sqrt{b_{1}} p_{1}&\ddots \\
0&\sqrt{b_{1}} p_{1}&\ddots&\ddots \\
\ &\ddots&\ddots&\ddots\\
\end{pmatrix}.\
\end{align}
\end{theorem}
\begin{proof}
Breuer's inductive procedure \cite[Theorem~2.4]{Br07} yields an orthonormal basis
 \begin{equation}\no
 \{\varphi_{n,k,j}: n\in\bbZ_+, 1\leq k\leq m(n), j\in\bbZ_+\}\subset \ell^2(\cV).
 \end{equation}
For all admissible triples $n, k, j$, the basis elements satisfy
\begin{align}
&\supp(\varphi_{n,k,j})\subset \{ v\in\cV: \text{gen}(v)=n+j\},\lb{56new}\\
&\varphi_{n,k,j+1}(v)=\begin{cases}
\frac{\varphi_{n,k,j}(u)}{\sqrt{b_{n+j}}},& u\sim v,\ \gen(v)=\gen(u)+1,\\
0,&\text{otherwise},
\end{cases}\lb{57nn}
\end{align}
and
\begin{equation}\lb{5.7}
\bbJ(b, p, q)\varphi_{n,k,{ j}}=
\begin{cases}
\sqrt{b_{n+j-1}}\, p_{n+j-1}\varphi_{n,k,{ j-1}}&\\
\hspace{5mm}+(b_{n+j} p_{n+j}+ p_{n+j-1})q_{n+j}\varphi_{n,k, { j}}&\\
\hspace{35mm}+\sqrt{b_{n+j}}\, p_{n+j}\varphi_{n,k,{ j+1}},&  j\geq 1,\\
(b_{n} p_{n}+ p_{n-1})q_{n}\varphi_{n,k,0} +\sqrt{b_{n}}\, p_{n}\varphi_{n,k,1},& j=0.
\end{cases}
\end{equation}
The latter shows that the operator $\bbJ(b,p,q)$ leaves the  subspaces
\begin{equation}
\cH_{n,k}:=\overline{\text{span}\{\varphi_{n,k,j}: j\in\bbZ_+\}}\subset\ell^2(\cV)
\end{equation}
invariant. Thus we have
\begin{equation}\lb{a6}
\ell^2(\cV)=\bigoplus_{n=0}^{\infty} \bigoplus_{k=1}^{m(n)} \cH_{n,k},\  \bbJ(b, q, w)P_{\cH_{n,k}}
=J(T^n b, T^np, T^n q),
\end{equation}
where $P_{\cH_{n,k}}$ denotes an orthogonal projection onto $\cH_{n,k}$ in $\ell^2(\cV)$. Let us define  unitary operators	
\begin{align}
&\cU_{n,k}: \cH_{n,k}\rightarrow \ell^2(\bbZ_+),\  n\in\bbZ_+, 1\leq k\leq m(n),\ \\
& \cU_{n,k} \varphi_{n,k,j}:= \delta_j,\  j\in\bbZ_+.
\end{align}
and
\begin{equation}
\Phi_b:= \bigoplus_{n=0}^{\infty} \bigoplus_{k=1}^{m(n)} \cU_{n,k}.
\end{equation}
Then \eqref{5.7} together with \eqref{a6} yield \eqref{a8} and \eqref{5.5} as asserted.
\end{proof}

\subsection{Dynamical and Exponential Localization for Discrete Random Trees}\lb{dynexp}
	
In this section we discuss spectral and dynamical localization for three discrete models: the random branching model (RBM), the random weights (RWM) model, and random Schr\"odinger operators (RSO).

Let us denote the nonzero entries of $J(b,p,q)$ by
\begin{align}
\begin{split}\lb{albet}
& \beta_{j}=\beta_j(b,p,q)=(b_{j} p_{j}+ p_{j-1})q_{j},\\
&\alpha_{j}=\alpha_j(b,p)=\sqrt{b_{j}} p_{j},\, j\in\bbZ_+.
\end{split}
\end{align}
Then a sequence $u=\{u_j\}_{j=0}^{\infty}$ satisfies $J(b,p,q)u=Eu$, $E\in\bbR$, that is,
\begin{equation}
\begin{cases}
\alpha_{j-1}u_{j-1}+(\beta_{j}-E)u_j+\alpha_{j}u_{j+1}=0,\ j\in\bbN, \lb{a9}\\
(\beta_{0}-E)u_0+\alpha_{0}u_1=0,
\end{cases}
\end{equation}
if and only if
\begin{equation}\lb{a11}
\begin{bmatrix}
u_{j+1}\\\alpha_{j}u_j
\end{bmatrix}=
M^{E, j}(b,p,q)
\begin{bmatrix}
u_{j}\\\alpha_{j-1}u_{j-1}
\end{bmatrix}, \text{\ for all } j\in\bbN.
\end{equation}
where
\begin{align}
\begin{split}\lb{a10}
M^{E,j}(b,p,q):&=\frac{1}{\alpha_{j}}\begin{bmatrix}
E-\beta_{j}&-1\\
\alpha^2_{j}&0
\end{bmatrix}\\
&=\begin{bmatrix}
\frac{E-(b_{j} p_{j}+ p_{j-1})q_{j}}{\sqrt{b_{j}} p_{j}}&-\frac{1}{\sqrt{b_{j}} p_{j}}\\
\sqrt{b_{j}} p_{j}&0
\end{bmatrix}.
\end{split}
\end{align}
The transfer matrix \eqref{a10} gives rise to an SL$(2,\bbR)$-cocycle
\begin{align}\no
&(T,M^E): \Omega\times\bbR^2\rightarrow  \Omega\times\bbR^2,\ (T,M^{ E})(\omega, v)=(T\omega, M^E(\omega)v),
\end{align}
where $M^E:\Omega\rightarrow \text{SL}(2,\bbR)$ and
\begin{align}
&M^E(\omega)=\begin{bmatrix}
\frac{E-(b_{\omega}(0)p_{\omega}(0)+ p_{\omega}(-1))q_{\omega}(0)}{\sqrt{b_{\omega}(0)} p_{\omega}(0)}&-\frac{1}{\sqrt{b_{\omega}(0)} p_{\omega}(0)}\\
\sqrt{b_{\omega}(0)} p_{\omega}(0)&0
\end{bmatrix}.
\end{align}
The $n$-step transfer matrix $M^{E}_n(\omega)$ and the Lyapunov exponent are defined as in \eqref{tm} and \eqref{LE1} respectively.

\begin{theorem}\lb{thm:discreteFurst}
		Assume Hypothesis~\ref{hyp61}. Then there is a  set $\mathcal{D} \subseteq \bbR$ of cardinality at most one such that $G = G_{\nu(E)}$ enjoys the following properties for $E \in \bbR \setminus \mathcal{D}$.
		\begin{enumerate}
		\item[{\rm(i)}] $G$ is noncompact
		\item[{\rm(ii)}] $G$ is strongly irreducible
		\item[{\rm(iii)}] $G$ is contracting {\rm(}cf.\ \cite[Definition~2.8]{BuDaFi}{\rm)}
		\item[{\rm(iv)}] $\mathrm{Fix}(G) = \emptyset$
		\end{enumerate}
In particular, $L$ is continuous and positive on $\bbR \setminus \mathcal{D}$.
	\end{theorem}
	
	\begin{proof}
	Following the proof of Theorem~\ref{thm34}, we choose
\begin{equation}
(b_1,p_1,q_1) \neq (b_2,p_2,q_2) \in \supp\wti\mu,
\end{equation}
 let $M_j(E)$ denote the transfer matrix corresponding to $(b_j,p_j,q_j)$, and form the matrices $A = M_1M_2^{-1}$ and $g = [M_1,M_2]$. Let us comment briefly on the method of proof. We can immediately apply \cite{BuDaFi2} to deduce that there is an unspecified discrete set of energies away from which (i)--(iv) hold.  In fact, the argument of \cite{BuDaFi2} applies away from energies at which $\tr M_j(E) = 0$ or $\det g(E) = 0$, which allows us to refine this to a discrete set with no more than 3 elements. However, we can do better still: Conditions (i)--(iv) hold for any $E$ for which the following criterion is met:
 \begin{equation} \label{eq:furstRefined}
\not\exists  \mathcal{F} \subseteq \bbR\bbP^1 \text{ with } \#\mathcal F \in \{1,2\} \text{ such that } M_j \mathcal{F} = \mathcal{F} \text{ for } j=1,2.
  \end{equation}
In particular, \eqref{eq:furstRefined} implies (iii) which in turn implies (i) by standard arguments about $\SL(2,\bbR)$. Once (i) holds, then \eqref{eq:furstRefined} immediately yields (iv) and also implies (ii) (cf.\ \cite{bougerollacroix}).
\medskip

\noindent\textbf{\boldmath Case 1: \eqref{515nn} holds.} We have $p_1=p_2=1$, so
\[
M_j
=
\frac{1}{\sqrt{b_j}}\begin{bmatrix} E - (b_j+1)q_j & -1 \\ b_j & 0 \end{bmatrix}.
\]
We calculate \footnotesize
\[
g
=
\frac{1}{\sqrt{b_1b_2} }\begin{bmatrix} b_1 - b_2 & (b_1+1)q_1 - (b_2+1)q_2 \\
(b_1-b_2)E + b_2(b_1+1)q_1 - b_1(b_2+1)q_2 & b_2-b_1  \end{bmatrix}.
\]\normalsize

\noindent \textbf{\boldmath Case 1a: $b_1 = b_2$.} It follows that $q_1 \neq q_2$ and hence $(b_1+1)q_1 \neq (b_2+1) q_2$. One can confirm that $\det g(E) \neq 0$ for all $E$, so that $M_1$ and $M_2$ have no eigenvectors in common. Thus, there is no $\mathcal{F}$ of cardinality one with $M_j \mathcal F = \mathcal F$ for $j=1,2$. Now, suppose that an invariant $\mathcal F \subseteq \bbR\bbP^1$ of cardinality two exists. We must then have have $\mathcal{F} = \{\bar u_1, \bar u_2\}$ and $M_j \bar u_1= \bar u_2$, $M_j \bar u_2 = \bar u_1$ for some $j$; without loss, assume $j=1$. This forces $\tr M_1 = 0$. However, since $(b_1+1)q_1 \neq (b_2+1)q_2$, we must have $\tr M_{2} \neq 0$, so $M_2 \mathcal F = \mathcal F$ forces $M_2 \bar u_k = \bar u_k$ for $k=1,2$, that is to say, each $\bar u_k$ is an eigendirection of $M_2$. Identifying $\bbC\bbP^1$ with the Riemann sphere in the usual way, write $z_k$ for the image of $\bar{u}_k$ under the identification $\bbC\bbP^1 \cong \bbC \cup \{\infty\}$. Since $M_2 z_k = z_k$, we have
\[
\frac{E - (b_2+1)q_2}{b_2} - \frac{1}{b_2 z_k} = z_k,
\quad k=1,2.
\]
From this, we deduce $z_1 z_2 = 1/b_2$. On the other hand, since $\tr M_1 = 0$, we observe
\[
M_1 z_1 = -\frac{1}{b_1z_1} \neq z_2,
\quad
M_1 z_2 = - \frac{1}{b_1 z_2} \neq z_1,
\]
a contradiction. Thus, when $b_1 = b_2$, \eqref{eq:furstRefined} holds and we have (i)--(iv) for every $E \in \bbR$.
\medskip

\noindent \textbf{\boldmath Case 1b: $b_1 \neq b_2$.} There are two further subcases to consider.

\noindent \textbf{\boldmath Case 1bi: $(b_1+1)q_1 = (b_2+1)q_2$.} Then, $\det g(E) \neq 0$ for every $E$. Thus, again $M_1$ and $M_2$ never share an eigenvector. At energy $E = E_0 := (b_1+1)q_1 = (b_2+1)q_2$, both $M_1$ and $M_2$ preserve $\mathcal{F} = \{\mathrm{span}(\vec{e}_1), \mathrm{span}(\vec{e}_2)\}$. Since $E_0$ is the only energy at which $\tr M_j$ vanishes for either $j$, we have (i)--(iv) for $E \in \bbR \setminus \{E_0\}$.
\medskip

\noindent \textbf{\boldmath Case 1bii: $(b_1+1)q_1 \neq (b_2+1)q_2$.} One can check that $\det g(E)$ vanishes for exactly one value of $E_1 \in \bbR$. Using the same argument as in Case~1a, we see that there is no invariant $\mathcal{F}$ of cardinality one or two away from $E = E_1$. Thus, (i)--(iv) hold away from $\mathcal{D} = \{E_1\}$.
\medskip

\noindent\textbf{\boldmath Case 2: \eqref{516nn} holds.} Then,
\[
M_j
=
\frac{1}{p_j\sqrt{b_j}}
\begin{bmatrix}
E & -1 \\
p_j^2 b_j & 0
\end{bmatrix},
\quad
\text{ and }
p_1 \sqrt{b_1} \neq p_2 \sqrt{b_2}.
\]
Notice that
\[
A:= M_1 M_2^{-1}
=
\frac{1}{p_1p_2 \sqrt{b_1 b_2}}
\begin{bmatrix}
p_2^2 b_2 & 0 \\ 0 & p_1^2 b_1
\end{bmatrix}.
\]
Since $p_1 \sqrt{b_1} \neq p_2 \sqrt{b_2}$, $A$ is hyperbolic\footnote{I.e., $|\tr(A)|>2$.} and any finite set of directions left invariant by $M_1$, $M_2$, and $A$ must be a subset of $\{\mathrm{span}(\vec{e}_1), \mathrm{span}(\vec{e}_2)\}$. It is easy to see that this cannot happen for $E \neq 0$, so we may take $\mathcal{D} = \{0\}$ in this case.
	\end{proof}
	
\begin{remark}
Let us note that the need to remove a single point is sharp. For example, in Case~1bi above, one can verify that $L(E_0)=0$. To see this, write $r = -(b_1/b_2)^{1/2}$ and $R = \mathrm{diag}(r,r^{-1})$, and observe that
\[
M_j(E_0) M_k(E_0)
=
\begin{cases}
- I & j=k \\
R^{-1} & (j,k) = (1,2) \\
R & (j,k)= (2,1).
\end{cases}
\]
Thus, by passing to blocks of length two and using the strong law of large numbers, we deduce $L(E_0) = 0$.
\end{remark}

\begin{proof}[Proof of Theorem~\ref{thm5.6}]
Now that we know that $L$ is positive and obeys a uniform LDT away from $\mathcal D$, spectral and dynamical localization for $J_{\omega}$ follows as in Theorem \ref{main1}, see also \cite{DS} where spectral localization was proved for the discrete RBM.  Let $\Omega^*$ be defined as in \eqref{340new} (where $\hatt\Omega$ is as in Theorem~\ref{thm5.5}, and $\wti\Omega$ is a full measure set realizing localization for $J_{\omega}$) and fix $\omega\in\Omega^*$.

 For all $n\in\bbZ_+$, the spectral subspace $\ran(\chi_I(J_{T^n\omega}))$ enjoys an orthonormal basis $\{f_{n,j}\}_{j=0}^{\infty}$ of eigenfunctions of $J_{T^n\omega}$ corresponding to energies $E\in I$. If we define  $\psi_{n,k,j} :=\cU_{n,k}^{-1}f_{n,j}$, then
\begin{equation}
\{\psi_{n,k,j}: n\in\bbZ_+, 1\leq k\leq m(n), j\in\bbZ_+\}
\end{equation}
is an orthonormal basis of $\ran(\chi_I(\bbJ_{\omega}))$.

{\it Proof of \eqref{1.8nn}}. For an arbitrary admissible triple $n,k,j$ we will prove \eqref{1.8nn} with $f=\psi_{n,k,j}$. First, we note that  by spectral localization for $J_{\omega}$ one has
\begin{align}
|f_{n,j}(p)|\leq C(f_{n,j})e^{-\lambda p},\ p\in\bbZ_+;\ \lambda:=\min_{E\in I}\frac{L(E)}{2}>0,
\end{align}
for some $ C(f_{n,j})>0$. Then for $|x|>n$ we get
\begin{align} \label{eq:psiTof}
|\psi_{n,k,j}(x)|
= |\cU_{n,k}^{-1}f_{n,j}(x)|
& = |f_{n,j}(|x|-n)\varphi_{n,k,|x|-n}(x)|\\
&\underset{ \eqref{57nn}}{\leq}
\frac{C(\psi_{n,k,j})e^{-\lambda(|x|-n)}}{\sqrt{w_{o}(|x|)}},
\end{align}
which implies \eqref{1.8nn}.

{\it Proof of \eqref{536}}. Due to dynamical localization for $J_{\omega}$ one has
\begin{equation}\lb{539nn}
\sum_{j\in\bbZ_+}|\langle f_{n,j}(p), f_{n,j}(q)\rangle_{\ell^2(\bbZ_+)}|\leq C_n e^{q}e^{-\theta(p-q)},
\end{equation}
for all $p\geq q$, $\theta < \min_{E \in I} L(E)$, and a constant $C_n=C(n, \omega, \theta)>0$ (cf., e.g., \cite[Proof of Theorem~6.4]{BuDaFi} where this step is discussed for the standard Anderson Hamiltonian). Next, we have
\begin{align}
&\sup\limits_{t>0}|\langle\delta_x, \chi_I(\bbJ_{\omega})e^{-it\bbJ_{\omega}} \delta_y\rangle_{\ell^2(\cV)}|
\leq \sum_{\substack{n\in\bbZ_+\\1\leq k\leq m(n)}}\ \  \sum_{j=0}^{\infty} |\psi_{n,k,j}(x)\psi_{n,k,j}(y)|\no\\
&\underset{\eqref{56new}}{\leq}
\sum_{\substack{0\leq n\leq |y| \\1\leq k\leq m(n)}}\ \   \sum_{j=0}^{\infty} |\psi_{n,k,j}(x)\psi_{n,k,j}(y)|\no\\
&\underset{\eqref{eq:psiTof}}{=}
\sum_{\substack{0\leq n\leq |y| \\1\leq k\leq m(n)}}\ \   \sum_{j=0}^{\infty}  |f_{n,j}\big(|x|-n\big)\varphi_{n,k,|x|-n}(x)\, f_{n,j}\big(|y|-n\big)\varphi_{n,k,|y|-n}(y)|\no\\
&\underset{\eqref{57nn}}{\leq} \sum_{\substack{0\leq n\leq |y| \\1\leq k\leq m(n)}}\ \  \sum_{j=0}^{\infty}  \frac{|f_{n,j}\big(|x| - n\big)f_{n,j}\big(|y| - n\big)|}{\sqrt{w_y(|x|-|y|-1)}}\no\\
&\underset{\eqref{539nn}}{\leq} \sum_{\substack{0\leq n\leq \gen(y)\\1\leq k\leq m(n)}}  \frac{C_n e^{ |y|}e^{-\theta(|x|-|y|)}}{\sqrt{w_y(|x|-|y|-1)}}
\leq \frac{C_y e^{-\theta(\dist(x,y))}}{\sqrt{w_y(|x|-|y|)}}.
\end{align}
Finally, \eqref{537} follows from \eqref{536} by summation in $x$.
\end{proof}

\section*{Acknowledgments} We thank G.\ Berkolaiko, M.\ Lukic, and G.\ Stolz for helpful discussions, and P. Hislop for bringing our attention to this subject and for motivating discussions.

\end{document}